\documentclass[hidelinks,onefignum,onetabnum]{siamart250211}

\usepackage{lipsum}
\usepackage{amsfonts}
\usepackage{graphicx}
\usepackage{epstopdf}
\usepackage{algorithmic}
\usepackage{amssymb}
\usepackage{graphicx}
\usepackage{comment}
\ifpdf
  \DeclareGraphicsExtensions{.eps,.pdf,.png,.jpg}
\else
  \DeclareGraphicsExtensions{.eps}
\fi

\newcommand{\midplus}{|_{_+}}
\newcommand{\midminus}{|_{_-}}

\newcommand{\ph}{\varphi}

\newcommand{\bp}{{\bf p}}

\newcommand{\ak}{\alpha_K}

\newfont{\iams}{msbm9}

\newsiamthm{example}{Example}
\newsiamthm{xca}{Exercise}
\newsiamremark{remark}{Remark}
\newsiamremark{hypothesis}{Hypothesis}
\crefname{hypothesis}{Hypothesis}{Hypotheses}

\headers{Entropy-stable DG for Boltzmann - Poisson}{J. Morales, and I.M. Gamba}

\title{Entropy-stable DG schemes for Boltzmann - Poisson models of collisional electronic transport along energy bands\thanks{
\funding{This  research was partially supported by NSF grants  NSF-RNMS DMS-1107465 (KI-Net),  DMS 1715515,  DMS 2009736, and DOE DE-SC0016283
project Simulation Center for Runaway Electron Avoidance and Mitigation. 
Support from the Oden Institute of Computational Engineering and Sciences at the University of Texas at Austin is gratefully acknowledged.
The authors gratefully acknowledge as well the support of the Hausdorff Research Institute for Mathematics (Bonn), through the Junior Trimester Program on Kinetic Theory. The first author gratefully acknowledges too start-up support from the University of Texas at San Antonio Departments of Mathematics and Physics \& Astronomy. 
}}}

\author{Jose Morales Escalante\thanks{University of Texas at San Antonio, San Antonio, TX 
  (\email{jose.morales4@utsa.edu}, \url{https://sciences.utsa.edu/faculty/profiles/morales-jose.html/}).}
\and Irene M. Gamba\thanks{Department of Mathematics, University of Texas at Austin, Austin, TX 
  (\email{gamba@math.utexas.edu}).}}

\usepackage{amsopn}

\ifpdf
\hypersetup{
  pdftitle={Entropy-stable DG for Boltzmann - Poisson},
  pdfauthor={J. Morales, and I.M. Gamba}
}
\fi

\begin{document}

\maketitle

\begin{abstract}
This work is related to the
development of entropy-stable 
positivity-preserving 
Discontinuous Galerkin (DG) methods
as a computational scheme for  Boltzmann-Poisson systems  modeling the probability density of collisional electronic transport along energy bands in  semiconductors. We pose, in momentum coordinates representing spherical / energy-angular variables, the respective Vlasov-Boltzmann equation with a linear collision operator and a singular measure, modeling scatterings as functions of the bandstructure appropriately for  hot electron nanoscale transport.
We show stability results of semi-discrete DG schemes under an entropy norm for 1D-position (2D-momentum) and 2D-position (3D-momentum), using dissipative properties of the collisional operator given its entropy inequality. The latter depends on an exponential of the Hamiltonian rather than the Maxwellian associated with only kinetic energy. For the 1D problem, knowing the analytic solution to the Poisson equation and convergence to a constant current is crucial to obtaining full stability (weighted entropy norm decreasing 
over time). For the 2D problem, specular reflection boundary conditions and periodicity are considered in estimating stability under an entropy norm. 
Regarding the positivity-preservation proofs in the DG scheme for the 1D problem, 
we treat collisions as a source and find convex combinations of the transport and collision terms which guarantee positivity of the cell average of our numerical probability density at the next time. The positivity of the numerical solution to the probability density in the domain is guaranteed by applying well known limiters
that preserve the cell average modifying the slope of the piecewise linear solutions to make the function non-negative. 
The use of a spherical  coordinate system whose radial component is the momentum magnitude 
is slightly different from choices in previous DG solvers for Boltzmann-Poisson, since the proposed DG formulation gives simpler integrals 
involving piecewise polynomials for both transport and collision.
\end{abstract}



\section{Introduction}

The Boltzmann - Poisson system is a model for electron transport in semiconductors
representing the balance of transport and collisions 
in the $(\vec{x},\vec{\bp})\in\mathbb{R}^3_{\vec{x}}\times\mathbb{R}^3_{\vec{\bp}}$ position-momentum phase space for electrons in the conduction band,
\begin{align}\label{vlasov}
\partial_t f +  \partial_{(\vec{x},\vec{\bp})}  \cdot ( f \alpha)
&=Q(f), \\[4pt]
\text{with transport vector} \qquad \qquad \qquad\qquad  
&\ \nonumber \\[4pt]
\alpha &= ( \partial_{\vec{\bp}} H
, - \partial_{\vec{x}} H) = ( \partial_{\vec{\bp}} \varepsilon
, q \partial_{\vec{x}}  \Phi) \in \mathbb{R}_{\vec{x}}^3\times\mathbb{R}^3_{\vec{\bp}},  \\[4pt]
\text{with the Hamiltonian} \qquad \qquad \qquad\qquad  &\ \nonumber \\[4pt]
H(\vec{x},\vec{\bp},t) &=  \varepsilon(\vec{\bp})
-  q \Phi(\vec{x},t),  \\[4pt]
\text{and collision} \qquad \qquad\qquad \qquad \qquad\qquad &\ \nonumber \\[4pt]
Q(f) = &\int_{\Omega_{\vec{\bp}}} \!\!S(\vec{\bp}\,' \rightarrow \vec{\bp} ) f' d\vec{\bp}\,'
\!- \! f \int_{\Omega_{\vec{\bp}}}\!\! S(\vec{\bp} \rightarrow \vec{\bp}\,' ) d\vec{\bp}\,',  
\end{align}
coupled to the Poisson equation for total charges
\begin{equation}\label{poisson}
-\partial_{\vec{x}} \cdot (\epsilon \partial_{\vec{x}} \Phi)(\vec{x},t) =
q \left[ N(\vec{x}) - \int_{\Omega_{\vec{\bp}}} f(\vec{x},\vec{\bp},t) d\vec{\bp} \right], \quad \vec{E}(\vec{x},t) = -\partial_{\vec{x}}\Phi(\vec{x},t),
\end{equation}

for $\vec{x}\in\Omega_{\vec{x}}\subset\mathbb{R}^3_{\vec{x}}$. 
The momentum variable is $\vec{\bp}=\hbar\vec{k}$,
$\vec{k}$ is the crystal momentum wave vector,
$\varepsilon(\vec{\bp})$ is the conduction energy band structure for electrons in the semiconductor,
$f(\vec{x},\vec{\bp},t)$ is the probability density function (pdf) in the phase space for electrons in the conduction band at a given time,
\begin{equation}\label{velocity-relation}
\dot{\vec{x}} = \vec{v}(\vec{\bp}) = \partial_{\vec{\bp}} \varepsilon (\vec{\bp}) 
\end{equation}
is the quantum mechanical electron group velocity, $q$ is the positive electric charge of a proton,   
 $\Phi(\vec{x},t)$ is the electric potential (we assume that the only force over the electrons is the self-consistent electric field, and that it is given by the negative gradient of the electric potential), $\epsilon$ is the permittivity of the material,  $N(\vec{x})$ is the  doping background (assumed fixed) in the semiconductor material, and $S(\vec{\bp}\,' \rightarrow \vec{\bp} )$ is the scattering kernel that defines the gain and loss operators whose difference gives the collision integral operator $Q(f)$.
 We will assume in this work a linear collision operator, which is valid in the regime of low electron density, as the enforcement of the Pauli exclusion principle in this case via the collision operator structure is not needed.
 {We assume $\Omega_{\vec{x}}$ is $\Pi_{P}$ (a torus corresponding to periodic boundary conditions), a bounded domain where the Poisson problem is well posed, with periodic boundary conditions. That means the solution of our Boltzmann-Poisson problem is well posed and $f\in L_2^1$, $\Phi \in L^2$.} 
 We stress out that the transport terms have an underlying Hamiltonian structure, where this Hamiltonian transport vector is related
 to the group velocity $\vec{v}$ of the electron wave and the electric force  $\vec{F}$ over it, 
\begin{equation*}
\alpha({\vec{x}},{\vec{\bp}}) = (\dot{\vec{x}},\dot{\vec{\bp}}) =
 (\vec{v}(\vec{\bp}), \vec{F}(\vec{x},t) ) = (\partial_{\vec{\bp}}\varepsilon(\vec{\bp}), q\partial_{\vec{x}} \Phi(\vec{x},t) ) = 
(\partial_{\vec{\bp}}H, -\partial_{\vec{x}} H ) 
\in \Omega_{\vec{x}}\times\mathbb{R}^3_{\vec{\bp}},
 \end{equation*}
 and then the transport part of the Boltzmann eq. for collisional electron transport in semiconductors has a Hamiltonian structure given by a Poisson bracket, as in 
\vspace{-0.05cm}
 \begin{align}
  \partial_t f +  \{f,H\}
&= Q(f), \nonumber \\
\{f,H\} &= 
\partial_{\vec{x}}f \cdot \partial_{\vec{\bp}}H -
\partial_{\vec{\bp}}f \cdot \partial_{\vec{x}}H.
 \end{align}

For many quantum collision mechanisms, such as in semiconductors,
the scattering kernel
$S(\vec{\bp}\,' \rightarrow \vec{\bp} )$ 
depends on the difference between energies
$\varepsilon(\vec{\bp}) - \varepsilon(\vec{\bp}\,')$, as
in collision operators of the form  
$\delta( \varepsilon(\vec{\bp}) - \varepsilon(\vec{\bp}\,') + l\hbar w_p)$
for electron - phonon collisions. This form is related to the energy conservation given by Planck's law, in which the jump in energy from one state to another is balanced with the energy of a phonon. The mathematical consequence of this is that 
we can obtain much simpler expressions for the integration of the collision operator if we express the momentum in curvilinear coordinates that involve the energy $\varepsilon(\vec{p})$ as one of the variables \cite{MP}, \cite{carr03},  \cite{cgms06}, \cite{CGMS-CMAME2008, MeG-JCP2018}. The other two momentum coordinates could be either an orthogonal system in the level set of energies, orthogonal to the energy gradient itself, or angular coordinates which are known to be orthogonal to the energy in the limit of low energies close to a local conduction band minimum, such as $(\mu, \varphi)$, the cosine of the polar angle and the azimuthal angle, respectively.

This gives both physical and mathematical motivations to pose the
Boltzmann Equation for semiconductors in curvilinear coordinates for the momentum 
$$\vec{p}(p_1,p_2,p_3) = \vec{\bp}(p_x,p_y,p_z),$$ 
where 
$$\vec{p}= (p_1,p_2,p_3),$$ 
and
$$\vec{\bp}=(p_x,p_y,p_z),$$ 
to later on choose
a particular case of curvilinear coordinates such as $(\varepsilon,\mu,\varphi)$. 
We will assume in the rest of this paper that our
system of curvilinear coordinates for the momentum 
 is orthogonal. This happens in particular for the case $(\varepsilon,\mu,\varphi)$ in
which $\varepsilon(|\vec{p}|)$ is monotone increasing,
so this set of coordinates is equivalent to the spherical coordinate representation $(|\vec{p}|,\mu,\varphi)$.  

The Boltzmann Equation for semiconductors (and more general forms of linear collisional plasma models as well) can be written in orthogonal curvilinear coordinates for the momentum vector $\vec{p}(p_1,p_2,p_3) =\vec{\bp}(p_x,p_y,p_z)$, 
under the velocity relation \eqref{velocity-relation}, 
with 
$h_j = \left| \frac{\partial \vec{\bp}}{ \partial p_j} \right|,
\,  j=1,2,3$, generate  the Jacobian determinant of the transformation from the curvilinear coordinates for the momentum to the respective cartesian ones by setting $J := J_{\frac{\partial {\vec{\bp}}}{\partial{\vec{p}}}}  = h_1 h_2 h_3$
and $\hat{e}_j$ the unitary vectors associated to each 
curvilinear coordinate $p_j$ at the point $(p_1,p_2,p_3)$,
\begin{align}\label{vlasov-op}
&\ \ \  \partial_t fJ +  \partial_{\vec{x}}   \cdot \left( J f  
\vec{v} \right)
+ q  \nabla_{\vec{x}}\Phi \cdot \nabla_{\vec{v}} (J f)= C(f), \nonumber
\\[6pt]
\text{with}&\ \\[6pt]
 \nabla_{\vec{x}}\Phi \cdot \nabla_{\vec{v}} (J f) \! &:= \!  \left[
\partial_{p_1}\left(\frac{ Jf \partial_{\vec{x}} \Phi\cdot\hat{e}_{p_1}}{h_1} \right) \! +\! 
\partial_{p_2}\left( \frac{Jf \partial_{\vec{x}} \Phi\cdot\hat{e}_{p_2}}{h_2} \right) \!  +\! 
\partial_{p_3}\left( \frac{ Jf \partial_{\vec{x}} \Phi\cdot\hat{e}_{p_3}}{h_3} \right)  \! 
\right],\nonumber
\end{align}
and a head on collisional form,
\begin{equation}\label{collision op}
C(f) \!:=  \!J Q(f)
 \!= \! J   \!\!\int_{\Omega_{\vec{p}}}  \! \!S(\vec{p}\,' \rightarrow \vec{p} )     
J' f' dp_1' dp_2' dp_3'
 \!- \! Jf  \!  \!\int_{\Omega_{\vec{p}}} \! \! S(\vec{p} \rightarrow \vec{p}' ) 
J'  dp_1' dp_2' dp_3' .
\end{equation}

We notice that we have expressed the Boltzmann Equation in divergence form with respect to the  momentum orthogonal curvilinear  coordinates.
We can write it even more compactly as
\begin{equation}
\partial_t (Jf)  +   \partial_{\vec{x}}   \cdot \left( J f  
\vec{v}(\vec{p}\,) \right)
+  \sum_{j=1}^3 
\partial_{p_j}\left( Jf
\frac{  q\partial_{\vec{x}} \Phi(\vec{x},t) \cdot\hat{e}_{p_j}}{h_j} \right)  
= C(f) .
\end{equation}

For $J \geq 0$, we can interpret $Jf(\vec{x},p_1,p_2,p_3,t)$ as a probability density function in the phase space $(\vec{x},p_1,p_2,p_3)$.
This Boltzmann Eq. is a general form for orthogonal curvilinear momentum  coordinates, from which previous energy-angular and spherical coordinate systems (such as in \cite{CGMS-CMAME2008} and \cite{MGCMSC-CMAME2017}) can be derived. For the spherical one in \cite{MGCMSC-CMAME2017}, the orthogonal curvilinear system is
\begin{equation}
(r \propto k^2,\mu=\cos\theta || p_x, \varphi=\arctan(p_z/p_y)), 
\end{equation}
 where $\theta$ is the polar angle with respect to the $p_x$-axis, $\varphi$ is the azimuthal angle, and $r \propto k^2$, which is proportional to the energy in the limit as $|k| \rightarrow 0$, approaching the conduction band local minimum. The energy-angular one in \cite{CGMS-CMAME2008} is 
 \begin{equation}
   (w \propto \varepsilon,\mu=\cos\theta || p_x, \varphi=\arctan(p_z/p_y)), 
 \end{equation}
 with
$w \propto \varepsilon$, under the assumption that 
$\varepsilon(|p|)$ is a Kane band model. 
The numerical method we will study in this work in connection to our Boltzmann - Poisson system is the Discontinuous Galerkin (DG) Finite Element Method (FEM), to be explained in Section 3. It was proposed by Reed and Hill \cite{Reed-Hill} for hyperbolic equations in the context of neutron transport. It is defined such that its numerics captures the mathematics of the hyperbolic transport by defining the so-called fluxes in such a way that the information propagates numerically in the same fashion as a hyperbolic equation propagates information analytically. In the particular context of electron collisional transport in semiconductors, the DG method has been used in works such as \cite{CGMS-CMAME2008, MGCMSC-CMAME2017, MeG-JCP2018} after an evolution of the numerical methods used to solve it that transitioned from Upwind Finite Differences \cite{MP} to WENO schemes \cite{carr03}, \cite{cgms06} and finally to mainly two schools, one related to the aforementioned development on DG methods for Boltzmann - Poisson, and the other related to Spherical Harmonics Expansions,
for which a good overview can be found in
\cite{Jungemann}.
{
There is a particular kind of DG methods, called Positivity Preserving (PPDG), which is designed to preserve the positivity of the functions that are the unknown to be solved for, usually having the interpretation of densities needed to be non-negative, 
such as fluid densities or probability density functions.
These positivity preserving DG methods were developed by Zhang and Shu \cite{ZhangShu1}, \cite{ZhangShu2}. The main idea is that, given a positive initial condition, PPDG is such that,
after a time iteration, the cell averages are preserved to be positive and, if the function is negative for a given region,
a limiter is applied in the interval(s) of interest, modifying the slope in such a way that after the modification the function is non-negative in the interval(s). It was designed in the context of the compressible Euler equations for fluids. 
Later works have incorporated this idea in different application contexts. Work related to our current problem was developed first in \cite{CGP} for linear Vlasov - Boltzmann collisional transport equations 
in cartesian coordinates under quadratically confined electrostatic potentials, where the Boltzmann collision operator was linear (with bounded scattering functions), and later in \cite{EECHXM-JCP} 
for a conservative phase space collisionless advection of neutral particles in curvilinear coordinates. }\\

On the other hand, regarding works on entropy dissipation laws and 
stability under entropy norms for Boltzmann semiconductor models, analytical results have been obtained, for example, by \cite{DBLP:journals/vlsi/Levermore98}, where moment closure hierarchies were studied for the Boltzmann - Poisson equation, getting an $f \log f$ type entropy dissipation law. 
We would also like to mention the work on \cite{BenAbdallah-Dolbeaut} that
studies the relative entropies for kinetic equations in bounded domains,
obtaining results on irreversibility, stationary solutions, and uniqueness.
Regarding specifically numerical works on this kind of problems, 
\cite{23ad23baaef84440b4f4005476604390}
studies the semiconductor Boltzmann equation based on spherical harmonics expansions and entropy discretizations. 
The convergence of numerical moment methods 
for linear kinetic equations was studied in \cite{Schmeiser:Convergence}. 
A mixed spectral-difference method for the steady state Boltzmann - Poisson system is presented in \cite{Ringhofer:Mixed}. 
DG discretizations of first-order systems of conservation laws derivable as moments of the Boltzmann equation with Levermore closure were considered in \cite{BARTH20063311} using  energy analysis techniques; the problem does not include acceleration terms as no force field is considered.
We should also mention the work in \cite{HLiu}, where a
high order DG scheme for solving nonlinear Fokker-Planck equations with a gradient flow structure was proposed, 
and it was shown to satisfy a discrete entropy dissipation law and to preserve steady-states, enforcing the positivity of the numerical solutions in the algorithm. \\

Our current work extends the contributions above mentioned by studying a Boltzmann - Poisson system for collisional electronic transport
in which the momentum is studied in curvilinear coordinates determined  by the energy band structure appearing both in the transport and the electron-phonon collision operator. The latter models the discrete energy jumps with a Dirac delta distribution scattering term, which is  clearly not a smooth bounded function and therefore more complex than previous collision operators studied in the context of entropy stable 
and positivity preserving 
DG methods for collisional electron transport.  
We study semi-discrete stability properties (under an entropy norm) of the  DG method for the Boltzmann - Poisson problem by means of the (semi-discrete)
Hamiltonian energy functional, where, without loss of generality, we assume a suitable discretization for the Poisson equation, yielding an accurate mass preserving approximation for both the electrostatic potential and corresponding electric field, and later we present the proof of positivity preservation for the DG scheme for BP. 
For the positivity preserving work, we introduce the 1D-position with 3D-momenta space problem with azimuthal symmetry, which can be reduced to a  3-dimensional set of coordinates (1D in position, 2D in momentum) plus time when modeling a 1D diode under such symmetry assumptions in momentum space.

The type of approximations and numerical schemes implementations  for the Poisson part of the Boltzmann - Poisson system (with the Boltzmann part in the usual $L_2$ Jacobian weighted norm, without enforcing the preservation of positivity) have been studied in  \cite{MGCMSC-CMAME2017,MeG-JCP2018}. Our contribution is to show the stability of a DG scheme under an entropy norm for the Boltzmann - Poisson system with curvilinear momentum (the Boltzmann part using an $L_2$ norm weighted not only by the Jacobian but in addition by the exponential of the Hamiltonian) that incorporates the full time dependent Hamiltonian, in the corresponding  curvilinear coordinates, obtaining 
a stability result that gives certain control over $\partial_t f$ for the two dimensional position domain (and 3D in momentum) problem
with specular reflection and periodic boundary conditions, 
and for the one dimensional space domain problem (with two dimensional momentum under symmetry assumptions) a full stability result, meaning specifically the decay of the entropy norm over time, using the knowledge of the convergence to a constant current in the limit of this time evolution problem
and the solution to the Poisson problem in 1D.

The mathematical difficulties overcome by our stability and convergence analysis involve different contributions.
One of them is motivated by the fact that indeed
the conduction band structure $\varepsilon(k)$ represent a particular challenge,
because on one hand it has a functional dependance on all momentum coordinates as in $\varepsilon(k_x,k_y,k_z)$, which
is not reflected in simplified analytical models such as the parabolic and Kane band models. Our first attempt to bridge this gap is to work with a general class of band structure models (with explicit analytic function to be defined) of the form 
$\varepsilon(r)$, with $r \propto k^2$, that not only include the parabolic and Kane bands but extend beyond their specific analytic form (these two being only two particular cases of the class of functions included in $\varepsilon(r)$). 
The novelty of this paper is also partly the joint use of both
the $L_2$ entropy norm for stability together with the curvilinear momentum coordinates (not only cartesian coordinates for the momentum).



\subsection{Preliminaries: On the problem of energy - curvilinear coordinates}

We would like to point out that the transport vector in curvilinear momentum coordinates
(that usually represent either a spherical or energy-angular set of coordinates, and
could equally represent an energy level set type of orthogonal system of coordinates)
\begin{equation}
 \beta = J \left(\vec{v}(\vec{p}), 
  \frac{F\cdot \hat{e}_{p_1}}{h_1},
  \frac{F\cdot \hat{e}_{p_2}}{h_2},
  \frac{F\cdot \hat{e}_{p_3}}{h_3}
\right),
 \end{equation}
 with $F(\vec{x},t) = -q \vec{E}(\vec{x},t) = q \partial_{\vec{x}} \Phi(\vec{x},t)$, 
 has two important properties in the curvilinear momentum space directly related to their 
Hamiltonian nature, namely 
$ \partial\cdot\beta = 0$ and
$\beta \cdot \partial H = 0$, 
defining $\partial = \partial_{(\vec{x},p_1,p_2,p_3)}$ as the gradient in the 
curvilinear 
space. 

Moreover, in the following sections we will present some physical examples to the reader in which it will be clear that the results on stability under an entropy norm for the DG scheme under study rely on these 2 properties. Once stablished for the general curvilinear momentum coordinate case, the same kind of entropy stability results to be proved in this paper would apply for any DG scheme under that formulation of a transformed Boltzmann equation. We therefore prove these two properties for the transport vector $\beta$ (multiplied by the Jacobian)  below.
For $ \beta \cdot \partial H = 0$, we have that
\begin{align}
 \beta \cdot \partial H &=
 J \left( 
 \nabla_{(p_x,p_y,p_z)} \varepsilon,
  \frac{F\cdot \hat{e}_{p_1}}{h_1},
  \frac{F\cdot \hat{e}_{p_2}}{h_2},
  \frac{F\cdot \hat{e}_{p_3}}{h_3}
\right) 
\cdot
\left(-F(\vec{x},t), 
 \partial_{(p_1,p_2,p_3)} \varepsilon
\right)
\nonumber
\\
&= J \left[
-  \nabla_{(p_x,p_y,p_z)} \varepsilon
\cdot F
 +
 F\cdot
 \left(
  \frac{ \hat{e}_{p_1}}{h_1} 
  \partial_{p_1}\varepsilon
  \frac{\hat{e}_{p_2}}{h_2}
  \partial_{p_2}\varepsilon
  \frac{\hat{e}_{p_3}}{h_3}
  \partial_{p_3}\varepsilon
\right)
\right]
= 0,
\nonumber
\end{align}
because the formula that relates the gradient with respect to the cartesian coordinates $(p_x,p_y,p_z)$ to the gradient with respect to the curvilinear momentum coordinates $(p_1,p_2,p_3)$ is precisely 
\begin{equation}
 \nabla_{(p_x,p_y,p_z)} \varepsilon
 =
 \frac{\partial \varepsilon}{ \partial{p_1}}
 \frac{ \hat{e}_{p_1}}{h_1} 
 +  \frac{\partial \varepsilon}{ \partial{p_2}}
 \frac{\hat{e}_{p_2}}{h_2}
 +  \frac{\partial \varepsilon}{ \partial{p_3}}
 \frac{\hat{e}_{p_3}}{h_3} \, .
 \end{equation}
 The second property, related to the divergence of the transport field $\beta$ in the phase space with curvilinear momentum, follows
 from the divergence free property of the transport vector in cartesian space. That is, we know that
 \begin{equation}
  \nabla_{(\vec{x},\vec{p})} \cdot 
  \left(
  \vec{v}(\vec{p}), F(\vec{x},t)
  \right) = 0
 \end{equation}
trivially due to dependance in alternate coordinates. 
However, relating the divergence with respect to the cartesian momentum coordinates to the divergence with respect to the curvilinear ones, we have 
\begin{equation}
0 = \nabla_{(p_x,p_y,p_z)} F(x,t)
=
\frac{1}{J}
\left[
\frac{\partial(F_1 J/h_1)}{\partial p_1}
+
\frac{\partial(F_2 J/h_2)}{\partial p_2}
+
\frac{\partial(F_3 J/h_3)}{\partial p_3}
\right],
\end{equation}
with $F_j = F\cdot \hat{e}_{p_j}, \, j=1,2,3$. So
\begin{align}
0 &= \partial_{(p_1,p_2,p_3)} \cdot
\left\lbrace
J \left(
F_1/h_1, F_2/h_2, F_3/h_3
\right)
\right\rbrace 
\nonumber\\
&=
\partial_{(p_1,p_2,p_3)} \cdot
\left\lbrace
J \left(
F\cdot \hat{e}_{p_1}/h_1, F\cdot \hat{e}_{p_2}/h_2, F\cdot \hat{e}_{p_3}/h_3
\right)
\right\rbrace
\nonumber\\
&=
\partial_{(p_1,p_2,p_3)} \cdot
\left\lbrace
J \left(
F\cdot \hat{e}_{p_1}/h_1, F\cdot \hat{e}_{p_2}/h_2, F\cdot \hat{e}_{p_3}/h_3
\right)
\right\rbrace
+ \nabla_{\vec{x}}\cdot(J \vec{v}(\vec{p}))
=
\partial \cdot \beta,
\nonumber
\end{align}

 with $\partial = \partial_{(\vec{x},p_1,p_2,p_3)} $,
 since $J$ only depends on the momentum. 
\section{Stability of DG schemes under entropy norms}
In this section we'll study the stability of DG schemes under entropy norms, focusing on 3D \& 5D plus time problems.
We'll make some comments as well on the  problem in general for orthogonal curvilinear momentum coordinates.
\subsection{1Dx-2Dp Diode Symmetric Problem}
In the particular case of a 1D silicon diode problem, the main collision mechanisms are electron-phonon scatterings
\begin{equation}
S(\vec{p}\,' \rightarrow \vec{p}) = \sum_{j=-1}^{1} c_j \delta(\varepsilon(\vec{p}\,') - \varepsilon(\vec{p}) + j\hbar \omega), \quad c_1 = (n_{ph} + 1) K, \, c_{-1} = n_{ph} K,
\end{equation}
with $\omega$ the phonon frequency, assumed constant, and $n_{ph}=n_{ph}(\omega)$ the phonon density. $K, \, c_0$ are constants. 
If we assume that the energy band just depends on the momentum norm, $\varepsilon(p), \quad p = |\vec{\bp}|$, and that the initial condition for the pdf has azimuthal symmetry, $f|_{t=0} = f_0(x,p, \mu, t), \, \partial_{\varphi} f_0 = 0$,   $\quad \vec{p} = p (\mu, \sqrt{1-\mu^2} \cos\varphi, \sqrt{1-\mu^2} \sin \varphi) $, then the dimensionality of the problem is reduced to 3D+time, that is, 1D in $x$ and 2D in $(p,\mu)$. 
Our curvilinear coordinates would be in this case
\begin{align}
 p_1 &= |\vec{\bp}| = p, \nonumber\\
 p_2 &= \mu = \cos \theta || E(x,t), \\
 p_3 &= \varphi = \arctan(p_z/p_y) , \nonumber
\end{align}
since the polar axis (reference direction) $p_x$ is chosen parallel to the direction of the 1D electric field $E(x,t)$ in the $x$-direction, 
$p_x || x$. 
The Jacobian determinant is
$$
J_{ \frac{\partial \vec{\bp}}{\partial \vec{p}} } = 
| \frac{\partial (p_x,p_y,p_z)}{\partial (p,\mu,\varphi)} | = - p^2 ,
$$
where the minus sign helps to transform the integral $\int_0^\pi d\theta \cdots$ into $\int_{-1}^{1} d\mu \cdots $ as $\mu=\cos\theta=v_x/|\vec{v}|$, where
$\vec{v}(\vec{p})=(v_x,v_y,v_z)$. 
Then, the BP system for $f(x,p,\mu,t)$ and $\Phi(x,t)$ is written in spherical coordinates $\vec{p}(p,\mu,\varphi)$ for the momentum (under azimuthal symmetry assumptions) as
\begin{align}\label{BP-system}
&\partial_t f + \partial_x (f \partial_p \varepsilon \mu) + \left[ \frac{\partial_p (p^2 f \mu) }{p^2} + \frac{\partial_{\mu} (f(1-\mu^2)) }{p} \right] q\partial_x \Phi(x,t) = Q(f), \nonumber\\
-& \partial_x^2 \Phi = \frac{q}{\epsilon} \left[ N(x) - 
2\pi \int_{-1}^{+1} \int_{0}^{p_{max}} f p^2 dp d\mu, \right], \,
\quad \Phi(0) = 0, \, \Phi(L) = \Phi_0 ,
\end{align}
where the charge density is given by
\begin{equation} \nonumber
\rho(x,t) = 2\pi \int_{-1}^{+1} \int_{0}^{p_{max}} f(x,p,\mu,t) p^2 dp d\mu. 
\end{equation}
We have assumed that the permittivity $\epsilon$ is constant. 
The Poisson BVP with Dirichlet BC in 1D above has an analytic integral solution for $\Phi(x,t)$ and for $E(x,t) = -\partial_x \Phi(x,t)$ as well, which can be projected in the appropriate spaces for our numerical method. The solution is given by the integral formula \cite{MP}
\begin{align} \nonumber
\Phi(x,t) 
 &= 
\left[
\Phi_0 \!+\! 
\frac{q}{\epsilon} 
\!\int_0^L \!\!
  \left[N(x') \!-\! \rho(t,x') \right] (L \!-\! x') dx'
\!\right] \!\frac{x}{L} 
\!-\! \frac{q}{\epsilon}  
\int_0^x \!\!\!
  \left[N(x') \!-\! \rho(t,x') \right] (x \!-\! x') dx' \!,
\\[4pt]
E(x,t) 
 &= 
- \!
\left(
\frac{\Phi_0}{L} \! +\! 
\frac{q}{\epsilon} \frac{1}{L}
\int_0^L \!   
  \left[N(x') \!-\! \rho(t,x') \right] (L-x') dx'
- \frac{q}{\epsilon}
\int_0^x  \!
  \left[N(x') \!-\! \rho(t,x') \right] dx'
\right)\!.
\end{align}

Therefore, in this 1D problem we only need to concern ourselves with the Boltzmann Equation, since given the electron density we know the solution for the potential and electric field in Poisson. \\

The collision operator in this problem has the form
\begin{align*}
Q(f) &=
2\pi \Bigg[ 
\sum_{j=-1}^{+1} c_j  \int_{-1}^{+1} d\mu'  
f(x,p(\varepsilon'),\mu') \, p^2(\varepsilon') \frac{dp'}{d\varepsilon'} \Bigg|_{\varepsilon' = \varepsilon(p) + j\hbar\omega } \chi(\varepsilon(p) + j\hbar\omega) 
\\
&\ \ \qquad\qquad - 
f(x,p,\mu,t)  \sum_{j=-1}^{+1} c_j \, 2  p^2(\varepsilon') \frac{dp'}{d\varepsilon'} \Bigg|_{\varepsilon' = \varepsilon(p) - j\hbar\omega } \chi(\varepsilon(p) - j\hbar\omega)
\Bigg],
\nonumber
\end{align*}
where $\chi(\varepsilon)$ is 1 if $\varepsilon \in [0,\varepsilon_{max}]$ and 0 if $\varepsilon \notin [0,\varepsilon_{max}]$, with $\varepsilon_{max} = \varepsilon(p_{max})$. 

The domain of the Boltzmann Poisson  problem is $x\in [0,L], \, p \in [0, p_{max}], \, \mu \in [-1,+1]$, $t>0$.
Moreover, since $\varepsilon(p)$, 
then $\partial_{\vec{p}} \varepsilon = \frac{d\varepsilon}{dp} \hat{p} $.
We assume that $\varepsilon(p)$ has a positive derivative
$\frac{d\varepsilon}{dp} > 0$, and moreover that
$\varepsilon(p)$
is a continuously differentiable function. 
Therefore the Inverse Function Theorem applies 
at points inside the domain and then
the respective inverse function is continuously differentiable, 
so $p(\varepsilon)$ is a monotonic function for which
$\frac{dp}{d\varepsilon} = ( \frac{d\varepsilon}{dp} )^{-1} $ exists. 
The collision frequency is
\begin{equation*}
\nu(\varepsilon(p)) = 
\sum_{j=-1}^{+1} c_j \, 4 \pi \, 
\chi(\varepsilon(p) - j\hbar\omega) \,
\left. p^2(\varepsilon') \frac{dp'}{d\varepsilon'} \right|_{\varepsilon' = \varepsilon(p) - j\hbar\omega } 
=
\sum_{j=-1}^{+1} c_j n(\varepsilon(p) - j \hbar \omega),
\end{equation*}
where 
 the density of states with energy $\varepsilon(p) - j \hbar \omega$ is
\begin{equation*}
n(\varepsilon(p) - j \hbar \omega) = \int_{\Omega_{\vec{p}}} \delta(\varepsilon(\vec{p}\,') - \varepsilon(\vec{p}) + j \hbar \omega) \, d\vec{p}\, ' . 
\end{equation*}
\subsection{Semidiscrete RK-DG entropy preserving scheme 
(Boltzmann - Poisson 1Dx-2Dp)}
Before writing the weak form of the Transformed Boltzmann Equation and introducing an 
{entropy preserving}
Runge-Kutta Discontinuous Galerkin (RKDG) scheme for the BP system (\ref{BP-system}), for $(x,p,\mu)\in \Omega=\Omega_{x}\times\Omega_{(p,\mu)}\subseteq 
\mathbb{R}^{d}$, we need to discuss the choice of the cut-off domain in momentum space, since the probability density in momentum space may be supported in all its space of definition, while decaying at infinity fast enough to get mass, momentum, and energy bounded. 
The problem can also be restricted to the first position dimension $\vec{x}=(x,0,0)$, $\mathbf{E}(x,t)=(E(x,t),0,0)$. 
%
Since for $f(x,p,\mu), \, g(x,p,\mu) $ we have that
\begin{equation*}
\int_{\Omega_x} \int_{\Omega_{\vec{p}}} f g \, d\vec{p} dx = 
2\pi \, \int_{\Omega_x} \int_{\Omega_{(p,\mu)}} f g \, p^2 \, dp d\mu dx,
\end{equation*}
we define our inner product of two functions $f$ and $g$
in the $(x,p,\mu)$ space as
\begin{equation}
(f,g)_{X \times K} = \int_X \int_K f g \, p^2 \, dp d\mu dx,
\end{equation}
where $X \subset [0,L] $ and $K \subset [0,p_{max}]\times[-1,+1]$.
The Boltzmann Equation for our problem can be written in weak form as 
\begin{align}
( \partial_t f,g)_{\Omega_C} + \left( \partial_x (f \partial_p \varepsilon \mu),\, g \right)_{\Omega_C} &+ \left( \left[ \frac{\partial_p (p^2 f \mu) }{p^2} + \frac{\partial_{\mu} (f(1-\mu^2)) }{p} \right] q\partial_x \Phi(x,t) , g\right)_{\Omega_C} \nonumber\\[4pt]
&\qquad \qquad \qquad \qquad \qquad= \left( Q(f), g \right)_{\Omega_C},
\end{align}
where $\Omega_C = X \times K $. 
More specifically, we have that (assuming $\partial_t g = 0 $)
\begin{align}
&\int_{\Omega_C}  Q(f) \, g \, p^2 dp d\mu dx
=
\partial_t
\int_{\Omega_C} f\,g \, p^2 \, dp d\mu dx \, +
\int_{\Omega_C} \partial_x (f \partial_p \varepsilon \mu) \, g \, p^2 dp d\mu dx  
\nonumber\\[4pt]
&+
\int_{\Omega_C}
{\partial_p (p^2 f \mu) } q\partial_x \Phi(x,t) g \, dp d\mu dx + 
\int_{\Omega_C}
{\partial_{\mu} (f(1-\mu^2)) }  q\partial_x \Phi(x,t) g \, p \, dp d\mu dx.  \nonumber 
\end{align}
{\subsubsection{The Cut-off-domain}
Let's start setting some notational components associated to the computational domain for  the proposed DG scheme. The computational  domain, in 1D in $x$-space and 2D in $\vec{p}$-space,  is denoted by 
\begin{equation}\label{comp domain}
\Omega_C:=L_{x}\times\Omega_{\vec{p},t}=[0,L]_x\times\left[[0,p^E_{\max}]\times[-1,1]\right]_{\vec{p}},
\end{equation} 
\begin{equation*}
L_{x}=[0,L]_x, \quad \Omega_{\vec{p},t}=\left[[0,p^E_{\max}]\times[-1,1]\right]_{\vec{p}}.
\end{equation*}
Thus, the domain in $\vec{p}$-momentum space needs to be  cut-off and redefined by  domain  $\Omega_{\vec{p},t}$ in momentum space{, as it} depends on the electric field $E(x,t)=-\nabla \Phi(x,t)$ according to the mean-field BP flow in \eqref{BP-system}. The particular choice of diameter constant $p^E_{\max}$ is chosen by taking
 \begin{align}\label{LE}
p^E_{\max} = m^*L/T +  TqE^*, \qquad \qquad \text {with}\ \       E^*\geq \max_{x\in\Omega_{x}} |E_n(x)|,
\end{align}
where $T$ is the time our physical evolution problem lasts, $m^*$ is the reduced mass of the electron, and  $E_n(x) \!:=\!E(x,t_n)\!= \!\int_0^x \! \int_{\Omega_{\vec{p},t}} \!f_n(x,p,\mu)p^2dpd\mu dx \!-\!C$, with $C$ a constant, where we denoted $f_n(x,p,\mu)\!:=\! f(x,p,\mu,t_n)$ to be the pdf at the given time 
 $t_n$ to be the discrete time evaluated in the RK method.  Thus, $E^*  \!\le \!\int_{\Omega_{x}\times\Omega_{\vec{p},t}} \!\! f_n(x,p,\mu) p^2 dp d\mu dx \! + \!C.$ 
In particular, this cut-off domain correction in momentum space allows the computational solution $f_{h,n}(x,p,\mu)$ of the Boltzmann flow along the Hamiltonian characteristic field, defined by $(x-v(p,\mu)t_n, p\!+\!qE_n(x)t_n,\mu)$, to keep its initial support transported by the characteristic curves inside the computational domain $\Omega_{\vec{p},t}$.
This assumption 
\textsl{on the cut-off domain depends} on the magnitude of the electric field  as well as on the support of the initial data. Under these conditions  $f(x,p,t)|_{\partial \Omega_{\vec{p},t}}=0$. 
We point out that using periodic boundary conditions in $x$-space set on $\Omega_x$ results in a uniform in time $E^*$, since the solution associated to the {Vlasov - Poisson} system in one dimension  in $x$-space 
yields global uniformly bounded 
electric fields. That means the set $\Omega_{\vec{p},t}$ does not need to be updated with the time step evolution, {regarding the transport process.}

{\subsubsection{The DG-FEM scheme formulation for the transformed Boltzmann equation in the   $(x,p,\mu)$-space computational domain.}
The discretazation of the computational domain   $\Omega_C$, defined in \eqref{comp domain}, is meshed as follows. Set
\begin{align}
\Omega_C&:= \bigcup^{N_x N_p N_\mu}_{i,k,m} \Omega_{ikm}, \quad 
\Omega_{ikm} :=
X_i \times K_{k,m} = 
[x_{i^-}, x_{i^+}] \times [p_{k^-}, p_{k^+}] \times 
[\mu_{m^-}, \mu_{m^+}], 
\nonumber
\end{align}
\vspace{-.25cm}
with the classical notation for mesh midpoint characterization 
\begin{equation}
x_{i^\pm} = x_{i\pm 1/2}, \quad
p_{k^\pm} = p_{k\pm 1/2}, \quad 
\mu_{m^\pm} = \mu_{m\pm 1/2}\, .
\end{equation}
Denote by $\mathcal{T}^{x}_{h}={I_{x}}$ and
$\mathcal{T}^{p,m}_{h}={K_{p,m}}$ the regular partitions of $\Omega_{x}$ and $\Omega_{(p,\mu)}$, respectively, with
\vspace{-.25cm}
\begin{eqnarray*}
\mathcal{T}^{x}_{h}\!\! &=&\!\!\bigcup^{N_{x}}_{1} I_{i}=\bigcup^{N_{x}}_{1}[x_{i^-}, x_{i^+}),  \,
\text{and}
\,
\mathcal{T}^{p,m}_{h}\!\! =\!\!\bigcup^{N_{p}\times N_{\mu}}_{|k,m|=1}K_{j}\!=\!\bigcup^{N_p\times N_\mu}_{k,m=1}  [p_{k^-}, p_{k^+}) \times 
[\mu_{m^-}, \mu_{m^+})\, ,
\end{eqnarray*}
with $x_{1/2}=0$, $x_{N_{x}+1/2}:=L$;\  $p_{1/2}=0$, $p_{N_{p}+1/2}:=p^{E}_{\max}$;\  $\mu_{1/2}=-1$, $\mu_{N_{\mu}+1/2}:=1$, respectively.
Then, $\mathcal{T}_{h}=\{\mathcal{E}: \mathcal{E}=I_{x}\times K_{k}; \ \forall I_{x}\in \mathcal{T}^{x}_{h}, \ \forall K_{p,m}\in\mathcal{T}^{p,m}_{h}\}$ defines a partition of $\Omega$.
Denote by $\varepsilon_{x}$ and $\varepsilon_{k,m}$ the set of edges of $\mathcal{T}^{x}_{h}$ and $\mathcal{T}^{k,m}_{h}$, respectively. Then, the edges of $\mathcal{T}_{h}$
will be $\varepsilon = \{I_{x}\times e_{k}: \forall I_{x}\in \mathcal{T}^{x}_{h}, \forall e_{k,m}\in \varepsilon_{k,m}\}\cup \{e_{x}\times K_{k,m}: \forall e_{x} \in \varepsilon_{x}, \forall K_{k,m}\in \mathcal{T}^{k,m}_{h} \}$. In addition, $\varepsilon_{x}=\varepsilon^{i}_{x}\cup \varepsilon^{b}_{x}$ with $\varepsilon^{i}_{x}$ and $\varepsilon^{b}_{x}$ being the interior and boundary edges, respectively
(same for the domain of momentum variables).
The mesh size is $h:=\max(h_{x},\ h_{k,m}):=\max_{\mathcal{E}\in \mathcal{T}_{h}}\text{diam}(\mathcal{E})$, with $h_{x}:=\max_{I_{x}\in \mathcal{T}^{x}_{h}}\text{diam}(I_{x})$ and \ 
$h_{k,m}:=\max_{K_{k,m}\in \mathcal{T}^{k,m}_{h}}\text{diam}(K_{k,m})$.
Hence, invoking the classical corresponding notation for  their respective internal products in the described mesh becomes
\vspace{-0.2cm}
\begin{equation}\label{discrete inner product}
(f,g)_{\Omega_{ikm}} := \int_{ikm} f g \, p^2 dp d\mu \, dx 
\, .
\end{equation}
\vspace{-0.2cm}
The finite element space is defined as
\begin{equation}
\label{FEspace} V_h^{\mathcal \kappa}=\{ \phi_h \in L^2(\Omega_C) : \forall K \in
\mathcal{T}(\Omega_C), \phi_h |_{K} \in P^{\mathcal \kappa}(K)\},
\end{equation}
where $P^{\mathcal \kappa}(K)$ is the set of polynomials of total degree at most
${\mathcal \kappa}$ on the simplex $K$.

The semi-discrete Discontinuous Galerkin Formulation for our transformed Boltzmann Equation in curvilinear coordinates is to
find $f_h \, \in \, V_h^{\mathcal \kappa}$ such that $\forall \, g_h \, \in V_h^{\mathcal \kappa}$ 
and $\, \forall \, \Omega_{ikm} $, 

\begin{align}
\partial_t  &\int_{ikm}  f_h \, g_h \, p^2 dp d\mu dx 
\pm
\int_{ik}\!\! (1\!-\!\mu_{m\pm}^2)
(-q \widehat{ E f_h })|_{\mu_{m\pm}}   g_h|_{\mu_{m\pm}}^{\mp}  pdp dx
\nonumber\\
&\ -\int_{ikm}  \partial_p \varepsilon (p) \, f_h \, \mu \, \partial_x g_h  p^2 dp d\mu dx 
\pm\! 
 \int_{km} \partial_p \varepsilon \, \widehat{ f_h  \mu }|_{x_{i\pm}}  \, g_h|_{x_{i\pm}}^{\mp}  p^2 dp d\mu  
\nonumber\\
& - \int_{ikm}
{ p^2 } (-qE)(x,t) f_h \mu  \partial_p g_h  d\mu dx
\pm  \!
\int_{im} p^2_{k^\pm} 
(-q  \widehat{ E f_h \mu } )|_{p_{k\pm}} g_h|_{p_{k\pm}}^{\mp}  d\mu dx
\nonumber\\
&\ - \int_{ikm} \!\!\!\!
{ (1\!-\!\mu^2) f_h }  (-qE)(x,t) \partial_{\mu} g_h p dp d\mu dx 
=
\int_{ikm} Q(f_h) g_h \, p^2 dp d\mu dx   \, .
\nonumber
\end{align}

The Numerical Flux used in this scheme is the Upwind Rule,
namely
\begin{align}\label{up-wing-flux}
\widehat{ f_h  \mu }|_{x_{i\pm}} &:=
\left(\frac{\mu + |\mu|}{2} \right) f_h |^-_{x_{i\pm}} +
\left(\frac{\mu - |\mu|}{2} \right) f_h |^+_{x_{i\pm}},
\nonumber\\
- \widehat{ qE \mu f_h  } |_{p_{k\pm}} & :=
\left(\frac{- qE \mu  + |qE \mu |}{2}\right) f_h |^-_{p_{k\pm}} +
\left(\frac{- qE \mu  - |qE \mu |}{2} \right) f_h |^+_{p_{k\pm}},
\\
- \widehat{ qE f_h }|_{\mu_{m\pm}} & := 
\left(\frac{-qE + |qE|}{2} \right) f_h |^-_{\mu_{m\pm}} +
\left(\frac{-qE - |qE|}{2} \right) f_h |^+_{\mu_{m\pm}}.
\nonumber
\end{align}
\subsubsection{Stability of DG scheme under  entropy norm for Periodic Boundary Conditions}
We can prove the stability of the scheme under the entropy norm 
related to the interior product
$
\int_{\Omega_C}   f_h \, g_h e^H \, p^2 dp d\mu dx \, ,
$
inspired in the strategy of Cheng, Gamba, and Proft \cite{CGP}.
These estimates are possible due to the dissipative property of the linear collisional operator applied to the curvilinear representation of the momentum, with the entropy norm related to the 
function $e^{H(\vec{x},\vec{p},t)} = \exp\left(\varepsilon(\vec{p}) -q\Phi(\vec{x},t) \right)$. 
The latter term makes manifest 
a Hamiltonian structure in the entropy norm,
with this same Hamiltonian structure
generating a divergence free characteristic field for the equations of motion.


Existence and uniqueness as well as regularity of initial-boundary value problems (IBVP) associated to the Vlasov - Boltzmann equation (\ref{vlasov},\ref{poisson}) with periodic boundary conditions have been shown by Y. Guo \cite{Guo-2002, Guo-2003} for the non-linear Vlasov - Boltzmann - Poisson - Maxwell system with initial data near a global Maxwellian distribution. It also shows the  regularity propagation of the initial behavior; and further, R. Strain \cite{Guo-Strain-2006}, calculated almost exponential decay rates to such Maxwellian equilibrium.
Additionally, in the particular case of the initial and boundary value problem, N. Ben-Abdallah and M. Tayeb \cite{BenA-Tay} showed existence and uniqueness of solutions to the linear Vlasov - Boltzmann with a continuous in space-time field $E(x,t)$ and non-negative initial and boundary conditions having the same polynomial decay in $L^1\cup L^\infty$ in one space dimension and higher dimensional phase-space (velocity). Such solution preserves the regularity and decay properties of the initial state. While this result uses low regularity of the integrating characteristic field $E(x, t)$ with nonvanishing gradients, it is hoped that higher order Sobolev regularity may propagate for more regular fields, as well as more regular initial and boundary conditions satisfying at least polynomial decay.
However, we are not aware, at this point, whether such result is available.  We also mention that in \cite{BenA-Tay} the authors showed the existence of weak solutions to Boltzmann - Poisson  for incoming data with polynomial decay in the case of one phase-space dimension.

Hence, whenever  the spatial domain is a rectangle, the   assumption of  periodic boundary data in $\vec{x}$-space is  the most suitable condition for stability of the transport flow along divergence free dynamics by means of entropy methods.  
Indeed,  for the numerical approximation to our Boltzmann probability density function $f_h \, \in \, V_h^k$ (the space of trial functions being $V_h^k$, denoting the space of piecewise continuous polynomials of degree $k$) and for all test functions $\forall \, g_h \, \in V_h^k$ 
and $\, \forall \, \Omega_{ikm} $ elements in which the domain is decomposed, it must hold that
\begin{align} \label{entropyPPDGform}
 &\int_{ikm}  \partial_t  f_h \, g_h e^H \, p^2 dp d\mu dx 
\pm  \int_{ik} \!\!(1\!-\!\mu_{m\pm}^2)
(-q \widehat{ E f_h })|_{\mu_{m\pm}}   \, g_h e^H|_{\mu_{m\pm}}^{\mp}  p dp dx  
 \nonumber\\
&\ \ -\int_{ikm}  \partial_p \varepsilon (p) \, f_h \, \mu \, \partial_x (g_h e^H) \, p^2 dp d\mu dx 
\pm \!
 \int_{km} \partial_p \varepsilon \, \widehat{ f_h  \mu }|_{x_{i\pm}}  \, g_h e^H|_{x_{i\pm}}^{\mp}  \, p^2 dp d\mu  
\nonumber\\
&\ \ - \int_{ikm}
{ p^2 } (-qE)(x,t) f_h \mu \, \partial_p (g_h e^H) \, d\mu dx
 \pm \! 
\int_{im} p^2_{k^\pm} \,
(-q  \widehat{ E f_h \mu } )|_{p_{k\pm}} g_h e^H|_{p_{k\pm}}^{\mp}  d\mu dx
\nonumber\\
&\ \ - \int_{ikm}\!\!\!
{ (1\!-\!\mu^2) f_h }  (-qE)(x,t) \, \partial_{\mu} (g_h e^H)  pdp d\mu dx = \int_{ikm} Q(f_h) g_h e^H \, p^2 dp d\mu dx, 
\end{align}
where we are including  as a factor the inverse of a 
Maxwellian along the characteristic flow generated 
by the Hamiltonian transport field 
$\left(\partial_p \varepsilon(p), q\partial_x \Phi(x,t) \right) $,
\begin{equation}
e^{H(x,p,t)} = \exp(\varepsilon(p) -q\Phi(x,t)) = 
\left( e^{q\Phi(x,t)} e^{-\varepsilon(p)} \right)^{-1} \, .
\end{equation}
This is clearly an exponential of the Hamiltonian energy, 
assuming the energy is measured in $K_B T $ units.
We include this modified inverse Maxwellian factor because 
we can use some entropy inequalities related to the collision operator. 
Our collision operator satisfies the dissipative property
\begin{equation}
\int_{\Omega_{C,\vec{p}}} Q(f) g d\vec{p} = 
-\frac{1}{2} \int_{\Omega_{\vec{p}}} S(\vec{p}'\rightarrow \vec{p})
e^{-\varepsilon(p')} \left(\frac{f'}{e^{-\varepsilon(p')}} \!-\!
\frac{f}{e^{-\varepsilon(p)}} \right)(g' \!-\! g) d\vec{p}d\vec{p},
\end{equation}
which can be also expressed as (multiplying and dividing by 
$e^{-q\Phi(x,t)}$) 
\begin{equation}
\int_{\Omega_{C,\vec{p}}}  Q(f) g d\vec{p} = 
-\frac{1}{2} \int_{\Omega_{\vec{p}}} S(\vec{p}\,'\rightarrow \vec{p})
e^{-H'} \left(\frac{f'}{e^{-H'}} -
\frac{f}{e^{-H}} \right) (g' - g) d\vec{p}\,' d\vec{p} .
\end{equation}

Therefore, if we choose a monotone increasing function 
$g({f}/{e^{-H}})$, namely $g = f/e^{-H} = f e^H $, we have
an equivalent dissipative property but now with
the exponential of the full Hamiltonian,
\begin{equation}
\int_{\Omega_{C,\vec{p}}}  Q(f) \frac{f}{e^{-H}} d\vec{p} = 
-\frac{1}{2} \int_{\Omega_{\vec{p}}} S(\vec{p}\,'\rightarrow \vec{p})
e^{-H'} \left(\frac{f'}{e^{-H'}} -
\frac{f}{e^{-H}} \right)^2 d\vec{p}\,' d\vec{p} \leq 0 .
\end{equation}
So we have found the dissipative entropy inequality
\begin{equation}
\int_{\Omega_{C,\vec{p}}}  Q(f) fe^H p^2 dp d\mu d\varphi = 
\int_{\Omega_{C,\vec{p}}}  Q(f) \frac{f}{e^{-H}} d\vec{p} 
\leq 0  \, .
\end{equation}

As a consequence of this entropy inequality we obtain the following stability theorem of the scheme under an entropy norm.
\begin{theorem}\label{thm4.2}
\emph{(Stability under the entropy norm $ \int f_h \, g_h e^H \, p^2 dp d\mu dx$ for a given periodic potential $\Phi(x,t)$):}
Consider the semi-discrete solution $f_h$ 
to the DG formulation in (\ref{entropyPPDGform})
for the BP system in momentum curvilinear coordinates. 
We have then that
\begin{equation}
0 \geq \int_{\Omega_C}f_h \partial_t  f_h \,  e^{H(x,p,t)} \, p^2 \, dp d\mu dx 
= \frac{1}{2}
\int_{\Omega_C}  \partial_t  f_h^2 e^{H(x,p,t)} \, p^2 \, dp d\mu dx  \, .
\end{equation}
\end{theorem}

\begin{proof}
Choosing $g_h = f_h$ in (\ref{entropyPPDGform}),
and considering the union of all the cells $\Omega_{ikm}$,
which gives us the 
 whole domain  $\Omega_C = \Omega_x \times \Omega_{p,\mu}$ for integration, we have
\begin{align} 
0 \geq 
\int_{\Omega_C} Q(f_h) f_h e^H \, p^2 dp d\mu dx
&=
  \int_{\Omega_C} \partial_t  f_h \, f_h e^H \, p^2 dp d\mu dx 
\nonumber\\
-\int_{\Omega_C}  \partial_p \varepsilon (p) \, f_h \, \mu \, \partial_x (f_h e^H) \, p^2 dp d\mu dx \,
&+ 
 \int_{\partial_x\Omega_C} \partial_p \varepsilon \, \widehat{ f_h  \mu }  \, f_h e^H  \, p^2 dp d\mu  
\nonumber\\
- \int_{\Omega_C}
{ p^2 } (-qE) f_h \mu \, \partial_p (f_h e^H) \, d\mu dx
&+  
\int_{\partial_p \Omega_C} p^2  \,
(-q  \widehat{ E f_h \mu } )  f_h e^H  \, d\mu dx
\nonumber\\
- \int_{\Omega_C}
{ (1-\mu^2) f_h }  (-qE) \, \partial_{\mu} (f_h e^H) \, p \, dp d\mu dx 
&+ 
\int_{\partial_{\mu} \Omega_C} (1-\mu^2)
(-q \widehat{ E f_h }) \, f_h e^H \, p \, dp dx \, .
\nonumber
\end{align}
We can express this in the more compact form
\begin{equation} 
0 \geq \int_{\Omega_C} \partial_t  f_h \, f_h e^H \, p^2 \, dp d\mu dx 
-\int_{\Omega_C} f_h \beta \cdot \partial (f_h e^H) \,dp d\mu dx\,
+
\int_{\partial \Omega_C}
 \widehat{ f_h }  \beta \cdot \hat{n} \, f_h e^H \, d\sigma,
\end{equation}
defining the transport vector 
 \begin{equation}
\beta = \left( p^2 \mu \partial_p \varepsilon(p), -qE \, p^2 \mu, 
-qE p (1-\mu^2) \right) \,  .
\end{equation}
We integrate by parts again the transport integrals, obtaining that
\begin{align}
\int_{\Omega_C} f_h \beta \cdot \partial (f_h e^H) \,dp d\mu dx\, 
&=
- \int_{\Omega_C} \partial \cdot ( f_h \beta) f_h e^H \,dp d\mu dx\, 
+ \int_{\partial \Omega_C} f_h \beta \cdot \hat{n} f_h e^H \,d\sigma\,
\nonumber\\
&=
- \int_{\Omega_C} (\beta \cdot \partial  f_h) f_h e^H \,dp d\mu dx\, 
+ \int_{\partial \Omega_C} f_h \beta \cdot \hat{n} f_h e^H \,d\sigma\, ,
\nonumber
\end{align}
but since
\begin{equation}
\beta \cdot \partial (f_h e^H) = \beta \cdot e^H \partial f_h 
+ \beta \cdot f_h e^H \partial H = e^H \beta \cdot \partial f_h,
\end{equation}
 we have then
\begin{equation}
\int_{\Omega_C} f_h \beta \cdot \partial (f_h e^H) \,dp d\mu dx\, 
=
\int_{\Omega_C}  f_h e^H \beta \cdot \partial  f_h \,dp d\mu dx\, 
=
\frac{\int_{\partial \Omega_C} f_h \beta \cdot \hat{n} f_h e^H \,d\sigma}{2} 
.
\end{equation}
We can express our entropy inequality then as
\begin{equation} 
0 \geq \int_{\Omega_C} \partial_t  f_h \, f_h e^H \, p^2 \, dp d\mu dx 
-\frac{1}{2} 
\int_{\partial \Omega_C} f_h \beta \cdot \hat{n} f_h e^H \,d\sigma\,
+
\int_{\partial \Omega}
 \widehat{ f_h }  \beta \cdot \hat{n} \, f_h e^H \, d\sigma \, ,
\end{equation}
remembering that we are integrating over the whole domain 
(the union of all cells defining our mesh).
We distinguish between boundaries of cells for which $\beta \cdot \hat{n} \geq 0$ and the ones for which $\beta \cdot \hat{n} \leq 0$, defining uniquely the boundaries. 
Remembering that the upwind rule \eqref{up-wing-flux} is such that $\hat{f}_h = f_h^-$, we have that the  solution value inside the cells close to boundaries for which $\beta\cdot\hat{n} \geq 0$ is $f_h^-$, and
for boundaries $\beta \cdot \hat{n} \leq 0 $ the solution value inside the cell close to that boundary is $f_h^+$. We have
\begin{align} 
0 \geq 
&\int_{\Omega_C} \partial_t  f_h \, f_h e^H  p^2  dp d\mu dx 
-\frac{1}{2} 
\int_{\partial \Omega_C} f_h \beta \cdot \hat{n} f_h e^H d\sigma
+
\int_{\partial \Omega}
 f_h^-   \beta \cdot \hat{n}  f_h e^H d\sigma
\nonumber
\end{align}
and
\begin{align}
0 \geq
&\int_{\Omega_C} \partial_t  f_h  f_h e^H  p^2 dp d\mu dx 
-\frac{1}{2} 
\int_{\beta\cdot\hat{n}\geq 0}\!\! f_h^- |\beta \cdot \hat{n}| f_h^- e^H d\sigma\,
+
\int_{\beta\cdot\hat{n}\geq 0}
 f_h^-  |\beta \cdot \hat{n}| \, f_h^- e^H \, d\sigma
 \nonumber\\
&\ \ + 
\frac{1}{2} 
\int_{\beta\cdot\hat{n}\leq 0} \!\!f_h^+ |\beta \cdot \hat{n}| f_h^+ e^H d\sigma\,
-
\int_{\beta\cdot\hat{n}\leq 0}
 f_h^-  |\beta \cdot \hat{n}| \, f_h^+ e^H \, d\sigma. 
 \nonumber
\end{align}
{The notation  $e_h$  is generic in finite elements, counting boundaries twice given that it indexes elements, so it must be balanced by a factor of 1/2}. 
We have
\begin{align} 
0 \geq 
&\int_{\Omega_C} \partial_t  f_h  f_h e^H \, p^2 dp d\mu dx 
\!+\!\frac{1}{2} \left(
-\frac{1}{2}
\int_{e_h} f_h^- |\beta \cdot \hat{n}| f_h^- e^H d\sigma
\!+\!
\int_{e_h}\!\!
 f_h^-  |\beta \cdot \hat{n}|  f_h^- e^H  d\sigma
\right. \nonumber\\
 &\ \ \ \qquad \qquad \qquad \qquad \qquad  + \left.
\frac{1}{2} 
\int_{e_h} f_h^+ |\beta \cdot \hat{n}| f_h^+ e^H d\sigma\,
-
\int_{e_h}
 f_h^-  |\beta \cdot \hat{n}| \, f_h^+ e^H d\sigma
 \right);
 \nonumber\\
 0 \geq 
&\int_{\Omega_C} \partial_t  f_h \, f_h e^H  p^2  dp d\mu dx 
+
\frac{1}{4} 
\int_{e_h} (f_h^+ - f_h^-)^2 |\beta \cdot \hat{n}|  e^H d\sigma. \nonumber
\end{align}

Since the second term is non-negative, we conclude that
\begin{equation}\label{control_2}
0 \geq
- \frac{1}{4} 
\int_{e_h} (f_h^+ - f_h^-)^2 |\beta \cdot \hat{n}|  e^H d\sigma\,
\geq
 \frac{1}{2}
\int_{\Omega_C}  \partial_t  f_h^2 e^{H(x,p,t)}  p^2  dp d\mu dx  \, ,
\end{equation}
and it is in this sense that the numerical solution has stability
with respect to the considered entropy norm. 
\end{proof}

In addition, the following  result holds.
\begin{corollary}\label{stab-Phi(x)}
\emph{(Stability under entropy norm for time independent Hamiltonian):}
If $\Phi=\Phi(x)$, so $\partial_t H = 0$,
the stability under our entropy norm gives us that for $t\geq 0$
 \begin{equation}
\left| \left| f_h \right| \right|_{L^2_{e^H p^2}}^2 (t)
=
\int_{\Omega_C}  f_h^2(x,p,\mu,t) e^{H(x,p)}  p^2 dp d\mu dx 
\leq 
\left| \left| f_h \right| \right|_{L^2_{e^H p^2}}^2 (0) \quad .
\end{equation}
\end{corollary}

\begin{proof}
The corollary follows from the fact that, since 
$\partial_t H = -q \partial_t \Phi = 0$, we have
 \begin{equation}
0 \geq 
\int_{\Omega_C}  \partial_t \left( f_h^2 e^{H(x,p)} \right)\, p^2  dp d\mu dx 
= \frac{d}{dt}
\int_{\Omega_C}  f_h^2(x,p,\mu,t) e^{H(x,p)}   p^2 dp d\mu dx  \, .
\end{equation}
Since the entropy norm decreases over time,
our result follows immediately. 
\end{proof}

\vspace{-0.25cm}
\subsection{2DX-3DK Problem: DG scheme and Stability under an entropy norm for Periodic and Reflective Boundary Conditions}
We consider now a 2D problem in position space (which requires a 3D dimensionality in momentum), using as momentum coordinates
the same as in \cite{CGMS-CMAME2008}: 
a normalized polar component $\mu$, the azimuthal angle $\varphi$, 
and a normalized Kane energy band $\omega$, 
\begin{equation}
 \omega = \frac{\varepsilon}{K_B T_L},
\end{equation}
with $K_B$ the Boltzmann constant and $T_L$ the temperature of the material lattice, the (unnormalized) Kane energy band being
\begin{equation}
 \varepsilon (1+\alpha \varepsilon)=\frac{\hbar^2 k^2}{2m^*}, 
\end{equation}
with $\alpha$ the so-called Kane constant, $\hbar$ Planck's constant, and $m^*$ the reduced mass for the semiconductor material. 
This problem has the following semi-discrete DG formulation using an entropy norm: to find a function $f_h$ such that for all $g_h$ it holds that
\begin{align}
&\int_{\Omega_C} Q(f_h) g_h e^H   s(\omega) d\vec{\omega} d\vec{x}
=
\int_{\partial \Omega_C}
 \widehat{ f_h }  \beta \cdot \hat{n}  g_h e^H  d\sigma
  \nonumber
\end{align}
\begin{align} \label{eq:underlyingKane}
&+ 
\int_{\Omega_C} 
(
\partial_t  f_h  g_h e^H  s(\omega)  
-
f_h \beta \cdot \partial (g_h e^H) 
)
d\vec{\omega} d\vec{x},
\end{align}
\noindent defining $\partial = \partial_{(x,y,\omega,\mu,\varphi)} $ and the transport vector $\beta$ s.t. $\partial \cdot \beta = 0$ and $\beta \cdot \partial H = 0$,
\begin{eqnarray}
\beta &=& (\beta_1, \beta_2,\beta_3, \beta_4,\beta_5)(x,y,\omega,\mu,\varphi), \\
(\beta_1, \beta_2) &=& 
c_x { {w(1+\ak w)}}{} (\mu
,  \sqrt{1-\mu^2} \cos\ph),
  \nonumber \\
\beta_3
& = &
- c_k 
{2 {w(1+\ak w)}}{}
\left[ \mu \, E_x + \sqrt{1-\mu^2} \cos\ph \, E_y
\right]  \, ,   \nonumber\\
\beta_4
& = &
- c_k 
 {\sqrt{1-\mu^2}}{ {(1+2\ak
w)}}
 \left[ \sqrt{1-\mu^2} \, E_x  - \mu \cos\ph \, E_y \right]
 \, ,   \nonumber\\
\beta_5 &=& - c_k 
 \frac{ - (1+2\ak w) E_y \sin\ph}{ \sqrt{1-\mu^2}}
\,  ,
  \nonumber\\
\frac{-
\partial \cdot \beta /c_k
}{1+2\ak w 
}
&=& 
2\left[ \mu \, E_x + \sqrt{1-\mu^2} \cos\ph \, E_y
\right]  ,
\nonumber\\
&-& 
 {2\mu} \, E_x  -\, \left(\sqrt{1-\mu^2} -\frac{\mu^2}{\sqrt{1-\mu^2}} \right)  \cos\ph \, E_y 
 -
 \frac{ E_y \cos\ph}{ \sqrt{1-\mu^2}} \, ,
\nonumber
\end{eqnarray}
which is zero since
\begin{eqnarray}
2\left[ \mu \, E_x + \sqrt{1-\mu^2} \cos\ph \, E_y
\right]  
 {-2\mu} \, E_x  
&&
\nonumber\\
- \left(\sqrt{1-\mu^2} -\frac{\mu^2}{\sqrt{1-\mu^2}} \right)  \cos\ph \, E_y 
 -
 \frac{ E_y \cos\ph}{ \sqrt{1-\mu^2}} 
&=&
\nonumber\\
2  \sqrt{1-\mu^2} \cos\ph \, E_y
 -
 \frac{ E_y \cos\ph}{ \sqrt{1-\mu^2}} 
&=&
\nonumber\\
  - \left(\sqrt{1-\mu^2} -\frac{\mu^2}{\sqrt{1-\mu^2}} \right)  \cos\ph \, E_y 
&=&
\nonumber\\
  \sqrt{1-\mu^2} \cos\ph \, E_y
  +   \frac{\mu^2}{\sqrt{1-\mu^2}}   \cos\ph \, E_y 
 -
 \frac{ E_y \cos\ph}{ \sqrt{1-\mu^2}} 
&=&
\nonumber\\
\frac{\cos\ph \, E_y
}{  \sqrt{1-\mu^2} }
\left(
  {1-\mu^2} 
  +   {\mu^2}   
 -
1
\right)
 &=& 0.
 \nonumber
\end{eqnarray}
As we know, $\partial \cdot \beta = 0$
and $  \beta \cdot \partial H =0$.
We state now our result regarding the stability under an entropy norm for this problem under periodic and specular reflection boundary conditions.
\begin{theorem}\label{thm4.2}
\emph{(Stability under the entropy norm $ \int f_h \, g_h e^H \, s(\omega) d\vec{\omega} d\vec{x}$ for a given  potential $\Phi(x,t)$):}
Consider the semi-discrete solution $f_h$ 
to the DG formulation 
for the BP system in radial energy-angular coordinates (assuming a Kane band model) under periodic and specular reflection boundary conditions. 
We have then 
\begin{equation}
0 \geq \int_{\Omega_C}f_h \partial_t  f_h \,  e^{H} \, s(\omega) \, d\vec{\omega} d\vec{x} 
= \frac{1}{2}
\int_{\Omega_C}  \partial_t  f_h^2 e^{H} \, s(\omega) \, d\vec{\omega} d\vec{x}  \, .
\end{equation}
\end{theorem}
\begin{proof}
Choosing $g_h = f_h$ in the entropy inequality for the Boltzmann equation,
and considering the union of all the cells (now $\Omega_{ijkmn}$),
which gives us the 
 whole domain  $\Omega_C = \Omega_{x,y} \times \Omega_{\omega,\mu,\varphi}$ for integration, we have
\begin{align}
& 0 \geq
 \int_{\Omega_C} Q(f_h) f_h e^H s(\omega) d\vec{\omega} d\vec{x}
=  
\end{align}
\begin{align}
& =\int_{\Omega_C} 
(
\partial_t  f_h f_h e^H  s(\omega) 
-
f_h \beta \cdot \partial (f_h e^H) 
)
d\vec{\omega} d\vec{x}
+
\int_{\partial \Omega_C}
 \widehat{ f_h }  \beta \cdot \hat{n} f_h e^H d\sigma \, .
  \nonumber
\end{align}
We integrate by parts again the transport integrals, obtaining
\begin{eqnarray}
\int_{\Omega_C} f_h \beta \cdot \partial (f_h e^H) \,d\vec{\omega} d\vec{x}\, 
&=&
- \int_{\Omega_C} \partial \cdot ( f_h \beta) f_h e^H \,d\vec{\omega} d\vec{x}\, 
+ \int_{\partial \Omega_C} f_h \beta \cdot \hat{n} f_h e^H \,d\sigma\,
\nonumber\\
&=&
- \int_{\Omega_C} (\beta \cdot \partial  f_h) f_h e^H \,d\vec{\omega} d\vec{x}\, 
+ \int_{\partial \Omega_C} f_h \beta \cdot \hat{n} f_h e^H \,d\sigma\, ,
\nonumber
\end{eqnarray}
but since
\begin{equation*}
\beta \cdot \partial (f_h e^H) = \beta \cdot e^H \partial f_h 
+ \beta \cdot f_h e^H \partial H = e^H \beta \cdot \partial f_h,
\end{equation*}
 we have then
\begin{equation*}
\int_{\Omega_C} f_h \beta \cdot \partial (f_h e^H) \,d\vec{\omega} d\vec{x}\, 
=
\int_{\Omega_C} (\beta \cdot \partial  f_h) f_h e^H \,d\vec{\omega} d\vec{x}\, 
=
\frac{1}{2} 
\int_{\partial \Omega_C} f_h \beta \cdot \hat{n} f_h e^H \,d\sigma\, \, .
\end{equation*}
We can express our entropy inequality then as
\begin{equation} 
0 \geq \int_{\Omega_C} \partial_t  f_h \, f_h e^H \, s(\omega) \, d\vec{\omega} d\vec{x} 
-\frac{1}{2} 
\int_{\partial \Omega_C} f_h \beta \cdot \hat{n} f_h e^H \,d\sigma\,
+
\int_{\partial \Omega}
 \widehat{ f_h }  \beta \cdot \hat{n} \, f_h e^H \, d\sigma \, ,
\end{equation}
remembering that we are integrating over the whole domain by considering the union of all the cells defining our mesh.
We will distinguish between the internal edges and the external ones where periodic and specular reflection boundary conditions are applied, that is,  $\partial \Omega = IE \cup PB \cup RB $, so
\begin{eqnarray} 
0 &\geq& \int_{\Omega_C} \partial_t  f_h \, f_h e^H \, s(\omega) \, d\vec{\omega} d\vec{x} 
-\frac{1}{2} 
\int_{IE\cup PB} f_h \beta \cdot \hat{n} f_h e^H \,d\sigma\,
 \nonumber\\ 
+
\int_{IE\cup PB}
 \widehat{ f_h }  \beta \cdot \hat{n} \, f_h e^H \, d\sigma
 &-& \frac{1}{2} 
\int_{RB} f_h \beta \cdot \hat{n} f_h e^H \,d\sigma\,
+
\int_{RB}
 \widehat{ f_h }  \beta \cdot \hat{n} \, f_h e^H \, d\sigma \, .
 \nonumber
 \end{eqnarray}
 
The calculation for internal edges and boundaries with periodic conditions both follow the same spirit, whereas the reflective boundary has to be treated separately as it will vanish by itself considering the inflow-outflow regions. 

\subsubsection{Specular Reflection Boundaries}
We divide the reflection boundaries
in inflow and outflow regions
(remembering that we are dealing with 
hyper-surfaces) 
\begin{eqnarray*} 
&& -\frac{1}{2} 
\int_{RB} f_h \beta \cdot \hat{n} f_h e^H \,d\sigma\,
+
\int_{RB}
 \widehat{ f_h }  \beta \cdot \hat{n} \, f_h e^H \, d\sigma \, 
 =
 -\frac{1}{2} 
\int_{\beta\cdot\ \hat{n}>0} f_h \beta \cdot \hat{n} f_h e^H \,d\sigma\,
\\ &&
+
\int_{\beta\cdot\ \hat{n}>0}
 \widehat{ f_h }  \beta \cdot \hat{n} \, f_h e^H \, d\sigma 
 -\frac{1}{2} 
\int_{\beta\cdot\ \hat{n}<0} f_h \beta \cdot \hat{n} f_h e^H \,d\sigma\,
+
\int_{\beta\cdot\ \hat{n}<0}
 \widehat{ f_h }  \beta \cdot \hat{n} \, f_h e^H \, d\sigma \,  ,
 \end{eqnarray*}
 and remembering also that
 the specular reflection boundary condition relates the inflow and the outflow boundary by 
 \begin{equation}
 f(\vec{x},\vec{p},t) \midminus
 =
 f(\vec{x},\vec{p}\,',t) \midplus \, .
 \end{equation}

 The integrals with a factor of 
 $1/2$ vanish each other due to the specular reflection
 (after having done a transformation
 of coordinates from the inflow boundary to the outflow boundary), the remaning term being
\begin{eqnarray*} 
\int_{\beta\cdot\ \hat{n}>0}
 \widehat{ f_h }  |\beta \cdot \hat{n}| \, f_h e^H \, d\sigma 
-
\int_{\beta\cdot\ \hat{n}<0}
 \widehat{ f_h }  |\beta \cdot \hat{n}| \, f_h e^H \, d\sigma 
 =\\
\int_{\beta\cdot\ \hat{n}>0}
 \widehat{ f_h }\midplus  |\beta \cdot \hat{n}| \, f_h\midplus e^H \, d\sigma 
-
\int_{\beta\cdot\ \hat{n}<0}
 \widehat{ f_h }\midminus  |\beta \cdot \hat{n}| \, f_h\midminus e^H \, d\sigma  
=\\
\int_{\beta\cdot\ \hat{n}>0}
 \widehat{ f_h }\midplus  |\beta \cdot \hat{n}| \, f_h\midplus e^H \, d\sigma 
-
\int_{\beta\cdot\ \hat{n}<0}
 \widehat{ f_h }\midminus  |\beta \cdot \hat{n}| \, f_h\midminus e^H \, d\sigma  
=\\
\int_{\beta\cdot\ \hat{n}>0}
 \widehat{ f_h }\midplus  |\beta \cdot \hat{n}| \, f_h\midplus e^H \, d\sigma 
-
\int_{\beta\cdot\ \hat{n}<0}
 \widehat{ f_h }\midminus  |\beta \cdot \hat{n}| \, f_h'\midplus e^H \, d\sigma  
=\\
\int_{\beta\cdot\ \hat{n}>0}
 \widehat{ f_h }\midplus  |\beta \cdot \hat{n}| \, f_h\midplus e^H \, d\sigma 
-
\int_{\beta\cdot\ \hat{n}>0}
 \widehat{ f_h}\midminus '  |\beta \cdot \hat{n}| \, (f_h'\midplus) ' e^{H'} \, d\sigma  ,
 \end{eqnarray*}
but $H=H'$ since
 $\varepsilon(|\vec{p}|) 
 = \varepsilon(|\vec{p}\, '|) $, so
\begin{eqnarray*}
 \int_{\beta\cdot\ \hat{n}>0}
 \widehat{ f_h }\midplus  |\beta \cdot \hat{n}| \, f_h\midplus e^H \, d\sigma 
-
\int_{\beta\cdot\ \hat{n}>0}
 \widehat{ f_h}\midminus '  |\beta \cdot \hat{n}| \, (f_h'\midplus) ' e^{H'} \, d\sigma  
=\\
 \int_{\beta\cdot\ \hat{n}>0}
\left[
 \widehat{ f_h }\midplus   
 -
\widehat{ f_h}\midminus '  
\right]
|\beta \cdot \hat{n}| \, f_h\midplus 
e^H \, d\sigma
=
0,
\end{eqnarray*}
by using the upwind rule as numerical flux \eqref{up-wing-flux},
denoting $z = (\vec{x},\vec{p})$,
\begin{equation}
 \hat{f}_h = 
 \lim_{\delta \to 0^+}
 f_h(z - \delta \beta(z)) .
\end{equation}


\subsubsection{Internal Edges and Periodic Boundaries}

Both internal edges and periodic boundaries are characterized by the fact that each edge has a pairing with another unique edge for which the conditions
$\hat{n}_1 = -\hat{n}_2$ and
$\beta_1 = \beta_2$ hold,
so the idea behind these calculations is to recognize and consider jointly these pairs. 

We distinguish between the boundaries of cells for which $\beta \cdot \hat{n} \geq 0$ and the ones for which $\beta \cdot \hat{n} \leq 0$, defining uniquely the boundaries. Remembering that the upwind flux rule is such that $\hat{f}_h = f_h^-$, we have that the value of the solution inside the cells close to boundaries for which $\beta\cdot\hat{n} \geq 0$ is $f_h^-$, and
for boundaries $\beta \cdot \hat{n} \leq 0 $ the value of the solution inside the cell close to that boundary is $f_h^+$. Then, it follows
\begin{eqnarray*} 
0 &\geq& 
\int_{\Omega_C}  f_h e^H  s(\omega) \partial_t  f_h d\vec{\omega} d\vec{x} 
- 
\int_{IE\cup PB}  
\frac{f_h \beta \cdot \hat{n} f_h e^H}{2}
d\sigma
+
\int_{IE \cup PB}
 f_h^-   \beta \cdot \hat{n}  f_h e^H  d\sigma, 
\nonumber\\
0 &\geq& 
\int_{\Omega_C} \partial_t  f_h  f_h e^H  s(\omega)  d\vec{\omega} d\vec{x} 
- 
\int_{\beta\cdot\hat{n}\geq 0}
\frac{f_h^- |\beta \cdot \hat{n}| f_h^- e^H}{2}
d\sigma
+
\int_{\beta\cdot\hat{n}\geq 0}
 f_h^-  |\beta \cdot \hat{n}|  f_h^- e^H  d\sigma
 \nonumber\\
&\ \ \ &+ 
\frac{1}{2} 
\int_{\beta\cdot\hat{n}\leq 0} f_h^+ |\beta \cdot \hat{n}| f_h^+ e^H d\sigma\,
-
\int_{\beta\cdot\hat{n}\leq 0}
 f_h^-  |\beta \cdot \hat{n}|  f_h^+ e^H  d\sigma.
 \nonumber
\end{eqnarray*}
Thus, using a notation $e_h$ indexing each element (which counts twice element boundaries, and due to this allowed redundancy, is balanced by a factor of 1/2), we have
\begin{eqnarray} 
0 &\geq& 
\int_{\Omega_C} \partial_t  f_h  f_h e^H  s(\omega)  d\vec{\omega} d\vec{x} 
+\frac{1}{2} \left(
\int_{e_h}
 f_h^-  |\beta \cdot \hat{n}| \, f_h^- e^H  d\sigma
-
\int_{e_h}  
\frac{f_h^- |\beta \cdot \hat{n}| f_h^- e^H}{2}
d\sigma
\right. \nonumber\\
&\ \ \ &+ \left.
\frac{1}{2} 
\int_{e_h} f_h^+ |\beta \cdot \hat{n}| f_h^+ e^H d\sigma\,
-
\int_{e_h}
 f_h^-  |\beta \cdot \hat{n}| \, f_h^+ e^H  d\sigma
 \right),
 \nonumber\\
 0 &\geq& 
\int_{\Omega_C} \partial_t  f_h \, f_h e^H \, s(\omega) \, d\vec{\omega} d\vec{x} 
+\frac{1}{2} \left( \frac{1}{2}
\int_{e_h} f_h^- |\beta \cdot \hat{n}| f_h^- e^H \,d\sigma\,
\right. 
 \nonumber\\
&\ \ \ &+  \left.
\frac{1}{2} 
\int_{e_h} f_h^+ |\beta \cdot \hat{n}| f_h^+ e^H \,d\sigma\,
-
\int_{e_h}
 f_h^-  |\beta \cdot \hat{n}| \, f_h^+ e^H \, d\sigma
\right),
 \nonumber\\
 0 &\geq& 
\int_{\Omega_C} \partial_t  f_h \, f_h e^H \, s(\omega) \, d\vec{\omega} d\vec{x} 
 \nonumber\\
&\ \ \ & + \frac{1}{4}  
 \left(
\int_{e_h} f_h^- f_h^- |\beta \cdot \hat{n}|  e^H d\sigma
-2
\int_{e_h}
 f_h^- f_h^+  |\beta \cdot \hat{n}|   e^H  d\sigma
+
\int_{e_h} f_h^+ f_h^+ |\beta \cdot \hat{n}|  e^H \,d\sigma\,
\right),
 \nonumber\\
 0 &\geq& 
\int_{\Omega_C} \partial_t  f_h \, f_h e^H \, s(\omega) \, d\vec{\omega} d\vec{x} 
+
\frac{1}{4} 
\int_{e_h} (f_h^+ - f_h^-)^2 |\beta \cdot \hat{n}|  e^H \,d\sigma\, .
\nonumber
\end{eqnarray}

Since the second term is non-negative, we conclude therefore that
\begin{equation}\label{control_2}
0 \geq
- \frac{1}{4} 
\int_{e_h} (f_h^+ - f_h^-)^2 |\beta \cdot \hat{n}|  e^H \,d\sigma\,
\geq
 \frac{1}{2}
\int_{\Omega_C}  \partial_t  f_h^2 e^{H(x,p,t)} \, s(\omega) \, d\vec{\omega} d\vec{x}  \, ,
\end{equation}
and it is in this sense that the numerical solution has stability
with respect to the entropy norm under consideration, namely
\begin{equation}\label{control_2}
 \frac{1}{2}
\int_{\Omega_C}  \partial_t  f_h^2 e^{H(x,p,t)} \, s(\omega) \, d\vec{\omega} d\vec{x}  \leq  \frac{1}{2}
\int_{\Omega_C}  \partial_t  f_h^2(x,t, 0) e^{H(x,p,0)} \, s(\omega) \, d\vec{\omega} d\vec{x} .
\end{equation}
\end{proof}


\subsection{Stability under entropy norm for time dependent Hamiltonian
in mean field limit (1Dx-2Dp Problem)}

We perform  further analysis of these discrete entropy inequalities  applied to the Boltzmann - Poisson system (\ref{vlasov},\ref{poisson}) for self consistent mean field charged transport, in the one space dimensional case, say $x\in[0,L]$, under the assumption for   the electrostatic potential  $\Phi(x)$ solving the  Poisson equation that it satisfies periodic boundary conditions as much as the probability density $f(x,p,t)$. These boundary  conditions on the electrostatic potential can be viewed as having neutral charges in a neighborhood containing  the endpoints $\{0;L\}$ and zero potential bias, that is, the corresponding Poisson boundary value problem for the potential is  
\begin{align}\label{per-poisson}
&- \partial_x^2 \Phi(x,t)=\frac{q}{\epsilon} \left[ N(x)- \rho_h(x,t) \right], \\
&\ \ \Phi(0,t) - \Phi(L,t),  \quad  \partial_x \Phi(0,t)= \partial_x \Phi(L), \  \forall \  t>0. \nonumber
\end{align}
Indeed, these boundary  conditions for the homogeneous problem imply that solutions are determined up to a constant. Thus, in order to obtain existence of solutions, the   Fredlhom  Alternative property indicates   that existence  holds provided the   compatibility condition  $\int_0^L  \left[ N(x') - \rho_h(x',t) \right] dx' =   0$,  which yields neutral total charges  for all times $t$. In addition,  to obtain uniqueness, one needs to prescribe an extra condition on the $x$-space average of the solution $ \int_0^L \Phi(x, t) dx$.

Therefore the following Theorem also holds for the semi-discrete Vlasov - Boltzmann - Poisson system in one $x$-space dimension, with spatial periodic boundary conditions, whose numerical solutions preserved the neutral charges for all times, where we take a compatible discretization of the periodic Poisson problem that is mass preserving as performed in  \cite{MeG-JCP2018}, and so preserves the above mentioned charge neutrality condition for all time.

\begin{theorem}\label{stab-VP}
\emph{(Stability under the entropy norm for a time dependent Hamiltonian in the mean field limit):}
If $\Phi=\Phi(x,t)$, solution of the  boundary value problem \eqref{per-poisson} so the corresponding Hamiltonian is $H(x,p,t)=\varepsilon(p) -q\Phi(x,t)$,   then
 \begin{equation}\label{stab-VP.0}
\left| \left| f_h \right| \right|_{L^2_{e^H p^2}}^2 (t)
=
\int_{\Omega_C}  f_h^2(x,p,\mu,t) e^{H(x,p,t)} \, p^2 \, dp d\mu dx 
\leq 
\left| \left| f_h \right| \right|_{L^2_{e^H p^2}}^2 (0).
\end{equation}
\end{theorem}
\begin{proof}
Because of the divergence free structure of the Hamiltonian $H(x,p,t)=\varepsilon(p) -q\Phi(x,t)$ for every $t>0$, all estimates of Theorem~\ref{thm4.2} apply. In particular, starting from estimate 
\eqref{control_2}, we perform the time differentiation with respect to 
$H$
to obtain
\begin{align} 
0 \!\geq \!
&- \int_{e_h} \!\! \frac{(f_h^+ - f_h^-)^2}{2}  |\beta \cdot \hat{n}|  e^H d\sigma
\geq
\int_{\Omega_C}\!\!  \partial_t \left(  f_h^2 e^{H} \right) p^2  dp d\mu dx  
\!- \!
\int_{\Omega_C} f_h^2 e^{H} \partial_t \Phi p^2 dp d\mu dx, 
\label{stab-VP.1}
\end{align}
since  the only time dependent contribution comes from the potential part of the Hamiltonian, which is given in terms of the local mass density by means of the Poisson equation. 
Hence the solution of the  periodic boundary value problem in $x$-space for the Poisson problem \eqref{per-poisson} takes the form 
\vspace{-0.25cm}
\begin{align*} \nonumber
& \Phi(x,t) =   B(t) +
\end{align*} 
\begin{align*} \nonumber
&
\frac{q}{\epsilon} \cdot
\frac{x}{L} 
\int_0^L 
  \!\Big[\!N(x') \! -\! \rho_h(t,x') \Big] (L - x') dx' 
\!\int_0^x 
  \! \Big[N(x') \!-\! \rho_h(t,x')\Big] (x-x') dx',
\\
-\partial_x \Phi
&=
-\frac{q}{\epsilon}
\Bigg(
 \frac{1}{L}
\int_0^L   
  \Big[N(x') - \rho_h(t,x') \Big] (L-x') dx' - \int_0^x   \Big[N(x') - \rho_h(t,x') \Big] dx'
\Bigg),
\end{align*}
where $B(t)$ is the integrating parameter due to the periodic boundary conditions.

Uniqueness of solutions is obtained from imposing that  $\Phi(x,t)$ has zero average over the domain for all times $t$,  determining $B(t)$ by the following representation,
\begin{equation}
0 = \int_0^L \Phi(x,t) dx=
B(t) L =  \frac{1}{2} 
\frac{q}{\epsilon}
\int_0^L   
  \left[N(x') - \rho_h(t,x') \right] (x'-L) {x'} dx'\, ,
\end{equation}
and thus,  the potential is determined uniquely by 

\begin{align} 
\nonumber
\Phi(x,t) =
\frac{q}{\epsilon} 
\left(
\int_0^L   
  \left[ N - \rho_h \right](t,x') \frac{(x'-L)x'}{2L} dx'
+
\frac{x}{L}  \int_0^L   \left[N - \rho_h \right](t,x') (L - x') dx' \right.  \\
\left.  -   
\int_0^x 
  \left[N - \rho_h \right](t,x') (x - x') dx' \right).
\end{align}

Notice that the time dependent contribution of the Hamiltonian
yields
\begin{align*} \nonumber
\partial_t \Phi(x,t) 
  = 
\frac{q}{\epsilon}
\Bigg[
&\int_0^L   
\partial_t \rho_h (t,x') \frac{(L-x')x'}{2L} dx'\\
&\ \ +
\int_0^x 
\partial_t \rho_h(t,x')  (x - x') dx' 
-\frac{x}{L} \int_0^L 
 \partial_t \rho_h(t,x') (L - x') dx'\Bigg] ,
\end{align*}
as we have assumed that the doping $N(x)$ is independent of time $t$, which can be expressed more compactly as 
\begin{equation} \nonumber
\partial_t \Phi(x,t) 
  = 
\frac{q}{\epsilon}
\left[
\int_0^L   
\partial_t \rho_h (t,x') (L-x') 
\left(
\frac{x'}{2L} - \frac{x}{L}
\right)
dx'
+
\int_0^x 
\partial_t \rho_h(t,x')  (x - x') dx' 
\right].
\end{equation}

 Therefore, replacing this exact formula for $ \partial_t \Phi(x,t) $  into the inequality \eqref{stab-VP.1} yields a mean field, non-local, discrete entropy inequality

\begin{align*} \nonumber
0 \geq &\partial_t
\int_{\Omega_C}   f_h^2 e^{H}  p^2 dp d\mu dx   
- \frac{q}{\epsilon} 
\int_{\Omega_C} f_h^2 e^{H}
\Bigg[
\int_0^L   
\partial_t \rho_h (t,x') (L-x') 
\left(
\frac{x'}{2L} - \frac{x}{L}
\right)
dx' \\ 
&\qquad \qquad \qquad \qquad \qquad +
\int_0^x 
\partial_t \rho_h(t,x')  (x - x') dx' \, 
 \Bigg]  \, p^2 \, dp d\mu dx .
\end{align*}
To this end, one can substitute the partial time derivative of the density by the right hand side of a conservation equation that can be derived simply by integration of the Boltzmann Eq. over the momentum domain. Therefore, since in 1D we have
\begin{equation}
 \partial_t \rho_h(x,t) + \partial_x J_h(x,t) = 0,
\end{equation}
with $J_h(x,t) = \int_{\Omega_p} v(p) f_h(x,p,t) dp$, this yields
\begin{align*} \nonumber
0 \geq \partial_t
\int_{\Omega_C}   f_h^2 e^{H}  &p^2  dp d\mu dx  +  \frac{q}{\epsilon} 
\int_{\Omega_C} f_h^2 e^{H}
\Bigg[
\int_0^L   
\partial_x J (t,x') (L-x') 
\left(
\frac{x'}{2L} - \frac{x}{L}
\right)
dx' \\
&\qquad \qquad \qquad \qquad \qquad \ \ +
\int_0^x 
\partial_{x'} J_h(t,x')  (x - x') dx' 
 \Bigg]  p^2 dp d\mu dx.
\end{align*}

We proceed with an integration by parts to simplify our second term. So

\begin{align*} \nonumber
0 \geq \partial_t
\int_{\Omega_C}   f_h^2 e^{H}   &p^2 dp d\mu dx 
+\frac{q}{\epsilon} 
\int_{\Omega_C} f_h^2 e^{H}
\Bigg[   
J_h(t,0) x 
-
\int_0^L   
J_h(t,x') 
\frac{x - x' + L/2}{L}  
dx'\\
&\qquad \qquad \qquad \qquad \qquad-
J_h(t,0)  x 
+
\int_0^x 
J_h(t,x')   dx' \, 
\Bigg]  \, p^2  dp d\mu dx .
\end{align*}

This last  term reduces to
\begin{align*} \nonumber
0 \geq \partial_t
 \int_{\Omega_C}   f_h^2 e^{H}  &p^2 dp d\mu dx   
+ \frac{q}{\epsilon} 
\int_{\Omega_C} f_h^2 e^{H}
\Bigg[   
\int_0^x 
J_h(t,x')dx'  \\
&\qquad\qquad\qquad -
\!\int_0^L   
\!\!J_h(t,x') 
\Big(
\frac{x - x'}{L}  
+
\frac{1}{2}  
\Big)
dx'
\Bigg]  p^2  dp d\mu dx,
\end{align*}
which can be written, equivalently, as 
\begin{align*} \nonumber
0 \geq \partial_t
\int_{\Omega_C}   f_h^2 e^{H}   &p^2 dp d\mu dx  \, 
+ \frac{q}{\epsilon} 
\int_{\Omega_C} f_h^2 e^{H}
\Bigg[   
\int_0^x   
J_h(t,x') 
\Big(
\frac{1}{2}  
 -
\frac{x - x'}{L}  
\Big)
dx'
\\
&\qquad\qquad\qquad -
\int_x^L   
J_h(t,x') 
\Big(
\frac{1}{2}  
+
\frac{x - x'}{L}  
\Big)
dx'
\Bigg]  \, p^2 \, dp d\mu dx .
\end{align*}

The asymptotic and regular behaviour of our Boltzmann - Poisson problem as time approaches infinity in the  spatial domain given by the interval $[0,L]$ is well known \cite{BenA-Tay, Guo-2002, Guo-2003, Guo-Strain-2006}, in particular  convergence to a stationary state given by the balance of transport due to collisions, and so the corresponding current $J(x,t)$ stabilizes over the whole  interval domain to a constant value, i.e.
$\lim_{t \rightarrow \infty} J_h(x,t) = J_{{0}_h}$. 
Hence, by continuity,  for any given $ \delta > 0$, there exists a finite time $t_{\delta} >0$ such that
$|J_h(x,t) - J_0| < \delta \quad  \forall x \in [0,L], \, \forall t>t_\delta$. 
Therefore, since
\vspace{-0.15cm}
\begin{align*} \nonumber
& 0 \geq
\int_0^x \!  \!
(J_h \!-\! J_{0_{h}})
\Big(\!
\frac{1}{2}  \!-\!
\frac{x \!-\! x'}{L}  
\!\Big)
dx'
\!+\!
\int_x^L   \!\!
(J_h \!-\! J_{0_{h}})
\Big(\!
\frac{x' \!-\! x}{L}  
\!-\!
\frac{1}{2}  
\Big)
dx'
\Bigg]  p^2 dp d\mu dx
\end{align*}
\begin{align*}
&+ \partial_t
\int_{\Omega_C}   f_h^2 e^{H}  p^2 \, dp d\mu dx  
+ 
\frac{q}{\epsilon}  \int_{\Omega_C} f_h^2 e^{H} \Bigg[ \frac{J_0}{2L}
\Big[    ({L}/{2}  -x + x')^2
\Big|_0^x    
+
(x' - x - L/2)^2 \Big|_x^L   
\Big],  
\end{align*}
and $ ({L}/{2}  
-x + x')^2
|_0^x    
+
(x' - x - {L}/{2})^2 |_x^L   
= 0
$,
our equation reduces to
\begin{align*} \nonumber
0 &\geq \partial_t
\!\int_{\Omega_C}   f_h^2 e^{H}  \, p^2 \, dp d\mu dx+
 \frac{q}{\epsilon}\times 
\\ \nonumber
 & \!
\int_{\Omega_C} f_h^2 e^{H}
\Bigg[   
\int_0^x  \!\! 
(J_h\! -\! J_{0_{h}})
\left(
\frac{1}{2}  
\!-\!
\frac{x - x'}{L}  
\right)
dx'
\!+\!
\int_x^L \!\!  
(J_h \!-\! J_{0_{h}})
\left(
\frac{x' \!-\! x}{L}  
\!-\!
\frac{1}{2}  
\right)
dx'
\Bigg]  p^2 dp d\mu dx 
\end{align*}

\smallskip

Using our argument of convergence to a constant current as time goes to infinity,
\begin{align*} \nonumber
\Bigg|
&\int_{\Omega_C} \!\! f_h^2 e^{H}
\Bigg[   
\int_0^x   
(J_h \!-\! J_{0_{h}})
\big(
\frac{1}{2}  
-
\frac{x \!-\! x'}{L}  
\big)
dx'
+
\int_x^L   
(J_h \!-\! J_{0_{h}})
\big(
\frac{x' \!-\! x}{L}  
\!-\!
\frac{1}{2}  
\Big)
dx'
\Bigg]  \, p^2 \, dp d\mu dx 
 \Bigg| \\[4pt]
& \leq 
 \Bigg|
\int_{\Omega_C} f_h^2 e^{H}
\delta
\Bigg[   
x 
\Big\|
\frac{1}{2}  
-
\frac{x - x'}{L}  
\Big\|_{\infty,[0,x]}
+
(L-x)
\Big\| 
\frac{x' - x}{L}  
-
\frac{1}{2}  
\Big\|_{\infty,[x,L]}
\Big]  \, p^2 \, dp d\mu dx 
 \Bigg|
 \nonumber
\\[4pt]
&= \left|
\int_{\Omega_C} f_h^2 e^{H}
\delta
\left[   
x 
\frac{1}{2}  
+
(L-x)
\frac{1}{2}  
\right]  \, p^2 \, dp d\mu dx 
 \right| 
=
\frac{\delta L}{2}
\left|
\int_{\Omega_C} f_h^2 e^{H}
\, p^2 \, dp d\mu dx 
 \right| . 
\end{align*}


Finally, choosing $\delta > 0$ such that for $t_{\delta} > 0$ we have   
$$ \frac{\delta L}{2}
\left|
\int_{\Omega_C} f_h^2 e^{H}
\, p^2 \, dp d\mu dx 
 \right|  
<
\frac{1}{2}
\left|
\partial_t
\int_{\Omega_C}   f_h^2 e^{H}  \, p^2 \, dp d\mu dx
\right|
, 
\quad \mathrm{then}
$$
\vspace{-0.5cm}
\begin{equation} 
0 \geq \frac{d}{dt}
\int_{\Omega_C}   f_h^2 e^{H}  \, p^2 \, dp d\mu dx  
\quad \forall t>t_{\delta}. 
\end{equation}
\vspace{-0.25cm}
In particular, inequality \eqref{stab-VP.0} holds and  Theorem~\ref{stab-VP} statement holds. 
\end{proof}

\section{Error estimates for semi-discrete DG scheme with curvilinear momentum coordinates}
\label{section:4}

We state in this section our main results regarding error estimation for our DG scheme in curvilinear momentum coordinates at the semi-discrete stage. We present in detail the proofs of those results in an Appendix at the end of this document. 
The underlying equation (\ref{eq:underlyingKane})
\cite{CGMS-CMAME2008} is considered for the Theorem 
to be presented below. 

\begin{theorem}
$L^{2}$
 error estimate: Consider the semi-discrete DG solution 
$f_{h}$
 to the linear Boltzmann equation (under cut-off and inflow BC, using the upwind rule for the numerical fluxes)
for $f_h$ belonging to the trial space $V_h^k$ 
and $g_h$ in the test space $V^k_h$ of piecewise continuous polynomials of degree $k$, 
\[
(\partial_{t}f_{h},g_{h})_{\mathcal{T}_{h}}+\mathcal{A}(f_{h},g_{h})=\mathcal{L}(g_{h}),
\]
\[
(\partial_{t}f_{h},g_{h})_{\mathcal{T}_{h}}=\sum_{ikm}\int_{ikm}\partial_{t}f_{h}g_{h}e^{H}s(w)d\vec{w}d\vec{x},
\quad
s(w)=p^{2},\quad\vec{w}=(p,\mu),
\]
\[
\mathcal{A}(f_{h},g_{h})=
\]
\[
-\sum_{ikm}\int_{ikm}\partial_{p}\varepsilon(p)\,f_{h}\,\mu\,\partial_{x}(g_{h}e^{H})\,p^{2}dpd\mu dx\pm\sum_{km}'\int_{km}\partial_{p}\varepsilon\,\widehat{f_{h}\mu}|_{x_{i\pm}}\,g_{h}e^{H}|_{x_{i\pm}}^{\mp}\,p^{2}dpd\mu -
\]
\[
\sum_{ikm}\int_{ikm}p^{2}(-qE)(x,t)f_{h}\mu\,\partial_{p}(g_{h}e^{H})\,d\mu dx\pm\sum_{im}'\int_{im}p_{k^{\pm}}^{2}\,(-q\widehat{Ef_{h}\mu})|_{p_{k\pm}}g_{h}e^{H}|_{p_{k\pm}}^{\mp}d\mu dx
\]
\[
-\sum_{ikm}\int_{ikm}(1\!-\!\mu^{2})f_{h}(-qE)(x,t)\,\partial_{\mu}(g_{h}e^{H})pdpd\mu dx
\]
\[
\pm\sum_{ik}'\int_{ik}\!\!(1\!-\!\mu_{m\pm}^{2})(-q\widehat{Ef_{h}})|_{\mu_{m\pm}}\,g_{h}e^{H}|_{\mu_{m\pm}}^{\mp}pdpdx
\]
\[
-\sum_{ikm}\int_{\Omega_{C}}Q(f_{h})g_{h}e^{H}s(w)d\vec{w}d\vec{x},
\]
where the indices $ikm$ help identify the volume elements in which the domain is decomposed, and $km$, $im$, etc. index the boundaries of the aforementioned elements over the respective indexed variables ($i$ indexing intervals of the variable $x$, $k$ for intervals of the variable $p$, $m$ for intervals of the variable $\mu$), 
with $E(x,t)$ being a given time-dependent electric field
over the points $x$ in the position domain, 
 where the primed sums for the boundary integrals indicate that the terms
 related to the cut-off and inflow boundaries are excluded, and finally,
\[
\mathcal{L}(g_{h})=-\left\langle f^{in},g_{h}\beta\cdot\hat{n}\right\rangle _{\Gamma^{-}}
\]
denoting a surface integral over the inflow boundary.
 We have 
\[
||f_{h}(t,\cdot,\cdot)-f(t,\cdot,\cdot)||_{L^{2}(\Omega_{D})}\leq C\sqrt{t}e^{Cht}h^{k+1/2}|f|_{L^{\infty}([0,t],H^{k+1}(\Omega_{D}))},
\]
 with 
$C=C(diam(\Omega_{D}),||\beta||_{W^{\{1,\infty\}}(\Omega_{D})})$
 not depending on 
 $h$
 or 
 $t$.
\end{theorem}

We also have the following result. 

\begin{theorem}
If 
$f_{h}$
 is the semidiscrete DG solution to our Boltzmann equation with a linear
 collision operator in the semiconductor problem, then 
\[
||f_{h}(t,\cdot,\cdot)-M||\leq 
C\sqrt{t}e^{Cht}h^{k+1/2}|f|_{L^{\infty}([0,t],H^{k+1}(\Omega_{D}))}
+
3e^{-\lambda t}||f_{0}-M||_{B^{2}(\mathbb{R}^{d})}
\]
with the constant 
$C=C(diam(\Omega_{D}),||\beta||_{W^{\{1,\infty\}}(\Omega_{D})}).$
\end{theorem}

\section{Conclusions}
The work presented here relates to the
development of entropy stable 
and positivity preserving
DG schemes 
for BP models of collisional electron transport in semiconductors. Due to the physics of 
energy transitions given by Planck's law, and to reduce the dimension of the associated collision operator, given its mathematical form,
we pose the Boltzmann Equation for electron transport in
curvilinear coordinates for the momentum. This is a more general form that includes two previous BP models in different coordinate systems used in \cite{CGMS-CMAME2008} and \cite{MGCMSC-CMAME2017}
as particular cases. 
We consider first the 1D diode problem with azimuthal symmetry assumptions, which give us a 3D plus time problem. 
We choose for this problem the spherical coordinate system $\vec{p}(p,\mu,\varphi)$, slightly different to choices in previous works in the literature, because its DG formulation gives simpler integrals involving just piecewise polynomial functions for both transport and collision terms, which is convenient for Gaussian quadrature. 
We have been able to prove the full stability of the 
semi-discrete DG scheme formulated under an entropy norm and the decay of this norm over time for a 3D plus time problem (1D in position and 2D in momentum), assuming periodic boundary conditions for simplicity. 
This highlights the importance of the dissipative properties of our collisional operator given by its entropy inequalities. The entropy norm depends on the full time dependent Hamiltonian rather than just the Maxwellian associated solely to the kinetic energy. 
We prove another stability result for a 5D plus time problem (2D in position, 3D in momentum) considering in this case not only periodic but also specular reflection boundary conditions, where the integral associated to each reflecting boundary vanishes by itself due to specularity.  
Regarding positivity preserving DG schemes, using the strategy in \cite{ZhangShu1}, \cite{ZhangShu2}, \cite{CGP}, we treat the collision operator as a source term, and find convex combinations of the transport and collision terms which guarantee the preservation of positivity of the cell average of our numerical probability density function at the next time step. The positivity of the numerical solution to the pdf in the whole domain can be guaranteed by applying the limiters in \cite{ZhangShu1,ZhangShu2} that preserve the cell average modifying the slope of the piecewise linear solutions to make the function non-negative. 

\section*{Appendix 1}
We present in this appendix the detailed proofs of our statements in Section 3.
\begin{theorem}
$L^{2}$
 error estimate: Consider the semi-discrete DG solution 
$f_{h}$
 to the linear Boltzmann equation (under cut-off and inflow BC)
\[
(\partial_{t}f_{h},g_{h})_{\mathcal{T}_{h}}+\mathcal{A}(f_{h},g_{h})=\mathcal{L}(g_{h}),
\]
\[
(\partial_{t}f_{h},g_{h})_{\mathcal{T}_{h}}=\sum_{ikm}\int_{ikm}\partial_{t}f_{h}g_{h}e^{H}s(w)d\vec{w}d\vec{x},
\quad
s(w)=p^{2},\quad\vec{w}=(p,\mu),
\]
\[
\mathcal{A}(f_{h},g_{h})=-\sum_{ikm}\int_{ikm}\partial_{p}\varepsilon(p) f_{h} \mu \partial_{x}(g_{h}e^{H}) p^{2}dpd\mu dx
\]
\[
\pm
\sum_{km}'\int_{km}\partial_{p}\varepsilon
\widehat{f_{h}\mu}|_{x_{i\pm}} g_{h}e^{H}|_{x_{i\pm}}^{\mp} p^{2}dpd\mu +
\]
\[
\sum_{ikm}\int_{ikm}p^{2}qE(x,t)f_{h}\mu\,\partial_{p}(g_{h}e^{H})\,d\mu dx\pm\sum_{im}'\int_{im}p_{k^{\pm}}^{2}\,(-q\widehat{Ef_{h}\mu})|_{p_{k\pm}}g_{h}e^{H}|_{p_{k\pm}}^{\mp}d\mu dx +
\]
\[
\sum_{ikm}\int_{ikm}(1\!-\!\mu^{2})f_{h}qE(x,t)\,\partial_{\mu}(g_{h}e^{H})pdpd\mu dx
\]
\[
\mp
\sum_{ik}'\int_{ik}\!\!(1\!-\!\mu_{m\pm}^{2})q\widehat{Ef_{h}}|_{\mu_{m\pm}}\,g_{h}e^{H}|_{\mu_{m\pm}}^{\mp}pdpdx
\]
\[
-\sum_{ikm}\int_{\Omega_{C}}Q(f_{h})g_{h}e^{H}s(w)d\vec{w}d\vec{x},
\]
 where the primed sums for the boundary integrals indicate that the terms
 related to the cut-off and inflow boundaries are excluded, and finally,
\[
\mathcal{L}(g_{h})=-\left\langle f^{in},g_{h}\beta\cdot\hat{n}\right\rangle _{\Gamma^{-}}
\]
denoting a surface integral over the inflow boundary.
 We have 
\[
||f_{h}(t,\cdot,\cdot)-f(t,\cdot,\cdot)||_{L^{2}(\Omega_{D})}\leq C\sqrt{t}e^{Cht}h^{k+1/2}|f|_{L^{\infty}([0,t],H^{k+1}(\Omega_{D}))},
\]
 with 
$C=C(diam(\Omega_{D}),||\beta||_{W^{\{1,\infty\}}(\Omega_{D})})$
 not depending on 
 $h$
 or 
 $t$.
\end{theorem}
\begin{proof}
The classical (exact) solution 
\ $f$
 satisfies the weak formulation for the solution of the DG scheme.
 Therefore it also holds that 
\ 
\[
(\partial_{t}f,g_{h})_{\mathcal{T}_{h}}+\mathcal{A}(f,g_{h})=\mathcal{L}(g_{h})
\]
for all test functions 
$g_{h}.$
 The error is naturally defined as 
\ $\mathbb{E=}f-f_{h},$
 so we have
\[
(\partial_{t}\mathbb{E},g_{h})_{\mathcal{T}_{h}}+\mathcal{A}(\mathbb{E},g_{h})=0
\]
 by linearity of the operator 
 $\mathcal{A}$.
 Now, the error will be decomposed in two parts, 
\[
\mathbb{E=\mathcal{E}+}E_{h},
\]
 with the first one related to the error in 
 $L^{2}-$
projecting 
 $f$
 in the DG-FEM space,
\[
\mathcal{E}=f-\mathbb{P}f,
\]
 and the second part being the difference between this projection and the numerical
 solution to the DG scheme, 
\[
E_{h}=\mathbb{P}f-f_{h},
\]
 the latter error contribution belonging to the FEM space.
 Therefore we can choose 
 $g_{h}=E_{h},$
 and then, 
\[
(\partial_{t}\mathbb{E},E_{h})_{\mathcal{T}_{h}}+\mathcal{A}(\mathbb{E},E_{h})=0,
\]
 so the remainder equation is 
\[
(\partial_{t}\mathbb{\mathcal{E}},E_{h})_{\mathcal{T}_{h}}+\mathcal{A}(\mathbb{\mathcal{E}},E_{h})+(\partial_{t}E_{h},E_{h})_{\mathcal{T}_{h}}+\mathcal{A}(E_{h},E_{h})=0.
\]

We have that 
 $(\partial_{t}\mathcal{E},E_{h})_{\mathcal{T}_{h}}=0.$
 This happens because 
 $\mathcal{E}=f-\mathbb{P}f$
 is the 
 $L^{2}-$
projection of 
 $f$
 into the FEM space, and by definition 
 $\mathbb{P}f$
 is the function in the FEM space whose 
 $L^{2}-$
Fourier
 coefficients when represented in a basis for the FEM space are the same
 as the 
 $L^{2}-$
inner products between 
$f$
 and the basis elements.
 Namely, 
\[
(f,w_{h})_{\mathcal{T}_{h}}=(\mathbb{P}f,w_{h})_{\mathcal{T}_{h}}
\]
 as the equality above holds for all elements in the basis set and therefore
 for any 
 $w_{h}$
 in the DG-FEM space.
 Therefore
we get the equality we wanted to prove,
\[
(\mathcal{E},w_{h})_{\mathcal{T}_{h}}=0.
\]

 Now we consider the last term in the remainder equation, 
\[
\mathcal{A}(E_{h},E_{h})=\frac{1}{4}\int_{e_{h}}(E_{h}^{+}-E_{h}^{-})^{2}|\beta\cdot\widehat{n}|e^{H}d\sigma-\int_{\Omega_{C}}Q(E_{h})E_{h}e^{H}s(w)d\vec{w}d\vec{x}
\]
which, if we use in the remainder equation, will give us 
\[
\mathcal{A}(\mathbb{\mathcal{E}},E_{h})+(\partial_{t}E_{h},E_{h})_{\mathcal{T}_{h}}+\frac{\int_{e_{h}}(E_{h}^{+}-E_{h}^{-})^{2}|\beta\cdot\widehat{n}|e^{H}d\sigma}{4}=\]
\[
\int_{\Omega_{C}}Q(E_{h})E_{h}e^{H}s(w)d\vec{w}d\vec{x}\leq0 
\]
 given the entropy inequality for our collisional operator.
 So 
\[
(\partial_{t}E_{h},E_{h})_{\mathcal{T}_{h}}+\frac{1}{4}\int_{e_{h}}(E_{h}^{+}-E_{h}^{-})^{2}|\beta\cdot\widehat{n}|e^{H}d\sigma\leq-\mathcal{A}(\mathbb{\mathcal{E}},E_{h}).
\]

 We will study the bound in this inequality.
 We have 
\ 
\[
\mathcal{A}(\mathcal{E},E_{h})=
\]
\[
-\sum_{ikm}\int_{ikm}\partial_{p}\varepsilon(p)\,\mathcal{E}\,\mu\,\partial_{x}(E_{h}e^{H})\,p^{2}dpd\mu dx\pm\sum_{km}'\int_{km}\partial_{p}\varepsilon\,\widehat{\mathcal{E}\mu}|_{x_{i\pm}}\,E_{h}e^{H}|_{x_{i\pm}}^{\mp}\,p^{2}dpd\mu
\]

\[
-\sum_{ikm}\int_{ikm}p^{2}(-qE)(x,t)\mathcal{E}\mu\,\partial_{p}(E_{h}e^{H})\,d\mu dx\pm\sum_{im}'\int_{im}p_{k^{\pm}}^{2}\,(-q\widehat{E\mathcal{E}\mu})|_{p_{k\pm}}E_{h}e^{H}|_{p_{k\pm}}^{\mp}d\mu dx
\]

\[
-\sum_{ikm}\int_{ikm}(1\!-\!\mu^{2})\mathcal{E}(-qE)(x,t)\,\partial_{\mu}(E_{h}e^{H})pdpd\mu dx
\pm
\]
\[
\sum_{ik}'\int_{ik}\!\!(1\!-\!\mu_{m\pm}^{2})(-q\widehat{E\mathcal{E}})|_{\mu_{m\pm}}\,E_{h}e^{H}|_{\mu_{m\pm}}^{\mp}pdpdx
\]

\[
-\sum_{ikm}\int_{\Omega_{C}}Q(\mathcal{E})E_{h}e^{H}s(w)d\vec{w}d\vec{x},
\]

 which we will decompose into three terms, 

\[
T_{1}=-\sum_{ikm}\int_{ikm}(1\!-\!\mu^{2})\mathcal{E}(-qE)(x,t)\,\partial_{\mu}(E_{h}e^{H})pdpd\mu dx
\]

\[
-\sum_{ikm}\int_{ikm}\partial_{p}\varepsilon(p)\,\mathcal{E}\,\mu\,\partial_{x}(E_{h}e^{H})\,p^{2}dpd\mu dx
-\sum_{ikm}\int_{ikm}p^{2}(-qE)(x,t)\mathcal{E}\mu\,\partial_{p}(E_{h}e^{H})\,d\mu dx,
\]

\[
T_{2}=\pm\sum_{km}'\int_{km}\partial_{p}\varepsilon\,\widehat{\mathcal{E}\mu}|_{x_{i\pm}}\,E_{h}e^{H}|_{x_{i\pm}}^{\mp}\,p^{2}dpd\mu
\pm\sum_{im}'\int_{im}p_{k^{\pm}}^{2}\,(-q\widehat{E\mathcal{E}\mu})|_{p_{k\pm}}E_{h}e^{H}|_{p_{k\pm}}^{\mp}d\mu dx
\]

\[
\pm\sum_{ik}'\int_{ik}\!\!(1\!-\!\mu_{m\pm}^{2})(-q\widehat{E\mathcal{E}})|_{\mu_{m\pm}}\,E_{h}e^{H}|_{\mu_{m\pm}}^{\mp}pdpdx,
\]

\[
T_{3}=-\sum_{ikm}\int_{\Omega_{C}}Q(\mathcal{E})E_{h}e^{H}s(w)d\vec{w}d\vec{x},
\]

 so 
\[
\mathcal{-A}(\mathcal{E},E_{h})\leq|\mathcal{A}(\mathcal{E},E_{h})|\leq|T_{1}|+|T_{2}|+|T_{3}|,
\]

 and we will bound the terms individually.
 For 
\ $T_{1},$
 we first point out that the Hamiltonian is related to the transport vector
 since (using 
 $K_{B}T$
 units)
\[
H(x,p,t)=\varepsilon(p)-qV(x,t),\quad\partial_{\mu}H=0,
\]
 where the potential gives the electric field by 
$
 E(x,t)=-\partial_{x}V(x,t)
$.
So
\[
T_{1}=-\sum_{ikm}\int_{ikm}\partial_{p}\varepsilon(p)\,\mathcal{E}\,\mu\,\partial_{x}(E_{h}e^{H})\,p^{2}dpd\mu dx
\]
\[
-\sum_{ikm}\int_{ikm}p^{2}(-qE)(x,t)\mathcal{E}\mu\,\partial_{p}(E_{h}e^{H})\,d\mu dx
\]
\[
-\sum_{ikm}\int_{ikm}(1\!-\!\mu^{2})\mathcal{E}(-qE)(x,t)\,e^{H}\partial_{\mu}E_{h}pdpd\mu dx=
\]
\[
-\sum_{ikm}\int_{ikm}\partial_{p}\varepsilon(p)\,\mathcal{E}\,\mu\,(\partial_{x}E_{h}e^{H}+E_{h}\partial_{x}e^{H})\,p^{2}dpd\mu dx
\]
\[
-\sum_{ikm}\int_{ikm}p^{2}(-qE)(x,t)\mathcal{E}\mu\,(\partial_{p}E_{h}e^{H}+E_{h}\partial_{p}e^{H})\,d\mu dx +
\]
\[
\sum_{ikm}\int_{ikm}(1\!-\!\mu^{2})\mathcal{E}qE(x,t)\,e^{H}\partial_{\mu}E_{h}pdpd\mu dx
=
\]
\[
\sum_{ikm}\int_{ikm}\partial_{p}\varepsilon(p)\,\mathcal{E}\,\mu\,(qE_{h}\partial_{x}V-\partial_{x}E_{h})e^{H}\,p^{2}dpd\mu dx
\]

\[
-\sum_{ikm}\int_{ikm}p^{2}(-qE)(x,t)\mathcal{E}\mu\,(\partial_{p}E_{h}+E_{h}\partial_{p}\varepsilon(p))e^{H}\,d\mu dx
\]

\[
-\sum_{ikm}\int_{ikm}(1\!-\!\mu^{2})\mathcal{E}(-qE)(x,t)\,e^{H}\partial_{\mu}E_{h}pdpd\mu dx=
\]

\[
=-\sum_{ikm}\int_{ikm}e^{H}\partial_{p}\varepsilon(p)\,\mathcal{E}\,\mu\,\partial_{x}E_{h}\,p^{2}dpd\mu dx
+\sum_{ikm}\int_{ikm}e^{H}qE(x,t)\mathcal{E}\mu\,\partial_{p}E_{h}p^{2}\,d\mu dx
\]

\[
+\sum_{ikm}\int_{ikm}(1\!-\!\mu^{2})\mathcal{E}qE(x,t)\,e^{H}\partial_{\mu}E_{h}pdpd\mu dx=
\]

\[
=\sum_{ikm}\int_{ikm}e^{H}\mathcal{E}[-\partial_{p}\varepsilon(p)\,\,\mu\,\partial_{x}E_{h}+qE(x,t)\mu\,\partial_{p}E_{h}+(1\!-\!\mu^{2})qE(x,t)\,\partial_{\mu}E_{h}/p]\,p^{2}dpd\mu dx
\]

\[
=\sum_{ikm}\int_{ikm}\mathcal{E}(-\partial_{p}\varepsilon\mu,qE\mu,qE\frac{(1\!-\!\mu^{2})}{p})\cdot\partial E_{h}\,e^{H}p^{2}dpd\mu dx,
\]

 where 
 $\partial=\partial_{(x,p,\mu)}$
 is understood.
 So we have expressed 
 $T_{1}$
 in the form 
\[
T_{1}=\sum_{ikm}(\mathcal{E},\beta\cdot\partial E_{h})_{ikm},
\]
 with the transport vector 
 $\beta$
 defined by 
\[
\beta=(-\partial_{p}\varepsilon\mu,qE\mu,qE
{(1\!-\!\mu^{2})}/{p}).
\]

This vector depends on 
 $(x,p,\mu),$
 so we proceed to take an average of it over the elements, such that its
 projection over each one of the normals 
 $\widehat{n}=\widehat{e}_{x},\widehat{e}_{p},\widehat{e}_{\mu}$
 at the boundaries satisfies
\[
<\beta^{0}\cdot\widehat{n},1>_{\partial K}=<\beta\cdot\widehat{n},1>_{\partial K},
\]

\[
\beta=(-\partial_{p}\varepsilon\,\mu,qE\mu,qE(1\!-\!\mu_{m}^{2})|_{\mu_{m-}}^{\mu_{m+}}/p),
\]
so this constant vector is defined by 
\[
\beta_{ikm}^{0}=\frac{<-\partial_{p}\varepsilon\mu,1>_{km}}{<1,1>_{km}}\widehat{e}_{x}+\frac{<qE\mu,1>_{im}}{<1,1>_{im}}\widehat{e}_{p}+\frac{<qE(1\!-\!\mu^{2})|_{\mu_{m-}}^{\mu_{m+}}/p,1>_{ik}}{<1,1>_{ik}}\widehat{e}_{\mu}.
\]

 By definition, the vector 
\[
\beta^{0}=\sum_{ikm}\chi_{ikm}\beta_{ikm}^{0}
\]
is such that 
\ $\beta^{0}\cdot\partial E_{h}$
 is in the DG-FEM space, which implies that 
\[
\sum_{ikm}(\mathcal{E},\beta^{0}\cdot\partial E_{h})_{ikm}=(\mathcal{E},\beta^{0}\cdot\partial E_{h})_{\mathcal{T}_{h}}=0,
\]

 and therefore 
\[
T_{1}=\sum_{ikm}(\mathcal{E},\beta\cdot\partial E_{h})_{ikm}=\sum_{ikm}(\mathcal{E},(\beta-\beta^{0})\cdot\partial E_{h})_{ikm}.
\]

Now we can bound this term as 
\[
|T_{1}|\leq\sum_{ikm}||\mathcal{E},(\beta-\beta^{0})\cdot\partial E_{h}||_{ikm}\leq\sum_{ikm}||\mathcal{E}||_{L_{ikm}^{2}}||\beta-\beta^{0}||_{L_{ikm}^{\infty}}||\partial E_{h}||_{L_{ikm}^{2}}.
\]

 We will now bound the last factor by making use of the inverse inequality,
 which states that there's a constant 
 $C_{ikm}$
 s.t.
 
\[
||\partial E_{h}||_{L_{ikm}^{2}}\leq C_{ikm}||E_{h}||_{L_{ikm}^{2}}/h_{ikm},
\]

 which we use now to get
\[
|T_{1}|\leq\sum_{ikm}||\mathcal{E}||_{L_{ikm}^{2}}||E_{h}||_{L_{ikm}^{2}}C_{ikm}||\beta-\beta^{0}||_{L_{ikm}^{\infty}}/h_{ikm}
\]
\[
\leq\{\sum_{ikm}||\mathcal{E}||_{L_{ikm}^{2}}||E_{h}||_{L_{ikm}^{2}}\}\{\max_{ikm}\quad C_{ikm}||\beta-\beta^{0}||_{L_{ikm}^{\infty}}/h_{ikm}\},
\]

\[
|T_{1}|\leq||\mathcal{E}||_{L_{\mathcal{T}_{h}}^{2}}||E_{h}||_{L_{\mathcal{T}_{h}}^{2}}\max_{ikm}\quad\{C_{ikm}||\beta-\beta^{0}||_{L_{ikm}^{\infty}}/h_{ikm}\},
\]

\[
|T_{1}|\leq||\mathcal{E}||_{L_{\mathcal{T}_{h}}^{2}}||E_{h}||_{L_{\mathcal{T}_{h}}^{2}}|\beta|_{W^{\{1,\infty\}}(\mathcal{T}_{h})},
\]

 using the results in \cite{Guzman}  for the last inequality.
 We proceed now to use a known estimate for 
 $L^{2}-$
projections, 
\[
||\mathcal{E}||_{L_{ikm}^{2}}+h_{ikm}^{1/2}||\mathcal{E}||_{L_{ikm}^{2}}+h_{ikm}||\partial\mathcal{E}||_{L_{ikm}^{2}}\leq Ch_{ikm}^{k+1}|f|_{H_{ikm}^{k+1}},
\]
therefore the $1^{\mathrm{st}}$ summand can be bounded over the triangulation and
 we'll have 
\[
|T_{1}|\leq Ch^{k+1}|f|_{H_{(\mathcal{T}_{h})}^{k+1}}||E_{h}||_{L_{\mathcal{T}_{h}}^{2}}|\beta|_{W^{\{1,\infty\}}(\mathcal{T}_{h})}.
\]
Now we'll proceed with the estimate of 
\ $T_{2},$
\[
T_{2}=\pm\sum_{km}'\int_{km}\partial_{p}\varepsilon\,\widehat{\mathcal{E}\mu}|_{x_{i\pm}}\,E_{h}e^{H}|_{x_{i\pm}}^{\mp}\,p^{2}dpd\mu
\pm\sum_{im}'\int_{im}p_{k^{\pm}}^{2}\,(-q\widehat{E\mathcal{E}\mu})|_{p_{k\pm}}E_{h}e^{H}|_{p_{k\pm}}^{\mp}d\mu dx
\]

\[
\pm\sum_{ik}'\int_{ik}\!\!\frac{(1\!-\!\mu_{m\pm}^{2})}{p}(-q\widehat{E\mathcal{E}})|_{\mu_{m\pm}}\,E_{h}e^{H}|_{\mu_{m\pm}}^{\mp}p^{2}dpdx,
\]

 which if we consider added over all the triangulation we have 
\[
T_{2}=-<E_{h}e^{H},(-\partial_{p}\varepsilon\widehat{\mathcal{E}\mu},q\widehat{E\mathcal{E}\mu},\frac{(1\!-\!\mu_{m\pm}^{2})}{p}q\widehat{E\mathcal{E}})\cdot\hat{n}\,>_{\partial\mathcal{T}_{h}},
\]

 and, if we remember that 
\ $\beta=(-\partial_{p}\varepsilon\,\mu,qE\mu,qE(1\!-\!\mu_{m}^{2})|_{\mu_{m-}}^{\mu_{m+}}/p)$,
 then 
\[
T_{2}=-<E_{h}e^{H},\widehat{\beta\mathcal{E}}\cdot\hat{n}\,>_{\partial\mathcal{T}_{h}},
\]

 therefore, using  DG-FEM results related to volume and boundary
 integrals,
\[
|T_{2}|=|<E_{h}e^{H},\widehat{\beta\mathcal{E}}\cdot\hat{n}\,>_{\partial\mathcal{T}_{h}}|=|\frac{1}{2}<e^{H}E_{h}|_{+}^{-},|\widehat{\beta\mathcal{E}}\cdot\widehat{n}|\,>_{e_{h}}|,
\]

 where the notation 
 $e_{h}$
 indicates edge redundancy in counting, and therefore the factor of 1/2,
 and proceeding further,
\[
|T_{2}|\leq\frac{1}{8}||[e^{H}E_{h}]\sqrt{|\widehat{\mathcal{\beta}}\cdot\widehat{n}|}\,||^{2}{}_{L^{2}(e_{h})}+\frac{1}{2}||\mathcal{E^{-}}\sqrt{|\widehat{\beta}\cdot\widehat{n}|}\,||^{2}{}_{L^{2}(e_{h})},
\]

 and bounding the second term, we have
\[
|T_{2}|\leq\frac{1}{8}||[e^{H}E_{h}]\sqrt{|\widehat{\mathcal{\beta}}\cdot\widehat{n}|}\,||^{2}{}_{L^{2}(e_{h})}+C||\beta||_{L_{\mathcal{T}_{h}}^{\infty}}||\mathcal{E}||^{2}{}_{L^{2}(e_{h})},
\]

\[
|T_{2}|\leq\frac{1}{8}||[e^{H}E_{h}]\sqrt{|\widehat{\mathcal{\beta}}\cdot\widehat{n}|}\,||^{2}{}_{L^{2}(e_{h})}+C||\beta||_{L_{\mathcal{T}_{h}}^{\infty}}h^{2k+1}|f|_{H^{k+1}(\mathcal{T}_{h})}^{2}.
\]

 The only difference in this case is the appearance of the exponential of
 the Hamiltonian in one of the bounding terms.
 Finally, we proceed with
 $T_{3}.$
 We have 
\[
T_{3}=-\sum_{ikm}\int_{\Omega_{C}}Q(\mathcal{E})E_{h}e^{H}s(w)d\vec{w}d\vec{x}=-(Q(\mathcal{E}),E_{h})_{\mathcal{T}_{h}},
\]

by definition of the inner product under the entropy norm
\[
(g_{h},f_{h})=\int_{\Omega_{C}}g_{h}f_{h}e^{H}s(w)d\vec{w}d\vec{x},
\]

where 
 $E_{h}$
 belongs to the DG-FEM space.
 So 
\[
|T_{3}|=|\sum_{ikm}\int_{\Omega_{C}}Q(\mathcal{E})E_{h}e^{H}s(w)d\vec{w}d\vec{x}|=|(Q(\mathcal{E}),E_{h})_{\mathcal{T}_{h}}|\leq||Q(\mathcal{E})||{}_{L^{2}(\mathcal{T}_{h})}\cdot||E_{h}||_{L^{2}(\mathcal{T}_{h})}
\]
 by Cauchy-Schwarz. We only need to study the 
 $L^{2}$-norm of 
\ $Q(\mathcal{E}).$
We have
\[
||Q(\mathcal{E})||^{2}{}_{L^{2}(\mathcal{T}_{h})}=(Q(\mathcal{E}),Q(\mathcal{E}))_{\mathcal{T}_{h}}=\sum_{ikm}\int_{\Omega_{C}}Q(\mathcal{E})Q(\mathcal{E})e^{H}s(w)d\vec{w}d\vec{x},
\]
 but we know our collision operator has the following structure, 
\[
\sum_{ikm}\int_{\Omega_{C}}Q(f)ge^{H}s(w)d\vec{w}d\vec{x}=
\]
\[
-\frac{1}{2}\sum_{ikm}\int_{\Omega_{C}}S(\vec{p}'\rightarrow\vec{p})e^{-H'}(\frac{f'}{e^{-H'}}-\frac{f}{e^{-H}})(g'-g)d\vec{p}d\vec{p}'d\vec{x},
\]
 so if we apply it to our particular case, we have 
\[
\sum_{ikm}\int_{\Omega_{C}}Q(\mathcal{E})Q(\mathcal{E})e^{H}s(w)d\vec{w}d\vec{x}=
\]
\[
-\frac{1}{2}\sum_{ikm}\int_{\Omega_{C}}S(\vec{p}'\rightarrow\vec{p})e^{-H'}(\frac{\mathcal{E}'}{e^{-H'}}-\frac{\mathcal{E}}{e^{-H}})(Q(\mathcal{E}')-Q(\mathcal{E}))d\vec{p}d\vec{p}'d\vec{x}.
\]

We therefore conclude that
\[
||Q(\mathcal{E})||^{2}{}_{L^{2}(\mathcal{T}_{h})}=(Q(\mathcal{E}),Q(\mathcal{E}))_{\mathcal{T}_{h}}=
\]
\[
\sum_{ikm}\int_{\Omega_{C}}|Q(\mathcal{E})|^{2}e^{H}s(w)d\vec{w}d\vec{x}\leq
\sum_{ikm}\int_{\Omega_{C}}\{\sum_{j=-1}^{+1}K_{j}
e^{H}s(w)d\vec{w}d\vec{x}
[\frac{\partial\vec{p}}{\partial(\varepsilon,\mu,\phi)}\chi](\varepsilon+j\hbar\omega_{p})
\times 
\]
\[
4\pi[e^{\frac{j\hbar\omega_{p}}{2K_{B}T_{L}}}||\mathcal{E}(\varepsilon'+j\hbar\omega_{p},\mu',\phi')||_{L^{\infty}(\varepsilon',\mu',\phi')}+e^{\frac{-j\hbar\omega_{p}}{2K_{B}T_{L}}}|\mathcal{E}|]\}^{2}
\leq
(4\pi)^{2} \times
\]
\[
\sum_{ikm}\int_{\Omega_{C}}\{\sum_{j=-1}^{+1}K_{j}[J\chi](\varepsilon+j\hbar\omega_{p})[e^{\frac{j\hbar\omega_{p}}{2K_{B}T_{L}}}||\mathcal{E}||_{L^{\infty}(\varepsilon',\mu',\phi')}+e^{\frac{-j\hbar\omega_{p}}{2K_{B}T_{L}}}|\mathcal{E}|]\}^{2}e^{H}s(w)d\vec{w}d\vec{x}
\]
\[
=
(4\pi)^{2}\sum_{ikm}\int_{\Omega_{C}}\{K_{0}[J\chi](\varepsilon)[||\mathcal{E}||_{L^{\infty}(\varepsilon',\mu',\phi')}+|\mathcal{E}|]+
\]
\[
+\sum_{\pm}K_{\pm}[J\chi](\varepsilon\pm\hbar\omega_{p})[e^{\frac{\pm\hbar\omega_{p}}{2K_{B}T_{L}}}||\mathcal{E}||_{L^{\infty}(\varepsilon',\mu',\phi')}+e^{\frac{\mp\hbar\omega_{p}}{2K_{B}T_{L}}}|\mathcal{E}|]\}^{2}e^{H}s(w)d\vec{w}d\vec{x},
\]
but since 
\[
K_{1}=K_{-1}=Kn_{q}e^{\frac{\hbar\omega_{p}}{2K_{B}T_{L}}},
\quad
n_{q}=(e^{\hbar\omega_{p}/K_{B}T_{L}}-1)^{-1},
\]
\[
K_{\pm1}=K\frac{e^{\frac{\hbar\omega_{p}}{2K_{B}T_{L}}}}{e^{\hbar\omega_{p}/K_{B}T_{L}}-1}=\frac{K}{e^{\frac{\hbar\omega_{p}}{2K_{B}T_{L}}}-e^{\frac{-\hbar\omega_{p}}{2K_{B}T_{L}}}}=\frac{K/2}{\sinh(\frac{\hbar\omega_{p}}{2K_{B}T_{L}})},
\]
then 
\[
||Q(\mathcal{E})||^{2}{}_{L^{2}(\mathcal{T}_{h})}\leq
(4\pi)^{2}\sum_{ikm}\int_{\Omega_{C}}\{K_{0}[J\chi](\varepsilon)[||\mathcal{E}||_{L^{\infty}(\varepsilon',\mu',\phi')}+|\mathcal{E}|]+
\]
\[
+K\sum_{\pm}[J\chi](\varepsilon\pm\hbar\omega_{p})[\frac{e^{\frac{\pm\hbar\omega_{p}}{2K_{B}T_{L}}}||\mathcal{E}||_{L^{\infty}(\varepsilon',\mu',\phi')}}{e^{\frac{\hbar\omega_{p}}{2K_{B}T_{L}}}-e^{\frac{-\hbar\omega_{p}}{2K_{B}T_{L}}}}+\frac{e^{\frac{\mp\hbar\omega_{p}}{2K_{B}T_{L}}}|\mathcal{E}|}{e^{\frac{\hbar\omega_{p}}{2K_{B}T_{L}}}-e^{\frac{-\hbar\omega_{p}}{2K_{B}T_{L}}}}]\}^{2}e^{H}s(w)d\vec{w}d\vec{x}=
\]
\[
(4\pi)^{2}\sum_{ikm}\int_{\Omega_{C}}\{K_{0}[J\chi](\varepsilon)[||\mathcal{E}||_{L^{\infty}(\varepsilon',\mu',\phi')}+|\mathcal{E}|]+
\]
\[
+K\sum_{\pm}[J\chi](\varepsilon\pm\hbar\omega_{p})[\frac{\pm||\mathcal{E}||_{L^{\infty}(\varepsilon',\mu',\phi')}}{1-e^{\frac{-\hbar\omega_{p}}{K_{B}T_{L}}}}\pm n_{q}|\mathcal{E}|]\}^{2}e^{H}s(w)d\vec{w}d\vec{x}=
\]
\[
(4\pi)^{2}\sum_{ikm}\int_{\Omega_{C}}\{K_{0}[J\chi](\varepsilon)[||\mathcal{E}||_{L^{\infty}(\varepsilon',\mu',\phi')}+
\]
\[
+|\mathcal{E}|]+K\sum_{\pm}[J\chi](\varepsilon\pm\hbar\omega_{p})[\pm e^{\frac{\hbar\omega_{p}}{K_{B}T_{L}}}n_{q}||\mathcal{E}||_{L^{\infty}(\varepsilon',\mu',\phi')}\pm n_{q}|\mathcal{E}|]\}^{2}e^{H}s(w)d\vec{w}d\vec{x}=
\]
\[
(4\pi)^{2}\sum_{ikm}\int_{\Omega_{C}}\{K_{0}[J\chi](\varepsilon)[||\mathcal{E}||_{L^{\infty}(\varepsilon',\mu',\phi')}+|\mathcal{E}|]+
\]
\[
+Kn_{q}\sum_{\pm}\pm[J\chi](\varepsilon\pm\hbar\omega_{p})[e^{\frac{\hbar\omega_{p}}{K_{B}T_{L}}}||\mathcal{E}||_{L^{\infty}(\varepsilon',\mu',\phi')}+|\mathcal{E}|]\}^{2}e^{H}s(w)d\vec{w}d\vec{x}=
\]
\[
(4\pi)^{2}\sum_{ikm}\int_{\Omega_{C}}\{K_{0}[J\chi](\varepsilon)[||\mathcal{E}||_{L^{\infty}(\varepsilon',\mu',\phi')}+|\mathcal{E}|]
\]
\[
+Kn_{q}([J\chi](\varepsilon+\hbar\omega_{p})-[J\chi](\varepsilon-\hbar\omega_{p}))[e^{\frac{\hbar\omega_{p}}{K_{B}T_{L}}}||\mathcal{E}||_{L^{\infty}(\varepsilon',\mu',\phi')}+|\mathcal{E}|]\}^{2}e^{H}s(w)d\vec{w}d\vec{x},
\]
 and using 
\[
K_{1}=K_{-1}=Kn_{q}e^{\frac{\hbar\omega_{p}}{2K_{B}T_{L}}},
\]
 then 
\[
||Q(\mathcal{E})||^{2}{}_{L^{2}(\mathcal{T}_{h})}\leq
(4\pi)^{2}\sum_{ikm}\int_{\Omega_{C}}\{K_{0}[J\chi](\varepsilon)[||\mathcal{E}||_{L^{\infty}(\varepsilon',\mu',\phi')}+|\mathcal{E}|]+
\]
\[
K_{1}e^{\frac{-\hbar\omega_{p}}{2K_{B}T_{L}}}([J\chi](\varepsilon+\hbar\omega_{p})-[J\chi](\varepsilon-\hbar\omega_{p}))[e^{\frac{\hbar\omega_{p}}{K_{B}T_{L}}}||\mathcal{E}||_{L^{\infty}(\varepsilon',\mu',\phi')}+|\mathcal{E}|]\}^{2}e^{H}s(w)d\vec{w}d\vec{x}=
\]
\[
(4\pi)^{2}\sum_{ikm}\int_{\Omega_{C}}\{K_{0}[J\chi](\varepsilon)[||\mathcal{E}||_{L^{\infty}(\varepsilon',\mu',\phi')}+|\mathcal{E}|]+
\]
\[
K_{1}([J\chi](\varepsilon+\hbar\omega_{p})-[J\chi](\varepsilon-\hbar\omega_{p}))[e^{\frac{\hbar\omega_{p}}{2K_{B}T_{L}}}||\mathcal{E}||_{L^{\infty}(\varepsilon',\mu',\phi')}+e^{\frac{-\hbar\omega_{p}}{2K_{B}T_{L}}}|\mathcal{E}|]\}^{2}e^{H}s(w)d\vec{w}d\vec{x}=
\]
\[
(4\pi)^{2}(\sum_{ikm}\int_{\Omega_{C}}\{K_{0}[J\chi](\varepsilon)[||\mathcal{E}||_{L^{\infty}(\varepsilon',\mu',\phi')}+|\mathcal{E}|]\}^{2}e^{H}s(w)d\vec{w}d\vec{x}+
\sum_{ikm}\int_{\Omega_{C}} e^{H}s(w)d\vec{w}d\vec{x} \times
\]
\[
\{K_{1}([J\chi](\varepsilon+\hbar\omega_{p})-[J\chi](\varepsilon-\hbar\omega_{p}))[e^{\frac{\hbar\omega_{p}}{2K_{B}T_{L}}}||\mathcal{E}||_{L^{\infty}(\varepsilon',\mu',\phi')}+e^{\frac{-\hbar\omega_{p}}{2K_{B}T_{L}}}|\mathcal{E}|]\}^{2}
\]
\[
+
\sum_{ikm}\int_{\Omega_{C}}2K_{0}K_{1}[J\chi](\varepsilon)[||\mathcal{E}||_{L^{\infty}(\varepsilon',\mu',\phi')}+
\]
\[
|\mathcal{E}|]([J\chi](\varepsilon+\hbar\omega_{p})-[J\chi](\varepsilon-\hbar\omega_{p}))[e^{\frac{\hbar\omega_{p}}{2K_{B}T_{L}}}||\mathcal{E}||_{L^{\infty}(\varepsilon',\mu',\phi')}+e^{\frac{-\hbar\omega_{p}}{2K_{B}T_{L}}}|\mathcal{E}|]e^{H}s(w)d\vec{w}d\vec{x})=
\]
\[
16\pi^{2}(\sum_{ikm}\int_{\Omega_{C}}K_{0}^{2}[J\chi]^{2}(\varepsilon)[||\mathcal{E}||_{L^{\infty}(\varepsilon',\mu',\phi')}+|\mathcal{E}|]{}^{2}e^{H}s(w)d\vec{w}d\vec{x}+
\sum_{ikm}\int_{\Omega_{C}}K_{1}^{2}
e^{H}s(w)d\vec{w}d\vec{x}
\]
\[
([J\chi](\varepsilon+\hbar\omega_{p})-[J\chi](\varepsilon-\hbar\omega_{p}))^{2}[e^{\frac{\hbar\omega_{p}}{2K_{B}T_{L}}}||\mathcal{E}||_{L^{\infty}(\varepsilon',\mu',\phi')}+e^{\frac{-\hbar\omega_{p}}{2K_{B}T_{L}}}|\mathcal{E}|]{}^{2}
\]
\[
+\sum_{ikm}\int_{\Omega_{C}}2K_{0}K_{1}[J\chi](\varepsilon)([J\chi](\varepsilon+\hbar\omega_{p})-[J\chi](\varepsilon-\hbar\omega_{p}))[||\mathcal{E}||_{L^{\infty}(\varepsilon',\mu',\phi')}+
\]
\[
|\mathcal{E}|][e^{\frac{\hbar\omega_{p}}{2K_{B}T_{L}}}||\mathcal{E}||_{L^{\infty}(\varepsilon',\mu',\phi')}+e^{\frac{-\hbar\omega_{p}}{2K_{B}T_{L}}}|\mathcal{E}|]e^{H}s(w)d\vec{w}d\vec{x})=
(4\pi)^{2} \times 
\]
\[
(\sum_{ikm}\int_{\Omega_{C}}K_{0}^{2}[J\chi]^{2}(\varepsilon)[||\mathcal{E}||_{L^{\infty}(\varepsilon',\mu',\phi')}^{2}+|\mathcal{E}|^{2}+2||\mathcal{E}||_{L^{\infty}(\varepsilon',\mu',\phi')}|\mathcal{E}|]e^{H}s(w)d\vec{w}d\vec{x}+ 
\]
\[
K_{1}^{2}
\sum_{ikm}\int_{\Omega_{C}}
(J\chi|^{\varepsilon+\hbar\omega_{p}}_{\varepsilon-\hbar\omega_{p}})^{2}
[e^{\frac{\hbar\omega_{p}}{2K_{B}T_{L}}}||\mathcal{E}||_{L^{\infty}(\varepsilon',\mu',\phi')}
+
e^{\frac{-\hbar\omega_{p}}{2K_{B}T_{L}}}|\mathcal{E}
]^2 
e^{H}s(w)d\vec{w}d\vec{x}
\]
\[
+\sum_{ikm}\int_{\Omega_{C}}2K_{0}K_{1}[J\chi](\varepsilon)([J\chi](\varepsilon+\hbar\omega_{p})-[J\chi](\varepsilon-\hbar\omega_{p}))[e^{\frac{\hbar\omega_{p}}{2K_{B}T_{L}}}||\mathcal{E}||_{L^{\infty}(\varepsilon',\mu',\phi')}^{2}
\]
\[
+e^{\frac{-\hbar\omega_{p}}{2K_{B}T_{L}}}|\mathcal{E}|^{2}+(e^{\frac{\hbar\omega_{p}}{2K_{B}T_{L}}}+e^{\frac{-\hbar\omega_{p}}{2K_{B}T_{L}}})||\mathcal{E}||_{L^{\infty}(\varepsilon',\mu',\phi')}|\mathcal{E}|]e^{H}s(w)d\vec{w}d\vec{x})=
\]
\[
(4\pi)^{2}(\sum_{ikm}\int_{\Omega_{C}}\{K_{0}^{2}[J\chi]^{2}(\varepsilon)+K_{1}^{2}([J\chi](\varepsilon+\hbar\omega_{p})-[J\chi](\varepsilon-\hbar\omega_{p}))^{2}e^{\frac{\hbar\omega_{p}}{K_{B}T_{L}}}
\]
\[
+2K_{0}K_{1}[J\chi](\varepsilon)([J\chi](\varepsilon+\hbar\omega_{p})-[J\chi](\varepsilon-\hbar\omega_{p}))e^{\frac{\hbar\omega_{p}}{2K_{B}T_{L}}}\}||\mathcal{E}||_{L^{\infty}(\varepsilon',\mu',\phi')}^{2}e^{H}s(w)d\vec{w}d\vec{x}
\]
\[
\sum_{ikm}\int_{\Omega_{C}}\{K_{0}^{2}[J\chi]^{2}(\varepsilon)+K_{1}^{2}([J\chi](\varepsilon+\hbar\omega_{p})-[J\chi](\varepsilon-\hbar\omega_{p}))^{2}e^{\frac{-\hbar\omega_{p}}{K_{B}T_{L}}}+
\]
\[
2K_{0}K_{1}[J\chi](\varepsilon)([J\chi](\varepsilon+\hbar\omega_{p})-[J\chi](\varepsilon-\hbar\omega_{p}))e^{\frac{-\hbar\omega_{p}}{2K_{B}T_{L}}}\}|\mathcal{E}|^{2}e^{H}s(w)d\vec{w}d\vec{x}+
\]
\[
2\sum_{ikm}\int_{\Omega_{C}}\{K_{0}^{2}[J\chi]^{2}(\varepsilon)+K_{1}^{2}([J\chi](\varepsilon+\hbar\omega_{p})-[J\chi](\varepsilon-\hbar\omega_{p}))^{2}
+
\]
\[
K_{0}K_{1}[J\chi](\varepsilon)
([J\chi]^{\varepsilon+\hbar\omega_{p}}_{\varepsilon-\hbar\omega_{p}})
(e^{\frac{\hbar\omega_{p}}{2K_{B}T_{L}}}+e^{\frac{-\hbar\omega_{p}}{2K_{B}T_{L}}})\}||\mathcal{E}||_{L^{\infty}(\varepsilon',\mu',\phi')}|\mathcal{E}|e^{H}s(w)d\vec{w}d\vec{x})
=
\]
\[
(4\pi)^{2} 
(\sum_{ikm}\int_{\Omega_{C}}\{K_{0}[J\chi](\varepsilon)+K_{1}
([J\chi]^{\varepsilon+\hbar\omega_{p}}_{\varepsilon-\hbar\omega_{p}})e^{\frac{\hbar\omega_{p}}{2K_{B}T_{L}}}\}^{2}||\mathcal{E}||_{L^{\infty}(\varepsilon',\mu',\phi')}^{2}e^{H}s(w)d\vec{w}d\vec{x}
\]
\[
+
\sum_{ikm}\int_{\Omega_{C}}\{K_{0}[J\chi](\varepsilon)+K_{1}([J\chi](\varepsilon+\hbar\omega_{p})-[J\chi](\varepsilon-\hbar\omega_{p}))e^{\frac{-\hbar\omega_{p}}{2K_{B}T_{L}}}\}^{2}|\mathcal{E}|^{2}e^{H}s(w)d\vec{w}d\vec{x}+
\]
\[
2\sum_{ikm}\int_{\Omega_{C}}\{[K_{0}[J\chi](\varepsilon)+K_{1}
([J\chi]^{\varepsilon+\hbar\omega_{p}}_{\varepsilon-\hbar\omega_{p}})]^{2}+2K_{0}K_{1}[J\chi](\varepsilon)([J\chi](\varepsilon+\hbar\omega_{p})
\]
\[
-[J\chi](\varepsilon-\hbar\omega_{p}))(\cosh(\frac{\hbar\omega_{p}}{2K_{B}T_{L}})-1)\}||\mathcal{E}||_{L^{\infty}(\vec{\varepsilon})}|\mathcal{E}|e^{H}s(w)d\vec{w}d\vec{x}).
\]
Now we can try to bound as
\[
||Q(\mathcal{E})||^{2}{}_{L^{2}(\mathcal{T}_{h})}\leq
2\max\{[K_{0}[J\chi](\varepsilon)+K_{1}([J\chi](\varepsilon+\hbar\omega_{p})-[J\chi](\varepsilon-\hbar\omega_{p}))]^{2}+
(4\pi)^{2} \times 
\]
\[
([\max\{K_{0}[J\chi](\varepsilon)+K_{1}([J\chi]^{\varepsilon+\hbar\omega_{p}}_{\varepsilon-\hbar\omega_{p}})e^{\frac{\hbar\omega_{p}}{2K_{B}T_{L}}}\}^{2}]\sum_{ikm}\int_{\Omega_{C}}||\mathcal{E}||_{L^{\infty}(\varepsilon',\mu',\phi')}^{2}e^{H}s(w)d\vec{w}d\vec{x}
\]
\[
[\max\{K_{0}[J\chi](\varepsilon)+K_{1}
([J\chi]^{\varepsilon+\hbar\omega_{p}}_{\varepsilon-\hbar\omega_{p}})e^{\frac{-\hbar\omega_{p}}{2K_{B}T_{L}}}\}^{2}]\sum_{ikm}\int_{\Omega_{C}}|\mathcal{E}|^{2}e^{H}s(w)d\vec{w}d\vec{x}+
\]
\[
2K_{0}K_{1}[J\chi](\varepsilon)([J\chi]^{\varepsilon+\hbar\omega_{p}}_{\varepsilon-\hbar\omega_{p}})(\cosh(\frac{\hbar\omega_{p}}{2K_{B}T_{L}})-1)\}\sum_{ikm}\int_{\Omega_{C}}||\mathcal{E}||_{L^{\infty}(\vec{\varepsilon})}|\mathcal{E}|e^{H}s(w)d\vec{w}d\vec{x}).
\]
If we define each one of the maxima above as 
 $M_{i},\, i\in\{0,1,2\}$
 respectively, we have 
\[
||Q(\mathcal{E})||^{2}{}_{L^{2}(\mathcal{T}_{h})}\leq
(4\pi)^{2}([M_{0}]^{2}\sum_{ikm}\int_{\Omega_{C}}||\mathcal{E}||_{L^{\infty}(\varepsilon',\mu',\phi')}^{2}e^{H}s(w)d\vec{w}d\vec{x}+
\]
\[
[M_{1}]^{2}\sum_{ikm}\int_{\Omega_{C}}|\mathcal{E}|^{2}e^{H}s(w)d\vec{w}d\vec{x}+
2\times M_{2}\sum_{ikm}\int_{\Omega_{C}}||\mathcal{E}||_{L^{\infty}(\vec{\varepsilon})}|\mathcal{E}|e^{H}s(w)d\vec{w}d\vec{x})
\]
\[
=(4\pi)^{2}(M_{0}^{2}\sum_{ikm}\int_{\Omega_{C}}||\mathcal{E}||_{L^{\infty}(\varepsilon',\mu',\phi')}^{2}e^{H}s(w)d\vec{w}d\vec{x}+
\]
\[
M_{1}^{2}\sum_{ikm}\int_{\Omega_{C}}|\mathcal{E}|^{2}e^{H}s(w)d\vec{w}d\vec{x}+M_{2}\sum_{ikm}\int_{\Omega_{C}}||\mathcal{E}||_{L^{\infty}(\vec{\varepsilon})}|\mathcal{E}|e^{H}s(w)d\vec{w}d\vec{x})
\]
\[
=16\pi{}^{2}(M_{0}^{2}\sum_{ikm}\int_{\Omega_{C}}||\mathcal{E}||_{L^{\infty}(\varepsilon',\mu',\phi')}^{2}e^{H}s(w)d\vec{w}d\vec{x}+
\]
\[
M_{1}^{2}\sum_{ikm}\int_{\Omega_{C}}|\mathcal{E}|^{2}e^{H}s(w)d\vec{w}d\vec{x}+M_{2}\sum_{ikm}\int_{\Omega_{C}}||\mathcal{E}||_{L^{\infty}(\vec{\varepsilon})}|\mathcal{E}|e^{H}s(w)d\vec{w}d\vec{x})
\]
\[
=16\pi{}^{2}[M_{0}^{2}||\{||\mathcal{E}||_{L^{\infty}(\vec{w}')}\}||_{L^{2}(\vec{x},\vec{w})}^{2}+M_{1}^{2}||\mathcal{E}||_{L^{2}(\vec{x},\vec{w})}^{2}+M_{2}(||\mathcal{E}||_{L^{\infty}(\vec{\varepsilon})},|\mathcal{E}|)_{L^{2}(\vec{x},\vec{w})}].
\]
If we combine the estimates  of 
$T_{i},\, i=1,2,3$ for our final estimate, we have that 
\[
(\partial_{t}E_{h},E_{h})_{\mathcal{T}_{h}}+\frac{1}{4}\int_{e_{h}}(E_{h}^{+}-E_{h}^{-})^{2}|\beta\cdot\widehat{n}|e^{H}d\sigma\leq-\mathcal{A}(\mathbb{\mathcal{E}},E_{h}),
\]
 with 
\[
\mathcal{-A}(\mathcal{E},E_{h})\leq|\mathcal{A}(\mathcal{E},E_{h})|\leq|T_{1}|+|T_{2}|+|T_{3}|,
\]
 where 
\[
|T_{1}|\leq Ch^{k+1}|f|_{H_{(\mathcal{T}_{h})}^{k+1}}||E_{h}||_{L_{\mathcal{T}_{h}}^{2}}|\beta|_{W^{\{1,\infty\}}(\mathcal{T}_{h})},
\]
\[
|T_{2}|\leq\frac{1}{8}||[e^{H}E_{h}]\sqrt{|\widehat{\mathcal{\beta}}\cdot\widehat{n}|}\,||^{2}{}_{L^{2}(e_{h})}+C||\beta||_{L_{\mathcal{T}_{h}}^{\infty}}h^{2k+1}|f|_{H^{k+1}(\mathcal{T}_{h})}^{2},
\]
\[
|T_{3}|=|\sum_{ikm}\int_{\Omega_{C}}Q(\mathcal{E})E_{h}e^{H}s(w)d\vec{w}d\vec{x}|=|(Q(\mathcal{E}),E_{h})_{\mathcal{T}_{h}}|\leq||Q(\mathcal{E})||{}_{L^{2}(\mathcal{T}_{h})}\cdot||E_{h}||_{L^{2}(\mathcal{T}_{h})},
\]
\[
||Q(\mathcal{E})||^{2}{}_{L^{2}(\mathcal{T}_{h})}\leq
\]
\[
16\pi{}^{2}[M_{0}^{2}||\{||\mathcal{E}||_{L^{\infty}(\vec{w}')}\}||_{L^{2}(\vec{x},\vec{w})}^{2}+M_{1}^{2}||\mathcal{E}||_{L^{2}(\vec{x},\vec{w})}^{2}+M_{2}(||\mathcal{E}||_{L^{\infty}(\vec{\varepsilon})},|\mathcal{E}|)_{L^{2}(\vec{x},\vec{w})}],
\]
 so we have 
\[
(\partial_{t}E_{h},E_{h})_{\mathcal{T}_{h}}+\frac{\int_{e_{h}}(E_{h}^{+}-E_{h}^{-})^{2}|\beta\cdot\widehat{n}|e^{H}d\sigma}{4}
\leq
\sum_{i=1}^3 |T_{i}|
\leq
\frac{||[e^{H}E_{h}]\sqrt{|\widehat{\mathcal{\beta}}\cdot\widehat{n}|}\,||^{2}{}_{L^{2}(e_{h})}}{8} +
\]
\[
Ch^{k+1}|f|_{H_{(\mathcal{T}_{h})}^{k+1}}||E_{h}||_{L_{\mathcal{T}_{h}}^{2}}|\beta|_{W^{\{1,\infty\}}(\mathcal{T}_{h})}
+C||\beta||_{L_{\mathcal{T}_{h}}^{\infty}}h^{2k+1}|f|_{H^{k+1}(\mathcal{T}_{h})}^{2}
+||E_{h}||_{L^{2}(\mathcal{T}_{h})}\times
\]
\[
4\pi M\sqrt{||\{||\mathcal{E}||_{L^{\infty}(\vec{w}')}\}||_{L^{2}(\vec{x},\vec{w})}^{2}+||\mathcal{E}||_{L^{2}(\vec{x},\vec{w})}^{2}+||\{||\mathcal{E}||_{L^{\infty}(\vec{\varepsilon})}\}||_{L^{2}(\vec{x},\vec{w})}\times||\mathcal{E}||{}_{L^{2}(\vec{x},\vec{w})}},
\]
\[
\mathrm{with}\quad 
M=\max\{M_{0},M_{1},\sqrt{M_{2}}\}.
\]
We can get the following bound then, 
\[
(\partial_{t}E_{h},E_{h})_{\mathcal{T}_{h}}+\frac{1}{4}\int_{e_{h}}(E_{h}^{+}-E_{h}^{-})^{2}|\beta\cdot\widehat{n}|e^{H}d\sigma\leq
Ch^{k+1}|f|_{H_{(\mathcal{T}_{h})}^{k+1}}||E_{h}||_{L_{\mathcal{T}_{h}}^{2}}|\beta|_{W^{\{1,\infty\}}(\mathcal{T}_{h})}
\]
\[
+\frac{1}{8}||[e^{H}E_{h}]\sqrt{|\widehat{\mathcal{\beta}}\cdot\widehat{n}|}\,||^{2}{}_{L^{2}(e_{h})}+C||\beta||_{L_{\mathcal{T}_{h}}^{\infty}}h^{2k+1}|f|_{H^{k+1}(\mathcal{T}_{h})}^{2}
+4\pi M||E_{h}||_{L^{2}(\mathcal{T}_{h})} \times 
\]
\[
\sqrt{||\{||\mathcal{E}||_{L^{\infty}(\vec{w}')}\}||_{L^{2}(\vec{x},\vec{w})}^{2}+||\mathcal{E}||_{L^{2}(\vec{x},\vec{w})}^{2}+2||\{||\mathcal{E}||_{L^{\infty}(\vec{\varepsilon})}\}||_{L^{2}(\vec{x},\vec{w})}\times||\mathcal{E}||{}_{L^{2}(\vec{x},\vec{w})}}
=
\]
\[
Ch^{k+1}|f|_{H_{(\mathcal{T}_{h})}^{k+1}}||E_{h}||_{L_{\mathcal{T}_{h}}^{2}}|\beta|_{W^{\{1,\infty\}}(\mathcal{T}_{h})}
\]
\[
+\frac{1}{8}||[e^{H}E_{h}]\sqrt{|\widehat{\mathcal{\beta}}\cdot\widehat{n}|}\,||^{2}{}_{L^{2}(e_{h})}+C||\beta||_{L_{\mathcal{T}_{h}}^{\infty}}h^{2k+1}|f|_{H^{k+1}(\mathcal{T}_{h})}^{2}
\]
\[
+4\pi M||E_{h}||_{L^{2}(\mathcal{T}_{h})}[||\{||\mathcal{E}||_{L^{\infty}(\vec{w}')}\}||_{L^{2}(\vec{x},\vec{w})}+||\mathcal{E}||_{L^{2}(\vec{x},\vec{w})}],
\]
which is the result we want.
Now, following similar arguments as in \cite{CGP}, if we assume that 
\[
|\beta|_{W^{\{1,\infty\}}(\mathcal{T}_{h})}\leq C',
\]

which holds for quadratically confined electrostatic
 potentials \cite{CGP},
 when
\[
\beta=(-\partial_{p}\varepsilon\,\mu,qE\mu,qE(1\!-\!\mu_{m}^{2})|_{\mu_{m-}}^{\mu_{m+}}/p),
\quad 
E=-\partial_{x}V,\quad|V|\leq K|x-x_{0}|^{2},
\]
with $K,\, x_{0}$
constants, we have then 
\[
(\partial_{t}E_{h},E_{h})_{\mathcal{T}_{h}}+\frac{1}{4}\int_{e_{h}}(E_{h}^{+}-E_{h}^{-})^{2}|\beta\cdot\widehat{n}|e^{H}d\sigma\leq
\]
\[
C^{*}h^{k+1}|f|_{H_{(\mathcal{T}_{h})}^{k+1}}||E_{h}||_{L_{\mathcal{T}_{h}}^{2}}
+\frac{1}{8}||[e^{H}E_{h}]\sqrt{|\widehat{\mathcal{\beta}}\cdot\widehat{n}|}\,||^{2}{}_{L^{2}(e_{h})}+C||\beta||_{L_{\mathcal{T}_{h}}^{\infty}}h^{2k+1}|f|_{H^{k+1}(\mathcal{T}_{h})}^{2}
\]
\[
+M'||E_{h}||_{L^{2}(\mathcal{T}_{h})}[||\{||\mathcal{E}||_{L^{\infty}(\vec{w}')}\}||_{L^{2}(\vec{x},\vec{w})}+||\mathcal{E}||_{L^{2}(\vec{x},\vec{w})}],
\]
with 
\ $M'=4\pi M$.
 So 
\[
(\partial_{t}E_{h},E_{h})_{\mathcal{T}_{h}}+\frac{1}{4}\int_{e_{h}}(E_{h}^{+}-E_{h}^{-})^{2}|\beta\cdot\widehat{n}|e^{H}d\sigma-\frac{1}{8}||[e^{H}E_{h}]\sqrt{|\widehat{\mathcal{\beta}}\cdot\widehat{n}|}\,||^{2}{}_{L^{2}(e_{h})}\leq
\]
\[
+C^{*}h^{k+1}|f|_{H_{(\mathcal{T}_{h})}^{k+1}}||E_{h}||_{L_{\mathcal{T}_{h}}^{2}}+C||\beta||_{L_{\mathcal{T}_{h}}^{\infty}}h^{2k+1}|f|_{H^{k+1}(\mathcal{T}_{h})}^{2}
\]
\[
+M'||E_{h}||_{L^{2}(\mathcal{T}_{h})}[||\{||\mathcal{E}||_{L^{\infty}(\vec{w}')}\}||_{L^{2}(\vec{x},\vec{w})}+||\mathcal{E}||_{L^{2}(\vec{x},\vec{w})}].
\]
We have  then that
\[
(\partial_{t}E_{h},E_{h})_{\mathcal{T}_{h}}\leq(\partial_{t}E_{h},E_{h})_{\mathcal{T}_{h}}+\frac{1}{8}||[e^{H}E_{h}]\sqrt{|\widehat{\mathcal{\beta}}\cdot\widehat{n}|}\,||^{2}{}_{L^{2}(e_{h})}\leq
\]
\[
C^{*}h^{k+1}|f|_{H_{(\mathcal{T}_{h})}^{k+1}}||E_{h}||_{L_{\mathcal{T}_{h}}^{2}}+C||\beta||_{L_{\mathcal{T}_{h}}^{\infty}}h^{2k+1}|f|_{H^{k+1}(\mathcal{T}_{h})}^{2}
\]
\[
+M'[||E_{h}||_{L^{2}(\mathcal{T}_{h})}||\{||\mathcal{E}||_{L^{\infty}(\vec{w}')}\}||_{L^{2}(\vec{x},\vec{w})}+||E_{h}||_{L^{2}(\mathcal{T}_{h})}||\mathcal{E}||_{L^{2}(\vec{x},\vec{w})}],
\]
 therefore 
\[
\frac{1}{2}\frac{d}{dt}||E_{h}||_{\mathcal{T}_{h}}^{2}=\frac{1}{2}\frac{d}{dt}(E_{h},E_{h})_{\mathcal{T}_{h}}=(\partial_{t}E_{h},E_{h})_{\mathcal{T}_{h}}\leq
\]
\[
C^{*}h^{k+1}|f|_{H_{(\mathcal{T}_{h})}^{k+1}}||E_{h}||_{L_{\mathcal{T}_{h}}^{2}}+C||\beta||_{L_{\mathcal{T}_{h}}^{\infty}}h^{2k+1}|f|_{H^{k+1}(\mathcal{T}_{h})}^{2}
\]
\[
+M'[||E_{h}||_{L^{2}(\mathcal{T}_{h})}||\{||\mathcal{E}||_{L^{\infty}(\vec{w}')}\}||_{L^{2}(\vec{x},\vec{w})}+||E_{h}||_{L^{2}(\mathcal{T}_{h})}||\mathcal{E}||_{L^{2}(\vec{x},\vec{w})}],
\]
 with 
 $E_{h}=0$
 at 
 $t=0$.
 We only need to proceed with a Gronwall inequality,
\[
\frac{1}{2}\frac{d}{dt}||E_{h}||_{\mathcal{T}_{h}}^{2}\leq C||\beta||_{L_{\mathcal{T}_{h}}^{\infty}}h^{2k+1}|f|_{H^{k+1}(\mathcal{T}_{h})}^{2}+
\]
\[
+||E_{h}||_{L^{2}(\mathcal{T}_{h})}[C^{*}h^{k+1}|f|_{H_{(\mathcal{T}_{h})}^{k+1}}+M'||\{||\mathcal{E}||_{L^{\infty}(\vec{w}')}\}||_{L^{2}(\vec{x},\vec{w})}+M'||\mathcal{E}||_{L^{2}(\vec{x},\vec{w})}].
\]
We now follow a similar argumentation as in \cite{CGP}.
 We claim that 
\[
\frac{1}{2}\frac{d}{dt}||E_{h}||_{\mathcal{T}_{h}}^{2}\leq Ch^{2k+1}|f|_{H^{k+1}(\mathcal{T}_{h})}^{2}+C^{*}h||E_{h}||_{L_{\mathcal{T}_{h}}^{2}}^{2}
\]
\[
+M'\sqrt{||E_{h}||_{L^{2}(\mathcal{T}_{h})}^{2}}[||\{||\mathcal{E}||_{L^{\infty}(\vec{w}')}\}||_{L^{2}(\vec{x},\vec{w})}+||\mathcal{E}||_{L^{2}(\vec{x},\vec{w})}].
\]
Due to the first two terms, we expect a contribution related to them
 similar to what happens in \cite{CGP}.
Our initial condition is 
$E_{h}|_{t=0}=0$. We claim 
\[
||E_{h}||_{\mathcal{T}_{h}}\leq C\sqrt{t}\exp(Cht)h^{k+1/2}|f|_{L^{\infty}([0,t],H^{k+1}(\Omega_{D}))}
\]
\[
+t\times M'[||\{||\mathcal{E}||_{L^{\infty}(\vec{w}')}\}||_{L^{2}(\vec{x},\vec{w})}+||\mathcal{E}||_{L^{2}(\vec{x},\vec{w})}].
\]
In conclusion, similarly to \cite{CGP}, the terms above will have a square root times
 exponential behavior as expected. However, most importantly, the terms above
 not appearing in \cite{CGP}
 (not similar to the ones appearing in \cite{CGP})
 are expected to
 have a bound with linear growth.
 Finally, our intention is to bound the whole error, meaning 
\[
\mathbb{E=\mathcal{E}+}E_{h},
\quad 
\mathrm{with} \quad
\mathcal{E}=f-\mathbb{P}f
\quad
\mathrm{and}
\quad
E_{h}=\mathbb{P}f-f_{h}.
\]
 Then 
\[
||\mathbb{E}_{||_{\mathcal{T}_{h}}}\leq{||}\mathcal{E}||_{\mathcal{T}_{h}} + ||E_{h}||_{\mathcal{T}_{h}}
\]
and we expect the projection error to contribute as below, 
\[
||\mathcal{E}||_{\mathcal{T}_{h}}\leq C'h^{k+1}.
\]
 Therefore we have 
\[
||\mathbb{E}||_{\mathcal{T}_{h}}\leq||\mathcal{E}||_{\mathcal{T}_{h}}+||E_{h}||_{\mathcal{T}_{h}}
\]
\[
\leq C\sqrt{t}\exp(Cht)h^{k+1/2}|f|_{L^{\infty}([0,t],H^{k+1}(\Omega_{D}))}
\]
\[
+t\times M'[||\{||\mathcal{E}||_{L^{\infty}(\vec{w}')}\}||_{L^{2}(\vec{x},\vec{w})}+||\mathcal{E}||_{L^{2}(\vec{x},\vec{w})}]+C'h^{k+1},
\]
 and if we simply take 
\[
\bar{C}=\max\{C,M',C'\},\quad \mathrm{then}
\]
\[
||\mathbb{E}||_{\mathcal{T}_{h}}\leq
\bar{C} \times 
\]
\[
\{\sqrt{t}e^{\bar{C}ht}h^{k+1/2}|f|_{L^{\infty}([0,t],H^{k+1}(\Omega_{D}))}+t[||\{||\mathcal{E}||_{L^{\infty}(\vec{w}')}\}||_{L^{2}(\vec{x},\vec{w})}+||\mathcal{E}||_{L^{2}(\vec{x},\vec{w})}]+h^{k+1}\}.
\]
 If we use the regularity property we get
\[
||\mathcal{M}^{-1/2}[f(t,\cdot,\cdot)-\mathcal{M}]||_{H^{k}}\leq K\exp(-\gamma t),\quad t\geq0\implies
\]
\[
||f(t=0)||_{H^{k+1}}\leq K\implies||f(t)||_{H^{k+1}}\leq K.
\]
Going back to our result, if 
 $t\geq Ch$, then we have (after a few time-steps)
\[
||\mathbb{E}||_{\mathcal{T}_{h}}\leq
\]
\[
\bar{C}\{\sqrt{t}\exp(\bar{C}ht)h^{k+1/2}|f|_{L^{\infty}([0,t],H^{k+1}(\Omega_{D}))}+t[||\{||\mathcal{E}||_{L^{\infty}(\vec{w}')}\}||_{L^{2}(\vec{x},\vec{w})}+||\mathcal{E}||_{L^{2}(\vec{x},\vec{w})}]\}.
\]

 This is interesting because we have square-root times exponential growth
 versus linear growth.
 It would seem the first one will win the latter for larger times, after
 a few time steps again, and then we should have the 
 $L^{2}$
 error estimates of the DG solution,
\[
||\mathbb{E}||_{\mathcal{T}_{h}}\leq\bar{C}\sqrt{t}\exp(\bar{C}ht)h^{k+1/2}|f|_{L^{\infty}([0,t],H^{k+1}(\Omega_{D}))}.
\]

So although we get an extra term (not similar to any term reported in \cite{CGP}), the terms analog to the ones appearing in \cite{CGP} become more meaningful at long times,
 and since the other one is negligible for large times, we get a similar estimate as in \cite{CGP}.
 Therefore, since 
 $\mathbb{E=}f-f_{h}$,
 we have proved that for large times (i.e. after some transient time), 
\[
||f-f_{h}||_{\mathcal{T}_{h}}\leq\bar{C}\sqrt{t}\exp(\bar{C}ht)h^{k+1/2}|f|_{L^{\infty}([0,t],H^{k+1}(\Omega_{D}))}.
\]

\end{proof}

We now present our second result regarding error estimation for our DG scheme with curvilinear momentum coordinates, at the semi-discrete stage. 

\begin{theorem}
If 
$f_{h}$
 is the semidiscrete DG solution to our Boltzmann equation with a linear
 collision operator in the semiconductor problem, then 
\[
||f_{h}(t,\cdot,\cdot)-M||\leq C\sqrt{t}\exp(Cht)h^{k+1/2}|f|_{L^{\infty}([0,t],H^{k+1}(\Omega_{D}))}+3e^{-\lambda t}||f_{0}-M||_{B^{2}(\mathbb{R}^{d})}
\]
with the constant 
$C=C(diam(\Omega_{D}),||\beta||_{W^{\{1,\infty\}}(\Omega_{D})}).$
\end{theorem}
\begin{proof} 
 We have by triangle inequality 
\[
||f_{h}-M||_{L^{2}(\Omega_{D})}=||f_{h}-f||_{L^{2}(\Omega_{D})}+||f-M||_{L^{2}(\Omega_{D})}\leq
||f_{h}-f||_{L^{2}(\Omega_{D})}+||f-M||_{B^{2}(\Omega_{D})}.
\]
Since by definition 
\[
||f-M||_{B^{2}(\Omega_{D})}^{2}=||(f-M)M^{-1}||_{L^{2}(\Omega_{D})}^{2}=||(f-M)\exp(\varepsilon/K_{B}T)||_{L^{2}(\Omega_{D})}^{2},
\]
 and 
 $\varepsilon\geq0\implies\exp(\varepsilon/K_{B}T)\geq1$, so  
\[
||f-M||_{B^{2}(\Omega_{D})}^{2}=||(f-M)\exp(\varepsilon/K_{B}T)||_{L^{2}(\Omega_{D})}^{2}\geq||(f-M)||_{L^{2}(\Omega_{D})}^{2}.
\]
We have then 
\[
||f_{h}-M||_{L^{2}(\Omega_{D})}=||f_{h}-f||_{L^{2}(\Omega_{D})}+||f-M||_{L^{2}(\Omega_{D})}\leq||f_{h}-f||_{L^{2}(\Omega_{D})}+||f-M||_{B^{2}(\Omega_{D})}
\]
\[
\leq
\bar{C}\sqrt{t}\exp(\bar{C}ht)h^{k+1/2}|f|_{L^{\infty}([0,t],H^{k+1}(\Omega_{D}))}+||f-M||_{B^{2}(\Omega_{D})},
\]
and we only need to justify that 
\[
||f-M||_{B^{2}(\Omega_{D})}\leq3\exp(-\lambda t)||f_{0}-M||_{B^{2}(\mathbb{R}^{d})},
\]
 with 
 $f_{0}$
 the initial condition.
 This is simply a consequence of the work in \cite{Herau}
 (as referenced in \cite{CGP}).
 So both inequalities hold and we get the desired result.
 Therefore, for large times (again, after some transient time), 
\[
||f_{h}(t,\cdot,\cdot)-M||\leq C\sqrt{t}\exp(Cht)h^{k+1/2}|f|_{L^{\infty}([0,t],H^{k+1}(\Omega_{D}))}
+3e^{-\lambda t}||f_{0}-M||_{B^{2}(\mathbb{R}^{d})}
\]
 as expected.
\end{proof} 

\section*{Appendix 2}
\section{Positivity Preservation in DG Scheme for BP}
\subsection{1Dx-2Dp problem: Preliminaries}
Using the notation in \cite{EECHXM-JCP}, the semi-discrete DG formulation is written in the form below. \\

Find $f_h \, \in \, V_h^k$ such that $\forall \, g_h \, \in V_h^k$ 
and $\, \forall \, \Omega_{ikm} $
\begin{align}
\int_{\Omega_{ikm}} Q(f_h) g_h \, p^2 dp d\mu dx  &=
 \partial_t  \int_{\Omega_{ikm}}  f_h \, g_h \, p^2 dp d\mu dx 
\nonumber\\
-\int_{\Omega_{ikm}} H^{(x)}  \, f_h \, \partial_x g_h \, p^2 dp d\mu dx \,
&\pm 
 \int_{\tilde{\Omega}^{(x)}_{km}}  \,  \widehat{H^{(x)}f_h}|_{x_{i\pm}} \, g_h|_{x_{i\pm}}^{\mp} \, p^2 dp d\mu  
\nonumber\\
- \int_{\Omega_{ikm}} { p^2 } \,
H^{(p)}  f_h \, \partial_p g_h \, d\mu dx
& \pm  
p^2_{k^\pm}
\int_{\tilde{\Omega}^{(p)}_{im}}  \,
\widehat{H^{(p)}f_h}|_{p_{k\pm}} \, g_h|_{p_{k\pm}}^{\mp} \, d\mu dx -
\nonumber\\
 \int_{\Omega_{ikm}} (1-\mu^2) H^{(\mu)} f_h   \partial_{\mu} g_h  p  dp d\mu dx 
&\pm 
(1-\mu_{m\pm}^2)
\int_{\tilde{\Omega}^{(\mu)}_{ik}} 
\widehat{H^{(\mu)}f_h}|_{\mu_{m\pm}}  g_h|_{\mu_{m\pm}}^{\mp} p  dp dx,
\nonumber
\end{align}
where we have defined the following terms (taking in account that $\partial_p \varepsilon(p)>0$), 
\begin{align}
&H^{(x)} = \mu \, \partial_p \varepsilon, \ \ 
 \widehat{H^{(x)}f}|_{x_{i\pm}}
= \partial_p \varepsilon
\widehat{ f_h  \mu }|_{x_{i\pm}};
\tilde{\Omega}^{(x)}_{km} = [r_{k-},r_{k+}]\times[\mu_{m-},\mu_{m+}] 
= \partial_x \Omega_{km},
\nonumber\\
&H^{(p)} = -q E \mu, \ \ 
\widehat{H^{(p)}f}|_{p_{k\pm}}
=
 -q \widehat{ E f_h \mu } |_{p_{k\pm}} 
;
\tilde{\Omega}^{(p)}_{im} = [x_{i-},x_{i+}]\times[\mu_{m-},\mu_{m+}] 
= \partial_p \Omega_{im},
\nonumber\\
&H^{(\mu)} 
=
-qE,\ \ 
\widehat{H^{(\mu)}f}|_{\mu_{m\pm}}
=
-q  \widehat{ E f_h }|_{\mu_{m\pm}};
\tilde{\Omega}^{(\mu)}_{ik} = [x_{i-},x_{i+}]\times[r_{k-},r_{k+}] 
= \partial_{\mu}\Omega_{ik}
.
\nonumber
\end{align}
The weak form of the collisional operator in the DG scheme is, specifically,
\begin{align}
&
\int_{\Omega_{ikm}} Q(f_h) \, g_h \, p^2 \, dp d\mu dx   \, 
 = 
\int_{\Omega_{ikm}} \left[ G(f_h) - \nu(\varepsilon(p)) f_h \right] \, g_h \, p^2 \, dp d\mu dx   \, = 2\pi \times
\nonumber\\
&
\int_{\Omega_{ikm}} 
\sum_{j=-1}^{1} c_j  \chi(\varepsilon(p) + j\hbar\omega)
\int_{-1}^{1} d\mu' \left. 
f_h(x,p(\varepsilon'),\mu')  p^2(\varepsilon') \frac{dp'}{d\varepsilon'} \right|_{\varepsilon(p) + j\hbar\omega }  
g_h p^2  dp d\mu dx  
\nonumber\\
&
-
4\pi 
\int_{\Omega_{ikm}} 
f_h(x,p,\mu,t)  
\sum_{j=-1}^{1} c_j \, \chi(\varepsilon(p) - j\hbar\omega)  \left.  p^2(\varepsilon') \frac{dp'}{d\varepsilon'}  \right|_{\varepsilon(p) - j\hbar\omega } 
g_h \, p^2 \, dp d\mu dx  \, .
\nonumber
\end{align}
The cell average of $f_h$ in $\Omega_{ikm}$ is
\begin{equation}
\bar{f}_{ikm} = 
{\int_{\Omega_{ikm}} f_h \,  p^2 \, dp d\mu dx }/{\int_{\Omega_{ikm}} p^2 \, dp d\mu dx} = 
{\int_{\Omega_{ikm}} f_h \, dV  }/{ V_{ikm} }, \,
\end{equation}
where, for our particular spherical curvilinear coordinates,
\vspace{-0.5cm}
\begin{equation*}
V_{ikm} = \int_{\Omega_{ikm}} dV , \, 
dV = \tau \prod_{d=1}^3 z_d ,
\,
(z_1,z_2,z_3) =
(x,p,\mu), 
\,
\tau = \sqrt{\gamma\lambda}, \,
\gamma = 1, \, \lambda = p^2 .
\end{equation*}
The time evolution of the cell average in the DG scheme is given by
\begin{align}
&
\partial_t \bar{f}_{ikm} = 
 -  \frac{1}{V_{ikm}} \left[
 \int_{\partial_x \Omega_{km}}  \,  \widehat{H^{(x)}f_h}|_{x_{i+}} 
 \, p^2 dp d\mu  
-
 \int_{\partial_x \Omega_{km}}  \,  \widehat{H^{(x)}f_h}|_{x_{i-}}  
 \, p^2 dp d\mu  
\right. 
\nonumber\\
& +  
p^2_{k^+}
\int_{\partial_p \Omega_{im}}  \,
\widehat{H^{(p)}f_h}|_{p_{k+}} 
\, d\mu dx
-  
p^2_{k^-}
\int_{\partial_p \Omega_{im}}  \,
\widehat{H^{(p)}f_h}|_{p_{k-}} 
\, d\mu dx
\nonumber\\
&+ 
\left.
(1-\mu_{m+}^2)
\int_{\partial_{\mu}\Omega_{ik}} 
\widehat{H^{(\mu)}f_h}|_{\mu_{m+}}
p dp dx
-
(1-\mu_{m-}^2)
\int_{\partial_{\mu}\Omega_{ik}} 
\widehat{H^{(\mu)}f_h}|_{\mu_{m-}} 
p dp dx
\right]
\nonumber\\
& +\frac{2\pi
\int_{\Omega_{ikm}} 
\sum_{j=-1}^{1} c_j  \chi(\varepsilon(p) + j\hbar\omega)
\int_{-1}^{1} d\mu' \left. 
f_h(x,p(\varepsilon'),\mu')  p^2(\varepsilon') \frac{dp}{d\varepsilon}  \right|_{\varepsilon(p) + j\hbar\omega }  
p^2  dp d\mu dx  
}{V_{ikm}} 
\nonumber\\
&-\frac{4\pi}{V_{ikm}} 
\int_{\Omega_{ikm}} 
f_h(x,p,\mu,t)  
\sum_{j=-1}^{+1} c_j \, \chi(\varepsilon(p) - j\hbar\omega)  \left.  p^2(\varepsilon') \frac{dp'}{d\varepsilon'} \right|_{\varepsilon(p) - j\hbar\omega } 
\, p^2 \, dp d\mu dx .
\nonumber
\end{align}

Regarding the time discretization, we will apply a TVD RK-DG scheme. These schemes are convex combinations of Euler methods. Therefore, it suffices if 
we consider the time evolution of the cell average in the DG scheme using a Forward Euler Method, so $\partial_t \bar{f}_{ikm} \approx (\bar{f}_{ikm}^{n+1} - \bar{f}_{ikm}^n)/\Delta t^n $, and  


\begin{align}
&
\bar{f}_{ikm}^{n+1} = \bar{f}_{ikm}^n
-
\frac{\Delta t^n}{V_{ikm}} 
 \int_{\partial_x \Omega_{km}}  \,  \widehat{H^{(x)}f_h}|_{x_{i+}} 
 \, p^2 dp d\mu  
-
 \int_{\partial_x \Omega_{km}}  \,  \widehat{H^{(x)}f_h}|_{x_{i-}}  
 \, p^2 dp d\mu  
\nonumber\\
-&\frac{\Delta t^n (1-\mu_{m+}^2)}{V_{ikm}}
\int_{\partial_{\mu}\Omega_{ik}} 
\widehat{H^{(\mu)}f_h}|_{\mu_{m+}}
p dp dx
-
(1-\mu_{m-}^2)
\int_{\partial_{\mu}\Omega_{ik}} 
\widehat{H^{(\mu)}f_h}|_{\mu_{m-}} 
p dp dx
\nonumber\\
- &\frac{\Delta t^n p^2_{k^+}}{V_{ikm}}
\int_{\partial_p \Omega_{im}}  \,
\widehat{H^{(p)}f_h}|_{p_{k+}} 
\, d\mu dx
-  
p^2_{k^-}
\int_{\partial_p \Omega_{im}}  \,
\widehat{H^{(p)}f_h}|_{p_{k-}} 
\, d\mu dx
+ 
\nonumber\\
\sum_{j=-1}^{1} &
\int_{\Omega_{ikm}} \frac{2\pi\Delta t^n}{V_{ikm}} 
 c_j  \chi(\varepsilon(p) + j\hbar\omega)
\int_{-1}^{1} d\mu' \left. 
f_h(x,p',\mu') \, p^2(\varepsilon') \frac{dp'}{d\varepsilon'}  \right|_{\varepsilon(p) + j\hbar\omega }  
p^2 dp d\mu dx  
\nonumber\\
-&\frac{4\pi\Delta t^n}{V_{ikm}} \int_{\Omega_{ikm}}  
f_h(x,p,\mu,t)  
\sum_{j=-1}^{1} c_j  \chi(\varepsilon(p) - j\hbar\omega)  \left.  p^2(\varepsilon') \frac{dp'}{d\varepsilon'} \right|_{\varepsilon(p) - j\hbar\omega } 
p^2 dp d\mu dx  ,
\nonumber
\end{align}
or more briefly,
$\bar{f}_{ikm}^{n+1} = \bar{f}_{ikm}^n + \Gamma_T + \Gamma_C,
$
where the transport and collision terms for the cell average time evolution are 
\begin{align}
\Gamma_T
& =  - \frac{\Delta t^n}{V_{ikm}} \left[
 \int_{\partial_x \Omega_{km}}  \,  \widehat{H^{(x)}f_h}|_{x_{i+}} 
 \, p^2 dp d\mu  
-
 \int_{\partial_x \Omega_{km}}  \,  \widehat{H^{(x)}f_h}|_{x_{i-}}  
 \, p^2 dp d\mu  
\right. 
\nonumber\\
& +  
p^2_{k^+}
\int_{\partial_p \Omega_{im}}  \,
\widehat{H^{(p)}f_h}|_{p_{k+}} 
\, d\mu dx
-  
p^2_{k^-}
\int_{\partial_p \Omega_{im}}  \,
\widehat{H^{(p)}f_h}|_{p_{k-}} 
\, d\mu dx
\nonumber\\
+ & (1-\mu_{m+}^2) 
\left.
\int_{\partial_{\mu}\Omega_{ik}} 
\widehat{H^{(\mu)}f_h}|_{\mu_{m+}}
p dp dx
-
(1-\mu_{m-}^2)
\int_{\partial_{\mu}\Omega_{ik}} 
\widehat{H^{(\mu)}f_h}|_{\mu_{m-}} 
p dp dx
\right],
\nonumber
\end{align}

\vspace{-.65cm}

\begin{align}
\Gamma_C
& = \frac{2\pi\Delta t^n}{V_{ikm}} 
\int_{\Omega_{ikm}} 
\sum_{j=-1}^{1} c_j \, \chi(\varepsilon(p) + j\hbar\omega)
\int_{-1}^{1} d\mu' \left. 
f_h'
p^2(\varepsilon') \frac{dp'}{d\varepsilon'}  \right|_{\varepsilon(p) + j\hbar\omega }  
\, p^2 \, dp d\mu dx  
\nonumber\\
-
&\frac{4\pi\Delta t^n}{V_{ikm}}
\int_{\Omega_{ikm}} 
f_h
\sum_{j=-1}^{1} c_j \, \chi(\varepsilon(p) - j\hbar\omega)  \left.  p^2(\varepsilon') \frac{dp'}{d\varepsilon'}  \right|_{\varepsilon(p) - j\hbar\omega } 
\, p^2 \, dp d\mu dx  .
\nonumber
\end{align}


\subsection{Positivity preservation for 1Dx-2Dp DG scheme}
We use the strategy of Zhang \& Shu in \cite{ZhangShu1,ZhangShu2} for conservation laws, applied as well in \cite{EECHXM-JCP} for conservative phase space collisionless transport in curvilinear
coordinates, and in \cite{CGP} for a Vlasov - Boltzmann problem with linear non-degenerate collisional forms, to preserve the positivity of the probability density function in our DG scheme treating the collision term as a source, this being possible since our collisional form is mass preserving. 
We will use a convex combination parameter $\eta \in [0,1]$
\vspace{-.1cm}
\begin{equation}
\bar{f}_{ikm}^{n+1} = 
\eta \cdot I +
(1-\eta)\cdot II,
\quad
I=
\bar{f}_{ikm}^n + {\Gamma_T}/{\eta},
\quad 
{II} = 
 \bar{f}_{ikm}^n + {\Gamma_C}/{(1-\eta)} ,
\end{equation}
and we will find conditions such that $I$ and $II$ are positive,
to guarantee the positivity of the cell average of our numerical probability density function for the next time step. 
The positivity of the numerical solution to the pdf in the whole domain can be guaranteed just by applying the limiters in \cite{ZhangShu1,ZhangShu2} that preserve the cell average but modify the slope of the piecewise linear solutions in order to make the function non - negative.
Regarding $I$, we derive its positivity conditions below.
\vspace{-.1cm}
\begin{align}
I &= 
 \bar{f}_{ikm}^n + \frac{\Gamma_T}{\eta}   
=
\frac{\int_{\Omega_{ikm}} f_h \,  p^2 \, dp d\mu dx }{ V_{ikm} }\!
- \! \frac{\Delta t^n/\eta}{  V_{ikm}} \left[
 \int_{\partial_x \Omega_{km}} \!\! 
 ( \widehat{H^{(x)}f_h}|^{x_{i+}}_{x_{i-}} ) p^2 dp d\mu  
\right. 
\nonumber\\
 +& 
p^2_{k^+}
\int_{\partial_p \Omega_{im}} \! 
(
\widehat{H^{(p)}f_h}|_{p_{k+}} 
\!-\!
\widehat{H^{(p)}f_h}|_{p_{k-}} 
)
 d\mu dx 
 \nonumber\\
 &
\!+\!
\left.
\int_{\partial_{\mu}\Omega_{ik}}
\left( 
(1\!-\!\mu_{m+}^2) 
\widehat{H^{(\mu)}f_h}|_{\mu_{m+}}
\!-\!
(1\!-\!\mu_{m-}^2)
\widehat{H^{(\mu)}f_h}|_{\mu_{m-}}
\right)
p  dp dx
\right].
\nonumber
\end{align}
We will split the cell average convexly, as in
$\left\{s_l \right\}_{l=1}^3$ s.t. 
$s_1 + s_2 + s_3 = 1$.
Then
\begin{align}
I &=
 \frac{ 1 }{ V_{ikm} }
\left[
(s_1 + s_2 + s_3) \int_{\Omega_{ikm}} f_h \,  p^2 \, dp d\mu dx
\right.
\nonumber\\
&
-  \frac{\Delta t^n}{ \eta } \left(
 \int_{\partial_x \Omega_{km}}  \,  \widehat{H^{(x)}f_h}|_{x_{i+}} 
 \, p^2 dp d\mu  
-
 \int_{\partial_x \Omega_{km}}  \,  \widehat{H^{(x)}f_h}|_{x_{i-}}  
 \, p^2 dp d\mu  
\right. 
\nonumber\\
&  +
p^2_{k^+}
\int_{\partial_p \Omega_{im}}  \,
\widehat{H^{(p)}f_h}|_{p_{k+}} 
\, d\mu dx
-  
p^2_{k^-}
\int_{\partial_p \Omega_{im}}  \,
\widehat{H^{(p)}f_h}|_{p_{k-}} 
\, d\mu dx
\nonumber\\
& +
(1-\mu_{m+}^2)
\int_{\partial_{\mu}\Omega_{ik}} 
\widehat{H^{(\mu)}f_h}|_{\mu_{m+}}
\, p \, dp dx
-
\left.
\left.
(1-\mu_{m-}^2)
\int_{\partial_{\mu}\Omega_{ik}} 
\widehat{H^{(\mu)}f_h}|_{\mu_{m-}} 
\, p \, dp dx
\right)
\right]
\nonumber\\
= &
 \frac{ 1 }{ V_{ikm} }
\left[
 \int_{\partial_x \Omega_{km}} 
\left\lbrace 
s_1 \int_{x_{i-}}^{x_{i+}} f_h p^2  dx
-  \frac{\Delta t^n}{ \eta } 
\sum_\pm \pm
\widehat{H^{(x)}f_h}|_{x_{i\pm}} 
 p^2   
\right\rbrace 
dp d\mu 
\right.
\nonumber\\
 +&   \int_{\partial_p \Omega_{im}} 
\left\lbrace
s_2 \int_{p_{k-}}^{p_{k+}}
f_h \,  p^2 \, dp 
-  \frac{\Delta t^n}{ \eta } \left(
p^2_{k^+}
\widehat{H^{(p)}f_h}|_{p_{k+}} 
-  
p^2_{k^-}
\widehat{H^{(p)}f_h}|_{p_{k-}} 
\right)
\right\rbrace
d\mu dx
\nonumber\\
&+
 \int_{\partial_{\mu}\Omega_{ik}} 
 \left.
 \left\lbrace 
s_3 \int_{\mu_{m-}}^{\mu_{m+}}
 f_h \,  p^2 \, d\mu 
-  \frac{\Delta t^n}{ \eta } \left[
\sum_\pm \pm
(1-\mu_{m\pm}^2)
\widehat{H^{(\mu)}f_h}|_{\mu_{m\pm}}
 p 
\right]
\right\rbrace dp dx
\right].
\nonumber
\end{align}
We can integrate 
all the functions
above by Gauss-Lobatto quadrature.
We use it for the integrals of $f_h \, p^2$ over intervals, so that the endpoints values can balance the flux terms of boundary integrals, obtaining then CFL conditions. So
\begin{align}
I & = 
 \frac{ 1 }{ V_{ikm} }
\left[
 \int_{\partial_x \Omega_{km}} 
\left\lbrace 
s_1 
\sum_{\hat{q}=1}^N \hat{w}_{\hat{q}}
f_h|_{x_{\hat{q}}} \,  p^2 \,  \Delta x_i
\right.
\right.
- 
\left. 
\left. 
\frac{\Delta t^n}{ \eta } \left(
\widehat{H^{(x)}f_h}|^{x_{i+}}_{x_{i-}}      
\right)p^2
\right\rbrace 
dp \, d\mu 
\right.
\nonumber\\
 + &  \int_{\partial_p \Omega_{im}} 
\left\lbrace
s_2 
\sum_{\hat{r}=1}^N \hat{w}_{\hat{r}}
f_h|_{p_{\hat{r}}} \,  p^2_{\hat{r}} \, \Delta p_k 
-  \frac{\Delta t^n}{ \eta } \left(
\sum_\pm \pm
p^2_{k^\pm}
\widehat{H^{(p)}f_h}|_{p_{k\pm}} 
\right)
\right\rbrace
d\mu \, dx +
\nonumber\\
 &
 \int_{\partial_{\mu}\Omega_{ik}} 
\left.
 \left\lbrace 
s_3 
\sum_{\hat{s}=1}^N \hat{w}_{\hat{s}}
 f_h|_{\mu_{\hat{s}}} p^2 \Delta \mu_m 
-  \frac{\Delta t^n}{ \eta } 
\sum_\pm \pm
(1-\mu_{m\pm}^2)
\widehat{H^{(\mu)}f_h}|_{\mu_{m\pm}}
p 
\right\rbrace dp dx
\right] 
\nonumber\\
&
= p^2 
dp d\mu 
\times 
\nonumber\\
 & 
\left[
 \int_{\partial_x \Omega_{km}} 
s_1 \Delta x_i
\left\lbrace 
  \sum_{\hat{q}=2}^{N-1} \hat{w}_{\hat{q}}
f_h|_{x_{\hat{q}}}  
+
\left(
\hat{w}_1 f_h|_{x_{i-}}^+  +
\hat{w}_N f_h|_{x_{i+}}^-  
\right)
-
\frac{\Delta t^n
\left(
\widehat{H^{(x)}f_h}|^{x_{i+}}_{x_{i-}}  
\right)
}{ \eta s_1 \Delta x_i} 
\right\rbrace 
\right.
\nonumber\\
&  +   \int_{\partial_p \Omega_{im}} 
s_2 \Delta p_k  
\left\lbrace
\left(
\hat{w}_{1} f_h|_{p_{k-}}^+ \,  p^2_{k-}  +
\hat{w}_{N}  f_h|_{p_{k+}}^- \,  p^2_{k+} 
\right)
+
\sum_{\hat{r}=2}^{N-1} \hat{w}_{\hat{r}}
f_h|_{p_{\hat{r}}} \,  p^2_{\hat{r}} \, 
\right.
\nonumber\\
&-
\left.
  \frac{\Delta t^n}{ \eta s_2 \Delta p_k } \left(
p^2_{k^+}
\widehat{H^{(p)}f_h}|_{p_{k+}} 
-  
p^2_{k^-}
\widehat{H^{(p)}f_h}|_{p_{k-}} 
\right)
\right\rbrace
d\mu \, dx
\nonumber\\
& +
 \int_{\partial_{\mu}\Omega_{ik}} 
s_3 \,  p^2 \, \Delta \mu_m 
\left\lbrace 
\left(
\hat{w}_{1} f_h|_{\mu_{m-}}^+
+
\hat{w}_{N} f_h|_{\mu_{m+}}^-
\right)
+
\sum_{\hat{s}=2}^{N-1} 
\hat{w}_{\hat{s}} f_h|_{\mu_{\hat{s}}}
\right.
\nonumber\\
&-
\left.
\left. 
  \frac{\Delta t^n}{ \eta s_3 p \, \Delta \mu_m } \left[
  \sum_\pm \pm
(1-\mu_{m\pm}^2)
\widehat{H^{(\mu)}f_h}|_{\mu_{m\pm}}
\right] 
\right\rbrace  dp dx
\right] \frac{ 1 }{ V_{ikm} } \, .
\nonumber
\end{align}
We reorganize the terms involving the endpoints, which are in parenthesis. So
\begin{align}
I & = 
 \frac{ 1 }{ V_{ikm} }
\left[
 \int_{\partial_x \Omega_{km}} 
s_1 \Delta x_i
\left\lbrace 
\left(
\hat{w}_1 f_h|_{x_{i-}}^+  
+  \frac{\Delta t^n}{ \eta s_1 \Delta x_i} 
\widehat{H^{(x)}f_h}|_{x_{i-}}  
\right)
\right.\right.\nonumber\\
&+\left.\left.
\left(
\hat{w}_N f_h|_{x_{i+}}^-  
-  \frac{\Delta t^n}{ \eta s_1 \Delta x_i} 
\widehat{H^{(x)}f_h}|_{x_{i+}} 
\right)
\right.
\right.
+
\left. \,
  \sum_{\hat{q}=2}^{N-1} \hat{w}_{\hat{q}}
f_h|_{x_{\hat{q}}}  
\right\rbrace 
p^2 
dp d\mu 
 + 
\nonumber\\
& 
 \int_{\partial_p \Omega_{im}} 
s_2 \Delta p_k  
\left\lbrace
\sum_{\hat{r}=2}^{N-1} \hat{w}_{\hat{r}}
f_h|_{p_{\hat{r}}} \,  p^2_{\hat{r}} 
\right.
+
\left.
p^2_{k^-}
\left(
\hat{w}_{1} f_h|_{p_{k-}}^+ \, 
+  \frac{\Delta t^n}{ \eta s_2 \Delta p_k } 
\widehat{H^{(p)}f_h}|_{p_{k-}} 
\right)
\right.\nonumber\\
&+
\left.
p^2_{k+} 
\left(
\hat{w}_{N}  f_h|_{p_{k+}}^- \,
-  \frac{\Delta t^n}{ \eta s_2 \Delta p_k } 
\widehat{H^{(p)}f_h}|_{p_{k+}} 
\right)
\right\rbrace
d\mu \, dx +
\nonumber\\
& 
 \int_{\partial_{\mu}\Omega_{ik}} 
dx \, dp \, p^2 \,
s_3 \,  \Delta \mu_m 
\left\lbrace 
\sum_{\hat{s}=2}^{N-1} 
\hat{w}_{\hat{s}} f_h|_{\mu_{\hat{s}}} 
\right.
+\left.
\left(
\hat{w}_{1} f_h|_{\mu_{m-}}^+
+  \frac{\Delta t^n (1-\mu_{m-}^2)
\widehat{H^{(\mu)}f_h}|_{\mu_{m-}} 
}{ \eta s_3 p \, \Delta \mu_m }
\right)
\right.
\nonumber\\
&+
\left.
\left. 
\left(
\hat{w}_{N} f_h|_{\mu_{m+}}^-
-  \frac{\Delta t^n (1-\mu_{m+}^2) }{ \eta s_3 p \, \Delta \mu_m } 
\widehat{H^{(\mu)}f_h}|_{\mu_{m+}}
\right)
\right\rbrace  
\right] .
\nonumber
\end{align}
To guarantee the positivity of $I$, assuming that the terms $f_h|_{x_{\hat{q}}}, \, f_h|_{p_{\hat{r}}}, \, f_h|_{\mu_{\hat{s}}} $
are positive at time $t^n$, we only need that the terms in parenthesis related to interval endpoints are positive. 
Since $\hat{w}_1 = \hat{w}_N$ for Gauss-Lobatto Quadrature, 
we want 
\begin{align}
0 & \leq 
\hat{w}_N f_h|_{x_{i\pm}}^{\mp}  
\mp  \frac{\Delta t^n}{ \eta s_1 \Delta x_i} 
\widehat{H^{(x)}f_h}|_{x_{i\pm}} 
,
\nonumber\\
0 & \leq  
\hat{w}_{N}  f_h|_{p_{k\pm}}^{\mp} \,
\mp \frac{\Delta t^n}{ \eta s_2 \Delta p_k } 
\widehat{H^{(p)}f_h}|_{p_{k\pm}} 
,
 & 0\leq 
\hat{w}_{N} f_h|_{\mu_{m\pm}}^{\mp}
\mp  \frac{\Delta t^n (1-\mu_{m\pm}^2) }{ \eta s_3 p \, \Delta \mu_m } 
\widehat{H^{(\mu)}f_h}|_{\mu_{m\pm}}
.
\nonumber
\end{align}
We have used the notation below for the numerical flux, given by the upwind rule
\begin{align}
 \widehat{H^{(x)}f}|_{x_{i\pm}}
&=
\partial_p \varepsilon
\widehat{ f_h  \mu }|_{x_{i\pm}} 
= \partial_p \varepsilon
\left[
\left(\frac{\mu + |\mu|}{2} \right) f_h |^-_{x_{i\pm}} +
\left(\frac{\mu - |\mu|}{2} \right) f_h |^+_{x_{i\pm}}
\right],
\nonumber\\
\widehat{H^{(p)}f}|_{p_{k\pm}}
&=
- q \widehat{ E f_h \mu } |_{p_{k\pm}} 
=
q
\left[
\left(\frac{ |E \mu | -E \mu  }{2}\right) f_h |^-_{p_{k\pm}} -
\left(\frac{E \mu  + |E \mu |}{2} \right) f_h |^+_{p_{k\pm}}
\right],
\nonumber\\
\widehat{H^{(\mu)}f}|_{\mu_{m\pm}}
&=
- q  \widehat{ E f_h }|_{\mu_{m\pm}} =
q
\left[
\left(\frac{-E + |E|}{2} \right) f_h |^-_{\mu_{m\pm}} +
\left(\frac{-E - |E|}{2} \right) f_h |^+_{\mu_{m\pm}}
\right] \, .
\nonumber
\end{align}
We have assumed the positivity of the pdf evaluated at Gauss-Lobatto points, which include endpoints, so we know that
$ f_h|_{x_{i\pm}}^{\mp} , \, f_h|_{p_{k\pm}}^{\mp} , \,  f_h|_{\mu_{m\pm}}^{\mp} $ are positive. The worst case scenario for positivity is having negative flux terms. In that case,
\begin{align}
0 & \leq 
\hat{w}_N f_h|_{x_{i\pm}}^{\mp}  
-  \frac{\Delta t^n}{ \eta s_1 \Delta x_i} 
\partial_p \varepsilon \, |\mu| f_h|_{x_{i\pm}}^{\mp}
=
f_h|_{x_{i\pm}}^{\mp}
\left(
\hat{w}_N   
-  \frac{\Delta t^n}{ \eta s_1 \Delta x_i} 
\partial_p \varepsilon \, |\mu| 
\right),
\nonumber\\
0 & \leq  
\hat{w}_{N}  f_h|_{p_{k\pm}}^{\mp} \,
- \frac{\Delta t^n}{ \eta s_2 \Delta p_k } 
q |E(x,t) \mu | f_h|_{p_{k\pm}}^{\mp}
=
f_h|_{p_{k\pm}}^{\mp} \,
\left( \hat{w}_{N}
- \frac{\Delta t^n}{ \eta s_2 \Delta p_k } 
q |E(x,t) \mu | 
\right),
\nonumber\\
0 & \leq 
\hat{w}_{N} f_h|_{\mu_{m\pm}}^{\mp}
-\frac{\Delta t^n (1-\mu_{m\pm}^2) }{ \eta s_3 p  \Delta \mu_m } 
q|E| f_h|_{\mu_{m\pm}}^{\mp}
=
 f_h|_{\mu_{m\pm}}^{\mp}
\left(
\hat{w}_{N}
-\frac{\Delta t^n (1-\mu_{m\pm}^2) }{ \eta s_3 p  \Delta \mu_m } 
q|E| 
\right)  .
\nonumber
\end{align}
We need then for the worst-case scenario that 
\begin{align}
\hat{w}_N   
& \geq 
\frac{\Delta t^n}{ \eta s_1 \Delta x_i} 
\partial_p \varepsilon \, |\mu| ,
\,
\hat{w}_{N}
& \geq 
 \frac{\Delta t^n}{ \eta s_2 \Delta p_k } 
q |E(x,t) \mu | ,
\,
\hat{w}_{N}
& \geq  
\frac{\Delta t^n (1-\mu_{m\pm}^2) }{ \eta s_3 p \, \Delta \mu_m } 
q|E(x,t)| \, ,  
\nonumber
\end{align}
or equivalently,
\begin{align}
\hat{w}_N   \frac{ \eta s_1 \Delta x_i}{ 
\partial_p \varepsilon \, |\mu|
}
& \geq 
 {\Delta t^n} ,
\quad 
\hat{w}_{N}
 \frac{ \eta s_2 \Delta p_k }{ 
q |E(x,t) \mu | 
}
& \geq  {\Delta t^n},
\quad
\hat{w}_{N}
\frac{ \eta s_3  \, \Delta \mu_m \, p }{ 
q|E(x,t)| 
(1-\mu_{m\pm}^2)
}
& \geq  \Delta t^n  \, .
\nonumber
\end{align}
Therefore the CFL conditions imposed to satisfy the positivity of $I$ are
\begin{align}
\frac{ \eta s_1 \hat{w}_N   \Delta x_i}{ 
\max_{\hat{r}}
\partial_p \varepsilon(p_{\hat{r}}) \cdot  \,
\max_{\pm} |\mu_{m\pm}|
}
& \geq  
 {\Delta t^n} ,
\,
\frac{ \eta s_2 \hat{w}_{N}
  \Delta p_k }{ 
q \max_{\hat{q}} |E(x_{\hat{q}},t)| 
\cdot
\max_{\pm} |\mu_{m\pm} | 
}
& \geq  {\Delta t^n},
\nonumber\\
\frac{ \eta s_3 \hat{w}_{N} \Delta \mu_m
\cdot p_{k-} \,  }{ 
q
\max_{\hat{q}}
|E(x_{\hat{q}},t)| 
\cdot
\max_{\pm}
(1-\mu_{m\pm}^2)
}
& \geq  \Delta t^n  \, .
\nonumber
\end{align}

Regarding $II$, there are several ways to guarantee its positivity.
One possible way to guarantee it is positive is given below,
by separating the gain and the loss part, combining the cell average
with the loss term and deriving a CFL condition related to the collision frequency, and imposing a positivity condition on 
the points where the gain term is evaluated, which differs
for inelastic scatterings from the previous Gauss-Lobatto points because of the addition or subtraction of the phonon energy. We would need an additional set of points in which to impose positivity to guarantee positivity of $II$ as a whole, since                                                                                                                                                                                                      
\begin{align}
&
II  = 
\bar{f}_{ikm}^n + \frac{\Gamma_C}{1-\eta}   =
\nonumber\\
&\bar{f}_{ikm}^n
 + \frac{\frac{\Delta t^n 2\pi}{V_{ikm}}}{1-\eta}
[
\int_{\Omega_{ikm}} 
\sum_{j=-1}^{1} c_j \, \chi(\varepsilon + j\hbar\omega)
\int_{-1}^{1} d\mu' \left. 
f_h' \, p^2(\varepsilon') \frac{dp}{d\varepsilon}  \right|_{\varepsilon + j\hbar\omega }  
\, p^2 \, dp d\mu dx  
\nonumber\\
&-
{2} 
\int_{\Omega_{ikm}} 
f_h
\underbrace{
\sum_{j=-1}^{+1} c_j \, \chi(\varepsilon(p) - j\hbar\omega)  \left.  p^2(\varepsilon') \frac{dp'}{d\varepsilon'} \right|_{\varepsilon' = \varepsilon(p) - j\hbar\omega } 
}_{ = \, \nu(p) > 0, \quad \mbox{since} \quad \partial_p\varepsilon > 0, \, c_j > 0, \, \chi \geq 0 }
\, p^2 \, dp d\mu dx  
]
=\frac{1}{V_{ikm}}\times
\left[
\right.
\nonumber\\
&
\left.
\frac{2\pi \Delta t^n}{1-\eta}
\right.
\left. 
\int_{\Omega_{ikm}} 
\sum_{j=-1}^{1} c_j \, \chi(\varepsilon + j\hbar\omega)
\int_{-1}^{1} d\mu' \left. 
f_h(x,p(\varepsilon'),\mu') \, p^2(\varepsilon') \frac{dp}{d\varepsilon} 
\right|_{\varepsilon + j\hbar\omega }  
\, p^2 \, dp d\mu dx  
\right. +
\nonumber\\
&
\int_{\Omega_{ikm}} f_h dV
-\frac{4\pi \Delta t^n}{1-\eta}
\left.
\int_{\Omega_{ikm}} 
f_h  
\left(
\sum_{j=-1}^{+1} c_j \, \chi(\varepsilon')  \left.  p^2(\varepsilon') \frac{dp'}{d\varepsilon'}  \right|_{\varepsilon'=\varepsilon(p) - j\hbar\omega } 
\right)
\, p^2 \, dp d\mu dx  
\right] 
\nonumber\\
& =
\frac{1}{V_{ikm}}
\left[
\frac{2\pi \Delta t^n}{ 1-\eta}
\sum_{j=-1}^{1} c_j
\int_{\Omega_{ikm}} 
\int_{-1}^{1} d\mu' \left. 
f_h' \, p^2(\varepsilon') \frac{dp'}{d\varepsilon'} \right|_{\varepsilon(p) + j\hbar\omega }  
\chi(\varepsilon(p) + j\hbar\omega)
\, p^2 \, dp d\mu dx  
\right. 
\nonumber\\
&
+
\left.
\int_{\Omega_{ikm}} 
f_h 
(
1 - \frac{4\pi \Delta t^n}{1-\eta}
\sum_{j=-1}^{1} c_j \, \chi(\varepsilon')  \left.  p^2(\varepsilon') \frac{dp}{d\varepsilon}  
\right|_{\varepsilon' = \varepsilon(p) - j\hbar\omega } 
)
\, p^2 \, dp d\mu dx  
\right] = \frac{1}{V_{ikm}} \times 
\left[
\right.
\nonumber\\
&
\left.
\frac{2\pi \Delta t^n}{1-\eta}
\sum_{j=-1}^{+1} c_j |\Omega_{ikm}|
\right.
\left. 
\underbrace{ 
\sum_{s,r,q}  w_{s,r,q} 
f_h(x_s,p'(\varepsilon(p_r) + j\hbar\omega),\mu'_q) \left.
p^2(\varepsilon') \frac{dp}{d\varepsilon} 
\chi(\varepsilon')
\right|_{\varepsilon(p_r) + j\hbar\omega}  
\, p_r^2  }_{ > 0 \, \mbox{if} \, f_h(x_s,p'(\varepsilon(p_r) + j\hbar\omega),\mu'_q) \, > 0 . \, \mbox{Additional positivity point set}  }
\right.
\nonumber\\
&
+
\int_{\Omega_{ikm}} 
f_h  
\underbrace{
(
1 - \frac{4\pi \Delta t^n}{1-\eta}
\sum_{j=-1}^{+1} c_j \, \chi(\varepsilon(p) - j\hbar\omega)  \left.  p^2(\varepsilon') \frac{dp'}{d\varepsilon'}  \right|_{\varepsilon(p) - j\hbar\omega } 
)
}_{> 0 \implies \frac{ (1-\eta)}{  \max_{GQp} \sum_{j=-1}^{1} c_j \, \chi(\varepsilon(p) - j\hbar\omega)  \left.  p^2(\varepsilon') \frac{dp}{d\varepsilon}  \right|_{ \varepsilon - j\hbar\omega } } > \Delta t  }
p^2 dp d\mu dx  
],
\nonumber
\end{align}
where the notation for the measure of the elements is
\begin{equation}
|\Omega_{ikm}| = \Delta x_i \Delta p_k \Delta \mu_m  \, .
\end{equation}
Another possible way to guarantee positivity for $II$ 
is by considering the collision term as a whole.
The difference between the gain and loss
integrals will give us a smaller source term, and therefore a more relaxed CFL condition for $\Delta t^n $. We have 
\begin{align}
&
II  =  \bar{f}_{ikm}^n + \frac{\Gamma_C}{1-\eta}   = \, 
\frac{\int_{\Omega_{ikm}} f_h dV}{V_{ikm}} + 
\frac{\Delta t^n \, \int_{\Omega_{ikm}}Q(f_h)dV}{ (1-\eta)V_{ikm}}
\,
= 
\,
\frac{\int_{\Omega_{ikm}} f_h dV}{V_{ikm}} +
\nonumber\\
&
\left[
\frac{1}{2}  
\int_{\Omega_{ikm}} 
\sum_{j=-1}^{+1} c_j \, \chi(\varepsilon(p) + j\hbar\omega)
\int_{-1}^{1} d\mu' \left. 
f_h' \, p^2(\varepsilon') \frac{dp'}{d\varepsilon'} 
\right|_{ \varepsilon(p) + j\hbar\omega }  
\, p^2 \, dp d\mu dx  
\right.
-
\nonumber\\
&
\int_{\Omega_{ikm}} 
f_h
\underbrace{
\sum_{j=-1}^{+1} c_j \, \chi(\varepsilon(p) - j\hbar\omega)  \left.  p^2(\varepsilon') \frac{dp'}{d\varepsilon'} \right|_{\varepsilon(p) - j\hbar\omega } 
}_{ = \, \nu(p) > 0, \quad \mbox{since} \quad \partial_p\varepsilon > 0, \quad c_j > 0, \quad \chi \geq 0 }
\, p^2 \, dp d\mu dx  
]
\frac{\frac{4\pi\Delta t^n}{V_{ikm}}}{(1-\eta)} = \frac{1}{V_{ikm}}\times
\nonumber
\end{align}
\begin{align}
&
\left[
\frac{2\pi \Delta t^n}{ (1-\eta)}
\left\lbrace
\int_{\Omega_{ikm}} 
\sum_{j=-1}^{+1} c_j \, \chi(\varepsilon(p) + j\hbar\omega)
\int_{-1}^{1} d\mu' \left. 
f_h' p^2(\varepsilon') \frac{dp'}{d\varepsilon'} 
\right|_{\varepsilon(p) + j\hbar\omega }  
p^2  dp d\mu dx  -
\right.
\right. 
\nonumber\\
&
\left.
2 
\left.
\int_{\Omega_{ikm}} 
f_h  
\sum_{j=-1}^{+1} c_j \, \chi(\varepsilon(p) - j\hbar\omega)  \left.  p^2(\varepsilon') \frac{dp'}{d\varepsilon'}  \right|_{\varepsilon(p) - j\hbar\omega } 
\, p^2 \, dp d\mu dx  
\right\rbrace
+ \int_{\Omega_{ikm}} f_h dV
\right]=
\nonumber
\end{align}
\begin{align}
&
[
\frac{ \Delta t^n}{1-\eta}
\int_{\Omega_{ikm}} 
\underbrace{
[
2\pi
\sum_{j=-1}^{+1} c_j \, 
\int_{-1}^{1} d\mu' \left.
f_h' \, p^2(\varepsilon') \frac{dp'}{d\varepsilon'} \chi(\varepsilon')  \right|_{\varepsilon(p) + j\hbar\omega }
- f_h \nu(p)   
]
}_{Q(f_h)}
\, p^2 \, dp d\mu dx  
\nonumber\\
&
\left. 
+ \int_{\Omega_{ikm}} f_h dV
\right]\frac{1}{V_{ikm}} \, ,
\quad 
\nu(p)  =
4 \pi
\sum_{j=-1}^{+1}  c_j \,   \left.  p^2(\varepsilon') \frac{dp'}{d\varepsilon'} \chi(\varepsilon') \right|_{\varepsilon(p) - j\hbar\omega } 
= \nu(\varepsilon(p)) \, .
\nonumber
\end{align}


We have then 
\begin{align}
II  &=  \bar{f}_{ikm}^n + \frac{\Gamma_C}{1-\eta}   = \, 
\frac{\int_{\Omega_{ikm}} f_h dV}{V_{ikm}} + 
\frac{\Delta t^n \, \int_{\Omega_{ikm}}Q(f_h)dV}{ (1-\eta)V_{ikm}}
\,
= 
\,
\nonumber\\
&=   
\frac{1}{V_{ikm}} 
\left[
\int_{\Omega_{ikm}} f_h \, p^2 \, dp d\mu dx
+ 
\frac{\Delta t^n }{ (1-\eta) }
\int_{\Omega_{ikm}}Q(f_h) \, p^2 \, dp d\mu dx
\right]
\, ,
\nonumber\\
Q(f_h) 
&= 
2\pi
\sum_{j=-1}^{+1} c_j \, 
\int_{-1}^{+1} d\mu' 
\left. 
f_h(x,p(\varepsilon'),\mu') \, p^2(\varepsilon') \frac{dp'}{d\varepsilon'} \chi(\varepsilon') 
\right|_{\varepsilon' = \varepsilon(p) + j\hbar\omega }
- f_h \nu(p)   
 \, ,
\nonumber\\
\nu(p) 
& = 
4 \pi
\sum_{j=-1}^{+1}  c_j \,   \left. \left[ p^2(\varepsilon') \frac{dp'}{d\varepsilon'} \chi(\varepsilon') \right] \right|_{\varepsilon' = \varepsilon(p) - j\hbar\omega } 
= \nu(\varepsilon(p)) \, .
\nonumber
\end{align}
We want $II$ to be positive.
If the collision operator part was negative,
we choose the time step $\Delta t^n$ such that $II$ is positive on total. We will get this way our CFL condition in order to guarantee the positivity of $II$. We want the following,
\begin{align}
II  &=
\frac{1}{V_{ikm}} 
\int_{\Omega_{ikm}} 
\left[
f_h(x,p,\mu ,t) \,
+ 
\frac{\Delta t^n }{ (1-\eta) }
Q(f_h) (x,p,\mu ,t) 
\right]
\, p^2 \, dp d\mu dx
\, \geq 0 ,
\nonumber\\
II   &=  
\frac{|\Omega_{ikm}| }{V_{ikm}} 
\sum_{q,r,s} 
w_q w_r w_s
\left[
f_h(x_q,p_r,\mu_s ,t) \,
+ 
\frac{\Delta t^n }{ (1-\eta) }
Q(f_h) (x_q,p_r,\mu_s ,t) 
\right]
\, p^2_r 
\, \geq 0 .
\nonumber
\end{align}
If $0 >  Q(f_h) $ for any of the points
$(x_q, p_r, \mu_s)$ at time $t = t^n$, 
then choose $\Delta t^n$ s. t.
\begin{align}
0 
& \leq 
f_h(x_q,p_r,\mu_s ,t) \,
+ 
\frac{\Delta t^n }{ (1-\eta) }
Q(f_h) (x_q,p_r,\mu_s ,t) ,
\nonumber\\
0 
& \leq  
f_h(x_q,p_r,\mu_s ,t) \,
- 
\frac{\Delta t^n }{ (1-\eta) }
|Q(f_h)| (x_q,p_r,\mu_s ,t) ,
& 
\Delta t^n
\leq 
\frac{ (1 - \eta) f_h(x_q,p_r,\mu_s ,t)  }{ |Q(f_h)| (x_q,p_r,\mu_s ,t)  } \, .
\nonumber
\end{align}
Our CFL condition in this case would be then
\begin{align}
\Delta t^n
&\leq 
(1 - \eta)
\min_{Q(f_h)(x_q,p_r,\mu_s,t^n) < 0}
\left\lbrace
\frac{  f_h(x_q,p_r,\mu_s ,t^n)  }{ |Q(f_h)| (x_q,p_r,\mu_s ,t^n)  }
\right\rbrace \, .
\end{align}
The minimum for the CFL condition is taken 
over the subset of Gaussian Quadrature points
$ (x_q,p_r,\mu_s) $
inside the cell $\Omega_{ikm}$ 
over which 
$$ Q(f_h)(x_q,p_r,\mu_s, t^n) < 0 .$$ 
This subset of points might be different for each
time $t^n$. \\ \, \\
We have figured out the respective CFL conditions 
for the transport and collision parts. 
Finally, we only need to choose
the optimal parameter $\eta$ that gives us the most 
relaxed CFL condition for $ \Delta t^n$ such that positivity is preserved for the cell average at the next time, $\bar{f}_{ikm}^{n+1} $. 
The positivity of the whole numerical solution 
to the pdf, not just its cell average, can be guaranteed by applying the limiters in \cite{ZhangShu1,ZhangShu2}, which preserve the cell average but modify the slope of the piecewise linear solutions in order to make the function non - negative in case it was negative before.


\section*{Acknowledgments}
This  research was partially supported by NSF grants\\  NSF-RNMS DMS-1107465 (KI-Net),  DMS 1715515,  DMS 2009736, and DOE DE-SC0016283
project Simulation Center for Runaway Electron Avoidance and Mitigation. 
Support from the Oden Institute of Computational Engineering and Sciences at the University of Texas at Austin is gratefully acknowledged.
The authors gratefully acknowledge as well the support of the Hausdorff Research Institute for Mathematics (Bonn), through the Junior Trimester Program on Kinetic Theory. The first author gratefully acknowledges too start-up support from the University of Texas at San Antonio Departments of Mathematics and Physics \& Astronomy. 
\bibliographystyle{siamplain}

\end{document}